\documentclass{article}

\usepackage{arxiv}

\RequirePackage{amsthm,amsmath,amsfonts,amssymb}
\usepackage[numbers]{natbib}

\RequirePackage[colorlinks,citecolor=blue,urlcolor=blue]{hyperref}
\RequirePackage{graphicx}
\usepackage{bm, bbm}
\usepackage{listings,algorithmic}
\usepackage{float}
\usepackage{anyfontsize}
\usepackage{url} 

\newtheorem{theorem}{Theorem} 
\newtheorem{lemma}{Lemma}

\newcommand{\E}{\mathbb{E}} 
 
\newcommand{\Prob}{\mathbb{P}}

\makeatletter
\newcommand*{\textlabel}[2]{%
  \edef\@currentlabel{#1}
  \phantomsection
  #1\label{#2}
}
\makeatother

\title{Non-parametric multiple change-point detection}

\author{Andreas Anastasiou \\
	Department of Mathematics and Statistics \\ 
    University of Cyprus\\
	\texttt{anastasiou.andreas@ucy.ac.cy} \\
	\And
	Piotr Fryzlewicz \\
	Department of Statistics\\
    London School of Economics\\
	\texttt{p.fryzlewicz@lse.ac.uk} \\
}

\date{}

\renewcommand{\shorttitle}{Non-parametric segmentation}

\hypersetup{
pdftitle={Non-parametric multiple change-point detection},
pdfauthor={Andreas Anastasiou, Piotr Fryzlewicz},
pdfkeywords={Non-parametric statistics; segmentation; threshold criterion; information criterion},
}
\begin{document}
\maketitle

\begin{abstract}
	We introduce a methodology, labelled Non-Parametric Isolate-Detect (NPID), for the consistent estimation of the number and locations of multiple change-points in a non-parametric setting. The method can handle general distributional changes and is based on an isolation technique preventing the consideration of intervals that contain more than one change-point, which enhances the estimation accuracy. As stopping rules, we propose both thresholding and the optimization of an information criterion. In the scenarios tested, which cover a broad range of change types, NPID outperforms the state of the art. An \textsf{R} implementation is provided.
\end{abstract}

\keywords{Non-parametric statistics; segmentation; threshold criterion; information criterion}

\section{Introduction}
\label{sec:intro}

The focus of this work is on non-parametric, offline, multiple change-point detection, the aim of which is to test whether a data sequence is distributionally homogeneous, and if not, to estimate the number and locations of changes. The problem has seen a recent interest in a range of application areas; a non-exhaustive list includes social networks \citep{soc_net}, electrocardiography \citep{ECG} and hydrology \citep{Hyd}.

Denote by $N$ the number of change-points and by $r_1, r_2, \ldots, r_N$ their locations, with $r_0 = 0$, $r_{N+1} = T$. We work in the model
\begin{equation}
\label{our_model}
X_t \sim F_k, \quad r_{k-1} + 1 \leq t \leq r_{k}, \quad k = 1, \ldots, N + 1, \quad t=1,\ldots,T,
\end{equation}
where $\left\lbrace X_t \right\rbrace_{t=1,\ldots,T}$ is the observed univariate data sequence of serially independent observations and $F_k$ is the distribution of the $X_t$'s between the change-points $r_{k-1}$ and $r_k$, with $F_k \neq F_{k+1}\,\,\, \forall\, k \in \left\lbrace 1,\ldots,N\right\rbrace$. The parameters $N$ and $r_k$, as well as the distributions $F_k$, are unknown. Our methodology can in principle be extended to variables with values in an arbitrary metric space, as long as the distributional difference is detectable on a Vapnik-Cervonenkis (VC) class. We address this briefly in Section \ref{sec:conclusion}. As the choice of such a VC class in general metric spaces (a necessity for the computation of the change-point location estimator) is a difficult practical problem in itself, our literature review, which follows, does not make a distinction between non-parametric methods specifically designed for univariate data and those applicable to more general spaces.

Much of the early literature on non-parametric change-point detection covers the case of a single change-point. Some authors \citep{CH1988, Pettitt} test whether there is a change-point at an unknown time point, while others \citep{Darkhovskh1976, Carlstein1988, Dumbgen1991} assume that there exists a change-point and focus on its estimation and on the construction of confidence regions. In \cite{Carlstein1988} and \cite{Dumbgen1991}, mean-dominant norms, as defined in \cite{Carlstein1988}, are employed as a measure of the difference of the empirical cumulative distribution functions before and after a change-point candidate.

In multiple non-parametric change-point detection, observing a connection between multiple change-points and goodness-of-fit tests, \cite{Zouetal2014} propose a non-parametric maximum likelihood approach using empirical process techniques. The number of change-points, $N$, is estimated via an information criterion; given an estimate of $N$, a non-parametric profile likelihood-based algorithm is used to recursively compute the maximizer of an objective function. In an attempt to improve on the computational cost of that method, \cite{Haynes_Fearnhead2014} use the Pruned Exact Linear Time (PELT) method introduced in \cite{Killick_PELT} in order to find the optimal segmentation. \cite{Matteson_James_2014} propose a non-parametric approach based on Euclidean distances between sample observations, which combines binary segmentation with a divergence measure from \cite{Szekely_Rizzo2005}; however, this approach departs from classical non-parametric change-point detection in the sense that it is not invariant with respect to monotone transformations of data (or in other words: does not use the ranks of the data only). Using the Kolmogorov-Smirnov (KS) statistic, \cite{Padilla_et_al_2019} present two consistent procedures for univariate change point localization; one based on the standard binary segmentation algorithm of \cite{Vostrikova} and the other one on the WBS methodology of \cite{Fryzlewicz_WBS}. In \cite{Vanegas2021}, a multiscale method is developed for detecting changes in pre-defined quantiles of the underlying distribution. Other techniques are based on the estimation of density functions \citep{Kawahara_Sugiyama2011}, or the extension of well-defined statistics, such as the Wilcoxon/Mann–Whitney rank-based criterion first employed in \cite{Darkhovskh1976} for the detection of at most one change-point, to the multiple change-point setting \citep{Lungetal_2011}. 

Our proposed approach, labelled Non-Parametric Isolate-Detect (NPID), is a generic technique for consistent non-parametric multiple change-point detection in a data sequence. It adapts the isolate-detect principle, introduced in a parametric change-point detection context by 
\cite{Anastasiou_Fryzlewicz} and studied also in \cite{ccid} and \cite{Anastasiou_Papanastasiou}, to the non-parametric setting.
NPID consists of two main stages; firstly, the isolation of each of the true change-points within subintervals of the domain $[1,\ldots,T]$, and secondly their detection. 
The isolation step ensures that, with high probability, no other change-point is present in the interval under consideration, which enhances the detection power, especially in difficult scenarios such as ones involving limited spacings between consecutive change-points or other low-signal-to-noise-ratio settings. NPID's ability to analyse such structures accurately, with low computational cost (see Sections \ref{subsec:Compcost} and \ref{sec:simulation}), adds to the appeal of the proposed approach.
Within each interval, where, by construction, there is at most one change-point, we use the non-parametric approach of \cite{Carlstein1988} to carry out detection. For the theoretical results regarding the accuracy of NPID in estimating the locations of the change-points, we show in Section \ref{subsec:theory} that the optimal consistency rate is attained by employing optimal rate results from \cite{Dumbgen1991} in our proof strategy.
Section \ref{subsec:theory} situates our method within the existing literature.

We now briefly introduce the main steps of NPID. For an observed data sequence $\left\lbrace X_t\right\rbrace_{t=1,\ldots,T}$,
and with $\lambda_T$ a suitably chosen positive integer, the method first creates a collection of right- and left-expanding intervals
$S_{RL} = \left\lbrace R_1, L_1, R_2, L_2, \ldots,R_K,L_K\right\rbrace$, where
$K = \left\lceil T/\lambda_T \right\rceil$,
$R_j = [1,\min\left\lbrace j{\lambda_T}+1, T\right\rbrace]$ and
$L_{j} = [\max\left\lbrace T - j\lambda_T,1 \right\rbrace,T]$.
A suitably chosen contrast function, 
whose value at location $b \in \left\lbrace s,s+1,\ldots,e-1\right\rbrace$ and for an input argument $u \in \mathbb{R}$ is denoted by
$\tilde{B}_{s,e}^{b}(u)$
(formula \eqref{contrast_non_par}),
measures the difference between the pre-$b$ and post-$b$ empirical distributions at $u \in \mathbb{R}$.
NPID first works within $R_1$ and calculates $\tilde{B}_{1,\lambda_T+1}^{b}(X_i)$, for each $i \in \left\lbrace 1,2,\ldots, T\right\rbrace$ and for each $b \in \left\lbrace 1,2,\ldots,\lambda_T\right\rbrace$. This creates $\lambda_T$ vectors $$\boldsymbol{y_b} = \left(\tilde{B}_{1,\lambda_T+1}^{b}(X_1), \tilde{B}_{1,\lambda_T+1}^{b}(X_2), \ldots, \tilde{B}_{1,\lambda_T+1}^{b}(X_T)\right),\quad b=1,\ldots,\lambda_T.$$
The next step is to aggregate the contrast function information across $i \in \left\lbrace 1,2,\ldots, T\right\rbrace$ by applying to each $\boldsymbol{y_b}$ a mean-dominant norm $L: \mathbb{R}^T \rightarrow \mathbb{R}$; the mean-dominance property is such for any $d \in \mathbb{Z}^{+}$ and $\forall \boldsymbol{x} \in \mathbb{R}^{d}$ with $x_i \geq 0, i = 1,\ldots, d$, it holds that $L(\boldsymbol{x}) \geq \frac{1}{d}\sum_{i=1}^{d}x_i$. The formal mathematical definition of a mean-dominant norm is given in Section 2 of \cite{Carlstein1988} and examples include
\begin{equation}
\label{mean_dominant}
L_1(\boldsymbol{y_b}) = \frac{1}{T}\sum_{i=1}^T |y_{b,i}|, \quad L_2(\boldsymbol{y_b}) = \frac{1}{\sqrt{T}}\sqrt{\sum_{i=1}^T y_{b,i}^2}, \quad L_\infty(\boldsymbol{y_b}) = \sup_{i=1,\ldots,T}\left\lbrace |y_{b,i}|\right\rbrace.
\end{equation}
Applying $L(\cdot)$ to each $\boldsymbol{y_b}, b = 1,\ldots, \lambda_T$,  returns a vector $\boldsymbol{v}$ of length $\lambda_T$. With $\tilde{b}_{R_1}$ := ${\rm argmax}_j\left\lbrace v_j \right\rbrace$, if $v_{\tilde{b}_{R_1}}$ exceeds a certain threshold $\zeta_T$, then $\tilde{b}_{R_1}$ is taken as a change-point. If not, then the process tests the next interval in $S_{RL}$. After detection, the algorithm makes a new start from the end-point (or start-point) of the right- (or left-) expanding interval on which the detection occurred. Given a suitable choice of the threshold $\zeta_T$ (more details are provided in Section \ref{subsec:parameter_choice}), NPID ensures that we work on intervals with at most one change-point, with a high global probability.


The paper is organized as follows. Section \ref{sec:methodology_theory} defines NPID and gives the associated consistency theory for the number and locations of the estimated change-points. 
In Section \ref{sec:variants}, we discuss the computational aspects of NPID and the selection of the tuning parameter. In Section \ref{sec:simulation}, we provide a comparative simulation study. Section \ref{sec:real_data} contains real-life data examples, and Section \ref{sec:conclusion} concludes.
The \textsf{R} code implementing the method is available at \url{https://github.com/Anastasiou-Andreas/NPID}.

\section{Methodology and Theory}
\label{sec:methodology_theory}
\subsection{Methodology}
\label{subsec:methodology}
The general non-parametric framework is given in \eqref{our_model}. We assume that the $X_t$'s are mutually independent.
$N$ can possibly grow with the sample size, $T$. 
We first explain the coupling of the Isolate-Detect (ID) scheme, as introduced in \cite{Anastasiou_Fryzlewicz}, with the use of the empirical cumulative distribution function (ECDF) for the detection of distributional changes as in \cite{Carlstein1988}. We start with ID, where the isolation of each change-point, prior to detection, is carried out as in \cite{Anastasiou_Fryzlewicz}. For clarity of exposition, we show graphically the different isolation phases through a simple example of two change-points, $r_1=22$ and $r_2=44$, in a data sequence $X_t$ of length $T=60$, and with the three different distributions denoted by $F_j, j=1, 2, 3$. We have
\begin{equation}
\label{example}
 X_t \sim \begin{cases}
    F_1, & t=1,\ldots,22\\
    F_2, & t=23,24,\ldots,44\\
    F_3, & t=45,46,\ldots,60.
  \end{cases}
\end{equation}
We have Phases 1 and 2 involving four and one intervals, respectively. These are clearly indicated in Figure \ref{isolation_example} and they are only related to this specific example, because for cases with a different number of change-points we would have a different number of such phases. At the beginning $s = 1, e = 60$, and we take the expansion parameter, $\lambda_{T}$, to be equal to $10$. The first change-point to be isolated is $r_2$; this happens at the fourth step of Phase 1 and in the interval $[41,60]$. Given a suitable contrast function (details are given later in this section), $r_2$ gets detected. Following the detection, $e$ is updated as the start-point of the interval on which the detection occurred; therefore, $e=41$. In Phase 2 indicated in the figure, NPID is then applied in $[s,e]=[1,41]$. Intervals 1 and 3 of Phase 1 will not be re-examined in Phase 2 and $r_1$ is isolated, and then detected, in the interval [1,30]. After the detection, $s$ is updated as the end-point of the interval where the detection occurred; therefore, $s=30$. Our method is then applied in $[s,e] = [30,41]$; there are no more change-points to be isolated and given a suitable choice of the threshold, the process will terminate.
\begin{figure}
\centering
\includegraphics[width=12cm,height=5cm]{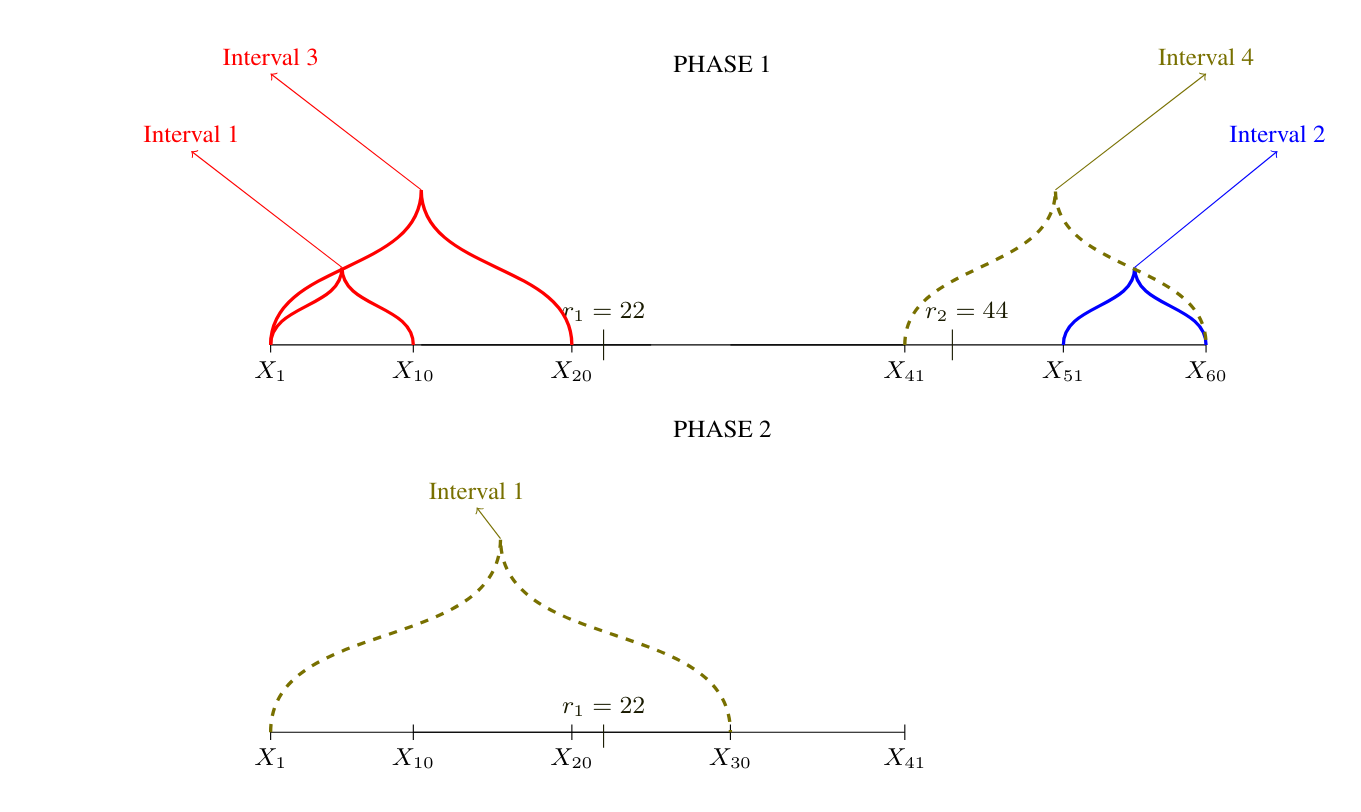}
\caption{The isolation and detection process for the example in \eqref{example}. The right and left expanding intervals are coloured in red and blue, respectively. The green dashed line is the first isolation interval at each phase.}
\label{isolation_example}
\end{figure}
The guaranteed (with high probability) isolation aspect in NPID ensures that the change-points will be detected one-by-one, while not being affected by changes in the data occurring at locations outside the current interval. This allows us to split the multiple change-point detection problem into a sequence of single change-point detection problems. As a result, optimal rate consistency results are obtained through the machinery appropriate for single change-point estimation; this is where the strategy of \cite{Dumbgen1991} comes in useful. To the best of our knowledge, \cite{Carlstein1988} is the first work on non-parametric single change-point estimation under no moment, support, or functional form assumptions for the distribution of the data before and after the change-point. The theoretical properties of the estimator proposed in \cite{Carlstein1988} include an exponential bound on the error probability and strong consistency, though with a sub-optimal rate. \cite{Dumbgen1991} embeds the non-parametric estimator of \cite{Carlstein1988} in a more general framework, in which the optimal consistency rate is attained.
\vspace{0.1in}
{\raggedright{{\textbf{The contrast function}}}}\\
The contrast we use is a non-parametric CUSUM function.
For a sample $X_1, X_2, \ldots, X_T$, with $B_t(u) := \mathbbm{1}_{\left\lbrace X_t \leq u\right\rbrace}$, define its ECDF as 
\begin{equation}
\label{empirical}
\hat{F}_T(u) = \frac{1}{T}\sum_{t=1}^{T}B_t(u), \quad u \in \mathbb{R}.
\end{equation} 
The non-parametric CUSUM is
\begin{equation}
\label{contrast_non_par}
\tilde{B}_{s,e}^{b}(u) = \sqrt{\frac{e-b}{(b-s+1)(e-s+1)}}\sum_{t=s}^{b}B_{t}(u) - \sqrt{\frac{b-s+1}{(e-b)(e-s+1)}}\sum_{t=b + 1}^{e}B_{t}(u),
\end{equation}
which is a weighted difference of the pre-$b$ and post-$b$ averages of $B_{t}(u)$ (or in other words a weighted difference of ECDFs), with $b$ being the time point under consideration.
\eqref{contrast_non_par} is used in the single change-point setting in \cite{Carlstein1988} and \cite{Dumbgen1991}, and in the multiple change-point setting in \cite{Padilla_et_al_2019}, where taking the supremum over all $u \in \mathbb{R}$ of the absolute value of \eqref{contrast_non_par} yields the ``CUSUM Kolmogorov-Smirnov statistic''. A detailed comparison of NPID with the method developed in \cite{Padilla_et_al_2019} is provided in Sections \ref{subsec:theory} and \ref{subsec:Compcost}. In \cite{Zouetal2014}, the sample of $B_t(u), t=1,\ldots,T$ is regarded as independent Bernoulli data with probability of success $\hat{F}(u)$. Working under this framework, the joint profile log-likelihood for a candidate set of change-points is obtained, and then maximized in an integrated form (over $\mathbb{R}$) in order to estimate the change-points.


We now describe our method generically through pseudocode. The aggregation of the contrast function values at each time point examined is an essential step in our algorithm. For the aggregation of the contrast values a mean-dominant norm, $L(\cdot)$, is used at each location $b \in [s,e)$, where $s$ and $e$ are given. The data sequence $\boldsymbol{X} = (X_1, X_2, \ldots, X_T)$ is also given.
\\
\vspace{-0.1in}
\\
{\textbf{function}} AGGREGATION($\boldsymbol{X}, s, e, L$)
\begin{algorithmic}
\FOR {$b = s, \ldots, e - 1$} \STATE { \FOR {$i = 1, 2, \ldots, T$}\STATE{$\tilde{B}_{s,e}^{b}(X_i) = \sqrt{\frac{e-b}{(b-s+1)(e-s+1)}}\sum_{t=s}^{b}B_{t}(X_i) - \sqrt{\frac{b-s+1}{(e-b)(e-s+1)}}\sum_{t=b + 1}^{e}B_{t}(X_i)$}\ENDFOR\\
$\boldsymbol{y_b} = (\tilde{B}_{s,e}^{b}(X_1), \tilde{B}_{s,e}^{b}(X_2), \ldots, \tilde{B}_{s,e}^{b}(X_T))^{\intercal}$\\
$v_b = L(\boldsymbol{y_b})$}\ENDFOR \\
Return($\boldsymbol{v}$)
\end{algorithmic}
{\textbf{end function}}
\\
\vspace{0.1in}
\\
We now specify the main function for the detection of change-points one by one. With
\begin{equation}
\label{def:deltaT}
    \delta_T = \min_{j=1,\ldots,N+1} \left|r_{j} - r_{j-1}\right|,
\end{equation}
we take the expansion parameter $\lambda_T$, which is an integer such that $0 < \lambda_T < \delta_T$ and set
\begin{equation}
\label{eq:K_intervals}
K = \lceil T/\lambda_T\rceil.
\end{equation}
We let $c_{j}^r = j\lambda_T + 1$ and $c_{j}^l = T - j\lambda_T$ for $j=1,\ldots, K-1$, with $c_{K}^r = T$ and $c_{K}^l = 1$. For a generic interval $[s,e]$, define the sequences
\begin{equation}
\label{expanding_points_s_e}
{\rm R}_{s,e} = \left[c_{k_1}^r, c_{k_1+1}^r, \ldots,e \right], \quad {\rm L}_{s,e} = \left[c_{k_2}^l, c_{k_2+1}^l, \ldots, s\right],
\end{equation}
where $k_1 := {\rm argmin}_{j \in \left\lbrace 1,\ldots,K \right\rbrace}\left\lbrace j\lambda_T + 1 > s \right\rbrace$ and $k_2 := {\rm argmin}_{j\in \left\lbrace 1,\ldots,K \right\rbrace}\left\lbrace T-j\lambda_T < e \right\rbrace$. Let $L(\cdot)$ be the mean-dominant norm used, and denote by $|A|$ the cardinality of $A$, and by $A(j)$ its $j^{th}$ element (if ordered). The pseudocode of the main change-point detection function is below.
\\
\vspace{-0.1in}
\\
{\textbf{function}} {\textbf{N}}{\small{ON}}{\textbf{P}}{\small{ARAMETRIC}}{\textbf{I}}{\small{SOLATE}}{\textbf{D}}{\small{ETECT}}($\boldsymbol{X}, s, e, \lambda_T, \zeta_T, L$)
\begin{algorithmic}
\IF{$e - s < 1$}\STATE{STOP} \ELSE \STATE{For $j \in \left\lbrace 1,\ldots,\left|{\rm R}_{s,e}\right|\right\rbrace,$ denote $\left[s_{2j-1}, e_{2j-1}\right] := \left[s, {\rm R}_{s,e}(j)\right]$\\
For $j \in \left\lbrace 1,\ldots,\left|{\rm L}_{s,e}\right|\right\rbrace,$ denote $\left[s_{2j}, e_{2j}\right] := \left[{\rm L}_{s,e}(j), e\right]$\\
$i = 1$\\
\textbf{(Main part)}\\
$s^* := s_{2i-1}, \; e^* := e_{2i-1}$\\
$\boldsymbol{v}$ = AGGREGATION($s^*, e^*, L$)
\STATE {$b^*_{2i-1} := \underset{s^*\leq b < e^*}{\rm argmax}\left\lbrace v_b\right\rbrace$
\IF{$v_{b^*_{2i-1}} > \zeta_T$}\STATE{add $b^*_{2i-1}$ to the set of estimated change-points.\\
{\textbf{N}}{\small{ON}}{\textbf{P}}{\small{ARAMETRIC}}{\textbf{I}}{\small{SOLATE}}{\textbf{D}}{\small{ETECT}}($e^*, e, \lambda_T, \zeta_T, L$)}\ELSE \STATE{$s^* := s_{2i}, \; e^* := e_{2i}$\\ $\boldsymbol{v}$ = AGGREGATION($s^*, e^*, L$)
\STATE {$b^*_{2i} := \underset{s^*\leq b < e^*}{\rm argmax}\left\lbrace v_b\right\rbrace$
\IF{$v_{b^*_{2i}} > \zeta_T$} \STATE{add $b_{2i}^*$ to the set of estimated change-points.\\
{\textbf{N}}{\small{ON}}{\textbf{P}}{\small{ARAMETRIC}}{\textbf{I}}{\small{SOLATE}}{\textbf{D}}{\small{ETECT}}($s, s^*, \lambda_T, \zeta_T, L$)}\\
 \ELSE \STATE{$i = i+1$\\
\IF {$i \leq \max\left\lbrace|{\rm L}_{s,e}|,|{\rm R}_{s,e}|\right\rbrace$}\STATE{Go back to {\textbf{(Main part)}} and repeat}\ELSE \STATE{STOP}\ENDIF}
\ENDIF}}\ENDIF}}\ENDIF\\
\end{algorithmic}
{\textbf{end function}}
\\
\vspace{0.1in}
\\
The call to launch NPID is {\textbf{N}}{\small{ON}}{\textbf{P}}{\small{ARAMETRIC}}{\textbf{I}}{\small{SOLATE}}{\textbf{D}}{\small{ETECT}}($\boldsymbol{X}, 1, T, \lambda_T, \zeta_T, L$).
An explanation of the pseudocode follows. With $K$ as in \eqref{eq:K_intervals}, we use the intervals $[s_1,e_1], [s_2,e_2], \ldots,$ $[s_{2K},e_{2K}]$ in the isolation step. The algorithm is looking for change-points interchangeably in right- and left-expanding intervals. The change-points, if any, are detected one by one in such intervals. Note that, between subsequent detections, in the odd-indexed intervals $[s_1, e_1], [s_3, e_3], \ldots, [s_{2K - 1}, e_{2K - 1}]$, the start-point is fixed and equal to $s$; in the even intervals $[s_2, e_2], [s_4, e_4], \ldots, [s_{2K}, e_{2K}]$, the end-point is fixed and equal to $e$. After the detection of a change-point, $s$ or $e$ is updated based on whether the change-point was detected in a right- or a left-expanding interval, respectively. The process follows until there are no intervals to be checked.
Change-point detection occurs if and when the aggregated contrast function values (obtained through the AGGREGATION($\cdot, \cdot, \cdot, \cdot$) function) at specific time points surpass a suitably chosen threshold $\zeta_T$. The algorithm stops when it is applied on an interval $[s,e]$, such that for all expanding intervals $[s^*_j,e^*_j] \subseteq [s,e]$, we have that for $\boldsymbol{v^{(j)}} :=$ AGGREGATION$(\boldsymbol{X}, s^*_j,e^*_j, L)$, there is no $b^*_j \in [s^*_j, e^*_j)$ with $v^{(j)}_{b^*_j} > \zeta_T$.
\subsection{Theoretical behaviour of NPID}
\label{subsec:theory}
For $x \in \mathbb{R}$, denote
\begin{align}
\label{general_notation}
\nonumber & \delta_j = r_{j} - r_{j-1}, \quad j = 1, \ldots, N+1,\\
\nonumber & F_{t}(x) = \Prob(X_t \leq x), \quad t = 1,\ldots, T,\\
& \Delta_{j}(x) = F_{r_{j+1}}(x) - F_{r_j}(x), \quad j=1,\ldots, N,
\end{align}
where $L(\cdot)$ is the mean-dominant norm used in the aggregation step. For the main result of Theorem \ref{consistency_theorem}, we make the following assumption.
\begin{itemize}
\item[(A1)]  For $\delta_T$ as in (\ref{def:deltaT}), it holds that $\delta_T \rightarrow \infty$ as $T \rightarrow \infty$. Furthermore, $\forall j \in \{1,\ldots,N\}$ there are constants
$\tilde{C}_j > 0$ independent of $T$, and sequences $\{\gamma_{j,T}\} \in \mathbb{R}^+$ such that
\begin{equation*}
\Prob\left(L(|\Delta_j(X_1)|, \ldots, |\Delta_j(X_T)|) \geq \frac{\tilde{C}_j}{\gamma_{j,T}}\right) \xrightarrow[T\rightarrow \infty]{} 1
\end{equation*}
and for $\delta_j$ as in \eqref{general_notation}, we require that $\underline{m}_T := \underset{j \in \left\lbrace 1, \ldots, N\right\rbrace}{\min}\left\lbrace \tilde{C}_j\sqrt{\min\left\lbrace \delta_j, \delta_{j+1}\right\rbrace}/\gamma_{j,T} \right\rbrace \geq \underline{C}\sqrt{\log T}$, for a large enough constant $\underline{C}$.
\end{itemize}
The number of change-points, $N$, is assumed neither known nor fixed, and can grow with $T$. Due to the minimum distance,
$\delta_T$, between two change-points, the only indirect assumption on the true number of change-points is that $N + 1 \leq T /\delta_T$. Below, we state the consistency result for the number and location of the estimated change-points, when any mean-dominant norm $L(\cdot)$ is employed. The proof is given in the supplementary material. However, in Appendix \ref{sec:App_B} we provide a brief discussion of the main steps followed in the proof.
\begin{theorem}
\label{consistency_theorem}
Let $\left\lbrace X_t \right\rbrace_{t=1,\ldots,T}$ follow model \eqref{our_model} and assume that (A1) holds. Let $N$ and $r_j, j=1,\ldots,N$ be the number and locations of the change-points, while $\hat{N}$ and $\hat{r}_j, j=1,\ldots,\hat{N}$ are their estimates (sorted in increasing order) when NPID is employed with any mean-dominant norm. Then, there exist positive constants $C_1, C_2$, independent of $T$, such that for $C_1\sqrt{\log T}\leq \zeta_T < C_2\underline{m}_T$, we have that for $d \rightarrow \infty$,
\begin{equation}
\label{mainresult_theorem}
\mathbb{P}\,\left(\hat{N} = N, \max_{j=1,\ldots,N}\left\lbrace\left|\hat{r}_j - r_j\right|/\gamma_{j,T}^2
\right\rbrace \leq d \right)\xrightarrow[T\rightarrow \infty]{} 1.
\end{equation}
\end{theorem}


It is known \citep{Dumbgen1991} that in the case of a single non-parametric change-point detection, it is possible to achieve (using the notation from \cite{Dumbgen1991} simplified to the real-valued variable setting) $\Delta^2\left|\hat{r} - r\right| = \mathcal{O}_P(1)$, where $r$ and $\hat{r}$  are the true and estimated change-points, respectively, and $\Delta = \sup_{u \in \mathbb{R}}|F_{r}(u) - F_{r+1}(u)|$, where $F_{t}(\cdot)$ is the cumulative distribution function at time point $t$. Previous attempts, for example  \cite{Carlstein1988}, showed consistency results for the estimated change-point but with a worse rate. Our result attains the optimal consistency rate in the case of multiple change-points, something that was also obtained in \cite{Zouetal2014} but under stronger assumptions than our (A1). In particular, \cite{Zouetal2014} require that there is an upper bound, denoted by $\bar{K}_n$ for the true number of change-points. This upper bound, $\bar{N}$, is such that $\delta_T/(\bar{N}^4(\log \bar{N})^2(\log T)^{2+c}) \xrightarrow[T\to\infty]{} \infty$, where $\delta_T$ is the minimum distance between consecutive change-points and $c > 0$; more details on this restriction can be found later in this section.
In addition, \cite{Zouetal2014} require almost sure convergence of the empirical to the true CDF (Glivenko-Cantelli theorem). Proof of Theorem \ref{consistency_theorem} uses Dvoretzky-Kiefer-Wolfowitz inequality, which is necessary for the Glivenko-Cantelli theorem, but not vice versa. Furthermore, the consistency results in \cite{Zouetal2014} hold for continuous distributions within each homogeneous segment (it is highlighted though that, in practice, the distributions can also be discrete or mixed), while in NPID the distributions between two change-points can be (either in theory or in practice) continuous, discrete, or mixed, as is also the case in \cite{Carlstein1988} and \cite{Dumbgen1991} for the single change-point detection scenario. A more thorough comparison between NPID and the main competitors in the literature is given later in this section. An earlier attempt by \cite{Lee_1996} uses weighted empirical measures to detect a distributional difference over a running window. A restrictive assumption on the difference between two neighbouring distributions in that work leads to the almost-sure consistency for locations with the rate of $\mathcal{O}(\log T)$; a practical drawback of the method is the need to choose the window length $A_T$, which our approach bypasses.

Our method, under relatively weak assumptions, attains the optimal rate for the distances between the true and the estimated change-point locations. It can be used in the detection of distributional changes in difficult structures such as ones involving short spacings between consecutive change-points and/or high degrees of distributional similarity across neighbouring segments. We illustrate the practical performance of NPID in Sections \ref{sec:simulation} and \ref{sec:real_data}.

Motivated by Theorem \ref{consistency_theorem}, we use in practice thresholds of the form
\begin{equation}
\label{threshold}
\zeta_T = C\sqrt{\log T},
\end{equation}
where the choice of the constant $C$ will be discussed in Section \ref{sec:variants} for two examples of mean-dominant norms: $L_2(\cdot)$ and $L_\infty(\cdot)$, as in \eqref{mean_dominant}. 
\vspace{0.05in}
\\
{\textbf{Comparison with the main competitors.}} We now provide a further comparison of NPID to its main competitors, the NMCD method of \cite{Zouetal2014} and the NWBS algorithm of \cite{Padilla_et_al_2019}. We do not include in this comparison the ECP approach of \cite{Matteson_James_2014} as the latter uses moment assumptions and is therefore not classically non-parametric (in particular, it is not invariant with respect to monotone transformations of the data). Furthermore, the method of \cite{Lee_1996} is also not included in the comparison as it is shown in \cite{Zouetal2014} that the estimated number and locations of the change-points, using the method of \cite{Lee_1996}, are not satisfactory; more specifically, in a simulation study, it was shown that the resulting models led to overfit in all scenarios tested.

Under the mild assumption (A1), NPID consistently estimates the true number and the locations of the change-points, achieving the optimal rate $\mathcal{O}_{P}(1)$ for the estimated locations. As discussed earlier, the same rate is attained by the NMCD, but under more restrictive assumptions. The NMCD approach is a two-step process in which the number of change-points is first estimated through a version of the Bayesian Information Criterion (BIC), and the locations are estimated next. For $\bar{N}$ being the maximum permitted number of change-points (the fixed-$N$ scenario), NMCD imposes the restrictive assumption $\delta_T/(\bar{N}^4(\log \bar{N})^2(\log T)^{2+c}) \rightarrow \infty$ for $c > 0$. This means that, for the consistency result to hold, the minimum distance between consecutive change-points needs to be of order at least $O((\log T)^{2+c})$; this is as long as the true number of change-points is kept fixed and not allowed to increase with the sample size. In the case in which $\bar{N}$ is allowed to increase with $T$, the condition on the order of $\delta_T$ becomes even stricter. Such strict assumptions on $\delta_T$ are absent in our method; Assumption (A1) essentially requires that the product between the distance of consecutive change-points and the square of the associated magnitude of change should be of order at least $\log T$. In addition, \cite{Zouetal2014} require the magnitude of the changes to be bounded away from zero. In contrast, in NPID, if the minimum distance between two consecutive change-points, $\delta_T$, is of a higher order than $\mathcal{O}(\log T)$, then Assumption (A1) implies that the magnitudes of the changes could go to zero as $T \to \infty$.

With regards to the NWBS approach of \cite{Padilla_et_al_2019}, the rate of consistency for the change-point locations obtained through NWBS is worse than that obtained by NPID, being off by a logarithmic factor. In addition, \cite{Padilla_et_al_2019} work under an assumption that requires $\sqrt{\min_{j \in \{1,\ldots,N+1\}}\{\delta_{j}\}} \min_{j}\{\sup_{x \in \mathbb{R}}|\Delta_j(x)|\} \geq \underline{C}\sqrt{\log T}$, while in the same concept we require the less restrictive $\min_{j \in \{1,\ldots,N\}}\{\min\{\delta_j,\delta_{j+1}\}\sup_{x \in \mathbb{R}}|\Delta_j(x)|\} \geq \underline{C}\sqrt{\log T}$, where $\underline{C}$ is a large enough positive constant. This milder assumption (imposing a lower bound on the minimum of the products rather than on the product of the minimums) is the result of the guaranteed (on a high-probability event) change-point isolation prior to detection, enjoyed by our methodology. Furthermore, in NWBS, the distributional change at each change-point is measured through the Kolmogorov-Smirnov (KS) distance. The NPID methodology is more general, as it uses mean-dominant norms, of which $L_\infty(\cdot)$ (corresponding to KS) is but one example.
In signals with a large number of change-points, NWBS needs to increase the number $M$ of intervals drawn and this has a negative effect on the computation time, which is significantly greater than that for NPID; see Section \ref{subsec:Compcost} as well as the simulations in Section \ref{sec:simulation}. By contrast, due to the interval expansion approach, NPID is certain to examine all possible change-point locations leading to better practical performance with more predictable execution times. 
\subsection{Information Criterion approach}
\label{subsec:sSIC}
In NPID, the detection is based on the comparison of aggregated (through mean-dominant norms) contrast function values to a threshold, $\zeta_T$. In Section \ref{subsec:parameter_choice}, we explain how the threshold constant $C$ (formula (\ref{threshold})) is carefully chosen through a large-scale simulation study. The values obtained for the $L_\infty$ and the $L_2$ mean-dominant norms are also shown to control the Type I error; see the discussion in Section \ref{subsec:parameter_choice}, as well as the relevant results in Tables 1 and 2 of the online supplement. Even though NPID is robust to small deviations from the optimal threshold value, misspecification of the threshold can possibly lead to the misestimation of the number and/or the location of change-points. In order to reduce the dependence of our methodology on the threshold choice, we propose in this section an approach which starts by possibly overestimating the number of change-points and then creates a solution path, with the estimates ordered in importance according to a certain predefined criterion. The best fit is then chosen, based on the optimization of a model selection criterion.
\vspace{0.04in}
\linebreak
{\textbf{The solution path algorithm.}} With $\zeta_T^*$ being the optimal choice for the threshold value (more details can be found in Section \ref{subsec:parameter_choice}), we denote $\hat{N} := \hat{N}(\zeta_T^*)$. For given input data, we run NPID with a threshold set to $\tilde{\zeta}_T < \zeta_T^*$. Let $\tilde{C}$ be the $\tilde{\zeta}_T$-associated constant in \eqref{threshold}. We estimate $J \geq \hat{N}(\zeta_T^*)$ change-points denoted by $\tilde{r}_j, j=1,\ldots, J$. These are sorted in increasing order in $\tilde{S} = \left[\tilde{r}_1, \tilde{r}_2, \ldots, \tilde{r}_{J}\right]$. After this overestimation step, the idea is to remove change-points according to their mean-dominant norm values. The first step in the solution path algorithm is, for $\tilde{r}_0 = 0$ and $\tilde{r}_{J+1} = T$, to collect triplets $(\tilde{r}_{j-1},\tilde{r}_j,\tilde{r}_{j+1})$, $\forall\left\lbrace 1,\ldots, J\right\rbrace$ and to calculate $\tilde{B}_{\tilde{r}_{j-1}+1, \tilde{r}_{j+1}}^{\tilde{r}_j}(X_i)$ with $\tilde{B}_{s,e}^{b}(X_i)$ defined in \eqref{contrast_non_par}. Then, for $$\boldsymbol{B}^{\tilde{r}_j} = \left(\tilde{B}_{\tilde{r}_{j-1}+1, \tilde{r}_{j+1}}^{\tilde{r}_j}(X_1), \tilde{B}_{\tilde{r}_{j-1}+1, \tilde{r}_{j+1}}^{\tilde{r}_j}(X_2), \ldots, \tilde{B}_{\tilde{r}_{j-1}+1, \tilde{r}_{j+1}}^{\tilde{r}_j}(X_T)\right)$$
we employ the same mean-dominant norm $L(\cdot)$, as in the overestimation step explained above, on $\boldsymbol{B}^{\tilde{r}_j}$, for each $j \in \left\lbrace 1,\ldots, J\right\rbrace$. For $m = {\rm argmin}_{j}\left\lbrace L\left(\boldsymbol{B}^{\tilde{r}_j}\right) \right\rbrace$ we remove $\tilde{r}_m$ from the set $\tilde{S}$, reduce $J$ by 1, relabel the remaining estimates (in increasing order) in $\tilde{S}$, and repeat this estimate removal process until $\tilde{S} = \emptyset$. At the end of this change-point removal process, we collect the estimates in a vector
\begin{equation}
\label{orderedcpts}
\boldsymbol{b} = \left(b_1,b_2,\ldots,b_J\right),
\end{equation}
where $b_J$ is the estimate that was removed first, $b_{J-1}$ is the one that was removed second, and so on. The vector $\boldsymbol{b}$ is referred to as the solution path and is used to give a range of different fits.
\vspace{0.04in}
\linebreak
{\textbf{Model selection through BIC.}} We define the collection $\left\lbrace\mathcal{M}_j\right\rbrace_{j = 0,1,\ldots,J}$ where $\mathcal{M}_{0} = \emptyset$ and $\mathcal{M}_j = \left\lbrace b_1,b_2,\ldots,b_j\right\rbrace$. Among the collection of models $\left\lbrace \mathcal{M}_j\right\rbrace_{j=0,1,\ldots,J}$, we propose to select the one that minimizes a version of Schwarz's Bayesian Information Criterion (BIC) as given in \cite{Zouetal2014}. For each model $\mathcal{M}_j, j=0,1,\ldots,J$, the BIC provides a balance between the likelihood and the number of change-points by incorporating an appropriate penalty for larger values of $j$. More specifically, for $\boldsymbol{b}$ as in \eqref{orderedcpts}, we denote by $b^{*,j}_{1}, \ldots b^{*,j}_{j}$ the first $j$ elements of $\boldsymbol{b}$ sorted in increasing order. Through the solution path algorithm, for each $j \in \{1,2,\ldots,J\}$, the model $\mathcal{M}_j$ includes the $j$ most important estimated change-points. Let $b^{*,j}_0 = 0$ and $b_{j+1}^{*,j} = T$ when we work with model $\mathcal{M}_j$, for any $j$. The preferred model is then identified by minimizing, for $j \in \left\lbrace 0,1,\ldots, J\right\rbrace$,
\begin{equation}
\label{BIC}
{\rm BIC}(j) = -S_{T}\left(b^{*,j}_0, b^{*,j}_{j+1}, \ldots,b^{*,j}_{j+1}\right) + j\,p_T,
\end{equation}
where
\begin{align}
\label{S_T}
\nonumber S_{T}\left(b_0^{*,j},b_1^{*,j},\ldots,b_{j+1}^{*,j}\right) &=  T\sum_{i=0}^{j}\sum_{l=2}^{T-1}\left\lbrace\vphantom{\sum_{j=0}^{N}}\frac{b_{i+1}^{*,j} - b_{i}^{*,j}}{l(T-l)}\left[\hat{F}_{b_i^{*,j}}^{b_{i+1}^{*,j}}(X_{[l]})\log\left(\hat{F}_{b_i^{*,j}}^{b_{i+1}^{*,j}}(X_{[l]})\right)\right.\right.\\
& \qquad \left.\left.\qquad + \left(1 - \hat{F}_{b_i^{*,j}}^{b_{i+1}^{*,j}}(X_{[l]})\right)\log\left(1 - \hat{F}_{b_{i}^{*,j}}^{b_{i+1}^{*,j}}(X_{[l]})\right)\right]\vphantom{\sum_{j=0}^{N}}\right\rbrace
\end{align}
for $\hat{F}_{b_{i}^{*,j}}^{b_{i+1}^{*,j}}(u) = \frac{1}{b_{i+1}^{*,j} - b_{i}^{*,j}}\sum_{t = b_{i}^{*,j}+1}^{b_{i+1}^{*,j}}1_{\{X_t \leq u\}}$. The expression in \eqref{S_T}, whose expectation is maximized at the true change-point locations, is an integrated form of the profile log-likelihood function in the presence of the $j$ change-point within each candidate model $\mathcal{M}_j$. The term $p_T$ in \eqref{BIC} is the penalty which goes to infinity with $T$. \cite{Zouetal2014} prove theoretical consistency of the BIC-related approach by taking $p_T = \bar{K}_T^3
(\log \bar{K}_T)^2(\log T)^{2+c}$, where $\bar{K}_T$ is an upper bound on the true number of change-points; such a bound is necessary for NMCD, but not for the thresholding version of NPID, introduced earlier. A small value of the constant $c$ helps to prevent underfitting. The practical choice of $p_T$ in \cite{Zouetal2014} is $p_T = \frac{1}{2}(\log T)^{2.1}$, which is what we also use with NPID in Sections \ref{sec:simulation} and \ref{sec:real_data} (as an alternative to thresholding).
We highlight, though, that the theoretical results have been proven only when NPID is combined with thresholding; the proof requires weaker assumptions than those used in \cite{Zouetal2014}, where results related to BIC are proved (see also the relevant discussion in Section \ref{subsec:theory}).
\section{Computational complexity and practicalities}
\label{sec:variants}
\subsection{Computational cost}
\label{subsec:Compcost}
With $\delta_T$ being the minimum distance between two change-points, and $\lambda_T$ the interval-expansion parameter, we need $\lambda_T < \delta_T$ to achieve isolation. Since $K = \lceil T/\lambda_T\rceil > \lceil T/\delta_T \rceil$ and the total number, $M$, of distinct intervals required to scan the data is no more than $2K$ ($K$ intervals from each expanding direction), in the worst case scenario we have $M = 2K > 2\left\lceil\frac{T}{\delta_T}\right\rceil$. The cost of computing the non-parametric CUSUM, $\tilde{B}_{s,e}^b(u)$, in \eqref{contrast_non_par} is linear in time. Because we are aggregating over all the non-parametric CUSUM values obtained for the different $X_1, \ldots, X_T$, and since this is done for at most $M$ intervals, we conclude that the computational cost of our algorithm in the worst possible case is of order $\mathcal{O}\left(M T^2\right) = \mathcal{O}\left(T^3 \delta_T^{-1}\right)$. For the NWBS algorithm of \cite{Padilla_et_al_2019}, in order to guarantee consistency of NWBS, the drawing of $S > \frac{T^2}{\delta_T^2}\log\left(T/\delta_T\right)$ intervals is needed, which means that the computational complexity of NWBS is at least of order $\mathcal{O}\left((T^4/\delta_T^2)\log (T/\delta_T)\right)$. Therefore, NPID enjoys significant speed gains over NWBS. The reason behind this difference in the computational complexity of the methods is that in NWBS the start- and end-points of the randomly drawn intervals have to be chosen, whereas in NPID, depending on the expanding direction, we keep the start- or the end-point fixed.  As explained in Section \ref{subsec:methodology}, in the case of no detection, our method will have to be applied to larger intervals until a detection occurs, which means that the most computationally expensive scenario for NPID is when there are no change-points in a given data sequence. On the other hand, for signals with a large number of frequently occurring change-points, NPID will keep operating on short intervals, leading to fast computation times. This is not the case for many of the competitors; see for example \cite{Padilla_et_al_2019} and \cite{Matteson_James_2014}, where it is explained that the computational time for binary segmentation related approaches and some dynamic programming algorithms (for example the Segment Neighbourhood Search (SNS) algorithm as in \cite{Auger_Lawrence}) increases proportionally to the number of change-points.
This is also the case with NMCD, in which the SNS dynamic programming algorithm is employed for the maximization of an objective function. Given a maximum number, $B$, of change-points to search for, SNS calculates, based on a given cost function, the optimal segmentations for each $i \in \left\lbrace 0, 1, \ldots, B\right\rbrace$. If all the segment costs have already been computed, SNS has a computational cost of $\mathcal{O}(BT^2)$. In NMCD, the cost for a single segment is $\mathcal{O}(T)$ \citep{Haynes_Fearnhead2014}, which means that for all segments the computational cost is cubic in time. Therefore, the total computational cost of NMCD as first developed in \cite{Zouetal2014} is $\mathcal{O}(BT^2 + T^3)$. To improve upon such a substantial cost, a screening algorithm was introduced in \cite{Zouetal2014}; it reduces the computational complexity, but it is shown in \cite{Haynes_Fearnhead2014} that such a pre-processing step affects the accuracy of NMCD. Taking all these into consideration, \cite{Haynes_Fearnhead2014} extend NMCD by developing a computationally efficient approach that first simplifies the segment cost to be of $\mathcal{O}(\log T)$ rather than linear in time, and, second, applies the PELT dynamic programming approach as in \cite{Killick_PELT}, which is substantially quicker than the SNS algorithm.
\subsection{Parameter choice}
\label{subsec:parameter_choice}

{\textbf{Choice of the threshold constant.}} In order to select the $C$ in \eqref{threshold}, we ran a large-scale simulation study involving a wide range of signals in terms of both the number and the type of the change-points. The number of change-points, $N$, was generated from the Poisson distribution with rate parameter $N_{\alpha}\in \left\lbrace 4,8,12 \right\rbrace$. For $T \in \left\lbrace 100,200,400,800 \right\rbrace$, we uniformly distributed the change-point locations in $\left\lbrace 1,\ldots,T \right\rbrace$. Then, at each change-point we introduced (with equal probability) one of the following three types of changes:
\begin{itemize}
\item[(M)] Change in the mean of the observations;
\item[(V)] change in the variance of the observations;
\item[(D)] general change in the distribution beyond the mean or the variance only.
\end{itemize}
More information on these three types of changes can be found in Section 2 of the supplement. The simulation procedure was followed for two examples of mean-dominant norms; $L_2(\cdot)$ and $L_\infty(\cdot)$ in \eqref{mean_dominant}. The best behavior occurred when, approximately, $C=0.6$ and $C = 0.9$ for $L_2$ and $L_\infty$, respectively. From now on, whenever relevant, these values will be referred to as the default constants.

In an attempt to measure the Type I error obtained from this choice of threshold constants under the scenario of no change-points, we ran 100 replications for twenty different scenarios of no change-points, covering scenarios from both continuous and discrete distributions. We highlight here that our procedure is invariant under monotone transformations of the data; so, for example, we would expect identical performance for Gaussian and Cauchy models with no change-points (or any other continuous distribution), and the fact that we explicitly study both these models below only serves as an extra check of the correctness of our implementation. The models used are:
\begin{itemize}
\item[Gaussian]: The distribution used is the standard normal. There are no change-points and the lengths of the data sequences are $T \in \left\lbrace 30, 75, 200, 500\right\rbrace$.
\item[Cauchy]: The distribution used is the Cauchy with location and scale parameters equal to zero and one, respectively. There are no change-points and the lengths of the data sequences are $T \in \left\lbrace 30, 75, 200, 500\right\rbrace$.
\item[Poisson]: The distribution used is the Poisson with mean equal to $\lambda$, where $\lambda \in \left\lbrace 
0.3, 3, 30 \right\rbrace$. There are no change-points and the lengths of the data sequences are $T \in \left\lbrace 30, 75, 200, 500\right\rbrace$.
\end{itemize}
Tables 1 and 2 in Section 1 of the supplement present the frequency distribution of $\hat{N} - N$ for all the above scenarios and for both the $L_{\infty}$ and $L_2$ mean-dominant norms, when their respective default threshold constants are used. We conclude that the default constants lead to acceptably low values of the Type I error in all scenarios examined.

In the BIC-based approach of Section \ref{subsec:sSIC}, the first step requires the detection of change-points using threshold $\tilde{\zeta}_T < \zeta_T$. In practice, we take $\tilde{C}$, the constant related to $\tilde{\zeta}_T$, to be 20\% less than the default constant $C$.
\vspace{0.05in}
\\
{\textbf{Choice of the expansion parameter $\lambda_T$.}} We use $\lambda_T = 15$ in all examples shown in this paper. A lower value of $\lambda_T$ would lead to finer examination of the data, at the expense of an increased computational cost. Table \ref{table_lambda} gives the execution speeds for models (T1) and (T2) as defined below, on a 2.80GHz CPU with 16 GB of RAM. The models are
\begin{itemize}
\item[(T1)] Length $l \in \left\lbrace 3000, 6000, 9000 \right\rbrace$, with changes in the mean of a normal distribution at locations $30, 60, \ldots, l - 30$. The magnitude of the changes in the mean are equal to $4$. The standard deviation is $\sigma=0.5$.
\item[(T2)] Length $l \in \left\lbrace 3000, 6000, 9000\right\rbrace$, with changes in the variance of a normal distribution at locations $250, 500, \ldots l - 250$. The mean is equal to 0, while the standard deviation between change-points takes the values $1$ and $2$ interchangeably.
\end{itemize}
The table below shows that NPID is quick for a non-parametric method, and can comfortably analyse signals with lengths in the thousands.
\begin{table}
\centering
\caption{The average computational time of NPID using either the $L_2$ or the $L_\infty$ mean-dominant norm for the data sequences (T1) and (T2)}{
\begin{tabular}{|l|l|l|l|l|}
\cline{1-5}
 & \multicolumn{2}{|c|}{} & \multicolumn{2}{|c|}{}\\
 & \multicolumn{2}{|c|}{Time for (T1) (s)} & \multicolumn{2}{|c|}{Time for (T2) (s)}\\
\hline
$T$ & $L_2$ & $L_\infty$ & $L_2$ & $L_\infty$\\
\hline
$3000$ & 4.08 & 3.39 & 8.29 & 9.27\\
$6000$ & 21.34 & 13.52 & 31.97 & 38.32\\
$9000$ & 130.66 & 36.01 & 68.92 & 90.55\\
\hline
\end{tabular}}
\label{table_lambda}
\end{table}
\subsection{Variants}
\label{subsec:var_proposed}
This section describes four different ways to further improve our method's practical performance with respect to both accuracy and speed.
\vspace{0.05in}
\\
{\textit{Data splitting:}} If the sample size is considered large, we can split the given data sequence uniformly into smaller parts (windows) to which our method will be applied. In practical implementations, 
we split the data sequence only if $T > 2000$; for smaller values of $T$ there are no significant differences in the execution times of NPID and its window-based variant.
\vspace{0.05in}
\\
{\textit{Sparse grid:}} For practical applications, instead of creating $T$ binary sequences by using every empirical quantile (see the pseudocode for our algorithm), one could take $Q$, equally spaced, quantile values $l_j, j = 1,\ldots, Q$ within the interval $[X_{[1]}, X_{[T]}]$, in order to reduce computation but still cover the range of the data. For example, for the sequence $\left\lbrace X_t\right\rbrace_{t=1,\ldots,T}$ with $R = X_{[T]} - X_{[1]}$, we could take
\begin{equation}
\label{horizontal100}
l_j = X_{[1]} + \frac{jR}{Q+1}, \quad j = 1,\ldots,Q.
\end{equation}
The loss in accuracy, caused by avoiding the use of all observations, and instead creating a more sparse grid of values, has been shown through an extensive simulation study to be negligible; most of the times the result was exactly the same. On the other hand, the grid method was approximately $T/Q$ times faster.
\vspace{0.05in}
\\
{\textit{Rescaling the CUSUM values:}} In NPID, we calculate CUSUMs based on $B_{t}(X_i) = \mathbbm{1}_{\left\lbrace X_t \leq X_i\right\rbrace}$. $\left\lbrace B_{t}(X_i)\right\rbrace_{t=1,\ldots,T}$ will mostly equal zero for low-ranked data points on input (such as $X_{[1]}$), and vice versa. Therefore, to bring the different sequences $\left\lbrace B_{t}(X_i)\right\rbrace_{t=1,2,\ldots,T}$ on a similar scale, we can rescale the contrast function values by their estimated standard deviations, $\sqrt{\hat{p}(1-\hat{p})}$, where $\hat{p}$ is the sample mean of the values of $B_{t}(X_i)$ for $t \in \left\lbrace 1,\ldots,T\right\rbrace$. On the other hand, if $\hat{p}$ is close to zero or one, then the division by $\sqrt{\hat{p}(1 - \hat{p})}$ can lead to spuriously large contrast function values and therefore, to overdetection issues. Therefore, in cases in which $\hat{p} < 0.1$ or $\hat{p} > 0.9$, we take $\sqrt{\hat{p}(1 - \hat{p})} = 0.3$, which is the value obtained when $\hat{p} \in \left\lbrace 0.1, 0.9\right\rbrace$.
\vspace{0.05in}
\\
{\textit{Restarting post-detection:}} In practice, instead of starting from the end-point (or start-point) of the right-expanding (or left-expanding) interval in which a detection occurred, the algorithm offers the option of restarting from the estimated change-point location. This increases the probability of type-I errors (in particular increasing the risk of double detection) but reduces the probability of type-II errors (i.e. reduces the chances of missing true change-points). The speed of the method is not substantially affected.
\section{Simulations}
\label{sec:simulation}
In this section, we provide a comprehensive simulation study on the performance of NPID against state-of-the-art methods in the non-parametric change-point detection literature. Table \ref{table_methods} shows the competitors. The multiscale method of \cite{Vanegas2021} is not included in our comparative simulation study because it is valid for a given quantile level, rather than their combination.

NMCD is implemented through the \textsf{R} package \textbf{changepoint.np}. To our knowledge, there is no {\textsf{R}} package for the NBS and NWBS methods described in \cite{Padilla_et_al_2019} and for our simulations, we used the {\textsf{R}} code available at \url{https://github.com/hernanmp/NWBS}.
\begin{table}
\centering
\caption{The competing methods used in the simulation study}
\begin{tabular}{|l|l|l|}
\cline{1-3}
 & & \\
Method notation & Reference & \textsf{R} package\\
\hline
ECP & \cite{Matteson_James_2014} & \textbf{ecp}\\
NMCD & \cite{Zouetal2014} & \textbf{changepoint.np}\\
NBS  & \cite{Padilla_et_al_2019} & --\\
NWBS  & \cite{Padilla_et_al_2019} & --\\
CPM & \cite{Ross_CPM} & \textbf{cpm}\\
\hline
\end{tabular}
\label{table_methods}
\end{table}
For ECP, we present the results for the E-Divisive method (as developed in \cite{Matteson_James_2014}), which performs a hierarchical divisive estimation of multiple change points. The change-points are estimated by iteratively applying a procedure for locating a single change point. While we keep ECP in our simulation study, we remark that this procedure is not classically non-parametric in the sense that it relies on moment assumptions; more specifically, it assumes that observations are independent with finite $\alpha$th absolute moments, for some $\alpha \in (0,2]$ \citep{MJ_ecp}. Therefore, unlike many other non-parametric approaches including ours, it is not invariant with respect to monotone transformations of the data.

The {\textsf{R}} package {\textbf{cpm}} contains implementation of several different models, both parametric
and non-parametric, for use on univariate streams. We give results for the Kolmogorov-Smirnov and the Cramer-von Mises statistics used in the {\textbf{cpm}} {\textsf{R}} package, suitable for continuous data. There is no function in the {\textbf{cpm}} {\textsf{R}} package for discrete random variables. Therefore, the method is excluded when our simulation models involve discrete distributions. Furthermore, in the \textbf{cpm} package, the threshold used for detection is decided through the average run length (ARL) until a false positive occurs. In our simulations, we give results for ${\rm ARL}=500$ (the default value) and, whenever the signal length, $l_s$, is greater than 500, results are also given for ${\rm ARL} = 1000\lceil l_s/1000 \rceil$. With $A$ being the value of the ARL used, the notation is CPMKS$A$ and CPMCvM$A$, respectively. For the remaining NMCD, NBS, and NWBS methods, we use the default hyperparameter values.

With respect to our method, results are presented when NPID is combined with the information criterion stopping rule discussed in Section \ref{subsec:sSIC}. The $L_\infty$ mean-dominant norm is employed and the obtained contrast function values are rescaled as explained in Section \ref{subsec:var_proposed}. We highlight, though, that the performance of our method is robust to the choice of the mean-dominant norm used.

All the signals are fully specified in Appendix \ref{sec:App_A}. Figure \ref{figure_signals} shows examples of the data generated by models (MM\_Gauss), (MM\_Gauss\_tr), (MV\_Gauss), (D1), (MD1), and (MD3).
\begin{figure}
\centering
\includegraphics[width=13cm,height=6.5cm]{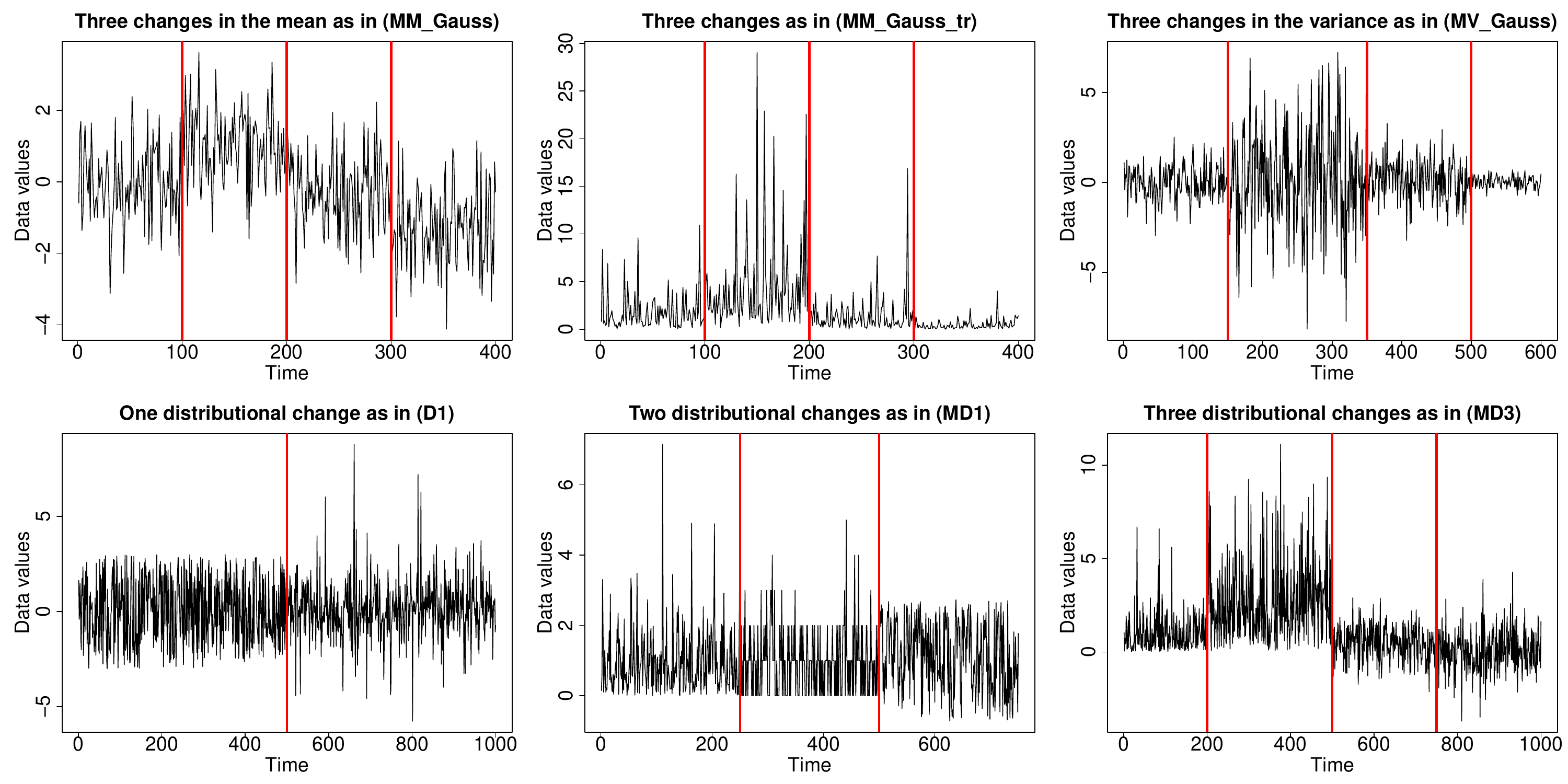}
\caption{Examples of data sequences, used in simulations. The change-point locations are indicated with red, vertical, solid lines.} 
\label{figure_signals}
\end{figure}
We ran 100 replications for each signal and the frequency distribution of $\hat{N} - N$ is shown for each method. The methods with the highest empirical frequency of $\hat{N} - N = 0$ and those within $10\%$ off the highest are given in bold. As a measure of the accuracy of the detected locations, we provide the average values (out of the 100 replications) for the scaled Hausdorff distance, $$d_H = n_s^{-1}\max\left\lbrace \max_j\min_k\left|r_j-\hat{r}_k\right|,\max_k\min_j\left|r_j-\hat{r}_k\right|\right\rbrace,$$ where $n_s$ is the length of the largest segment. In the first example of Table \ref{NC1}, which is a signal without any change-points, $d_H$ is not given since it is non-informative in no change-point scenarios. The average computation time for all methods is also provided. The \textsf{R} code to replicate the simulation study explained in this section can be found at \url{https://github.com/Anastasiou-Andreas/NPID}. Tables \ref{NC1}--\ref{MV1toMD3} summarize the results.

The results in Tables \ref{NC1}--\ref{MV1toMD3} show that NPID consistently, among all scenarios covered in the models, performs accurately regarding both the number and the locations of the estimated change-points. To be more precise, NPID is always among the top-performing methods when considering accuracy in any aspect (estimated number of change-points and estimated change-point locations); in the vast majority of cases, it is the best method overall. We highlight that, as shown by the results for the models (MM\_Gauss\_tr) and (MM\_Pois\_tr), NPID's performance remains unaffected under transformations of the data; this is expected since our method (and some of the competitors) rely on the ranks of the data only for the detection of the change-points. Furthermore, we notice a significant gap in the performance between NPID and its competitors in all models that undergo variance changes, as well as in models (MD1) and (MD3) that cover the case of more general distributional changes. Regarding the practical computational cost of NPID, the results in Tables \ref{NC1}-\ref{MV1toMD3} indicate that our proposed method can accurately analyze long signals (length in the range of thousands) in seconds.

Regarding the competing methods, CPM, NMCD, NBS, and NWBS appear to struggle to detect the change-points correctly, especially in the presence of changes in the variance. The ECP algorithm exhibits very good performance in the at-most-one change-point scenarios. While it performs well in some models with multiple change-points (mainly those involving changes in the mean), this behavior is not consistent; the method seems to struggle (as is the case for the rest of the competitors too) in the presence of variance changes.

We now focus on a comparison of the behavior of the competitors in (MM\_Gauss) and (MM\_Pois) to that under (MM\_Gauss\_tr) and (MM\_Pois\_tr); in the latter two models the exponential function was applied to the data sequences obtained by the former two models. We highlight that NPID remained unaffected under such transformations of the data. The algorithms NMCD, NBS, NWBS, and CPM seem to also remain unaffected by such transformations of the data. In contrast, ECP exhibits very accurate behavior in (MM\_Gauss) and (MM\_Pois), but its performance vastly declines when the exponential function is applied to the data; see the results for models (MM\_Gauss\_tr) and (MM\_Pois\_tr) in Table \ref{MM1MM2}. ECP works under the assumption that observations have finite $\alpha$th absolute moments for some $\alpha \in (0,2]$ \citep{Matteson_James_2014}, and transformations will inevitably affect the performance of methods that rely on moment assumptions.

Based on the aforementioned comparison of the behaviour of NPID with the best competitors in the literature, we can conclude that our method has the best behaviour overall. NPID is an accurate, reliable, and quick method for non-parametric change-point detection.
\begin{table}
\centering
\caption{Distribution of $\hat{N} - N$ over 100 simulated data sequences from the structures (NC), (M1), (V1), and (D1). The average $d_H$ and computation time are also given}
{\small{
\begin{tabular}{|l|l|l|l|l|l|l|l|}
\cline{1-8}
 &  & \multicolumn{4}{|c|}{} & & \\
 &  & \multicolumn{4}{|c|}{$\hat{N} - N$} & & \\
Method  & Model & $-1$ & 0 & $1$ & $\geq 2$ & $d_H$ & Time (s)\\
\hline
{\textbf{NPID}} &  & - & {\textbf{97}} & 3 & 0 & - & 0.84\\
CPMKS500 &  & - & 28 & 25 & 47 & - & 0.07\\
CPMCvM500 &  & - & 34 & 18 & 48 & - & 0.07\\
{\textbf{ECP}} & (NC) & - & {\textbf{94}} & 4 & 2 & - & 1.41\\
NMCD &  & - & 59 & 6 & 35 & - & 0.05\\
{\textbf{NBS}} &  & - & {\textbf{100}} & 0 & 0 & - &  0.44\\
{\textbf{NWBS}} &  & - &{\textbf{97}} & 0 & 3 & - & 9.73\\
\hline
\hline
{\textbf{NPID}}  & & 0 & {\textbf{94}} & 5 & 1 & 0.344 & 0.15\\
CPMKS500 & & 0 & 71 & 18 & 11 & 0.179 & 0.01\\
CPMCvM500 & & 0 & 69 & 17 & 14 & 0.165 & 0.01\\
{\textbf{ECP}} & (M1) & 0 & {\textbf{96}} & 4 & 0 & 0.045 & 0.31\\
NMCD & & 0 & 70 & 19 & 11 & 0.174 & 0.01\\
NBS  & & 19 & 70 & 3 & 8 & 0.303 & 0.09\\
NWBS & & 22 & 69 & 4 & 5 & 0.314 & 2.28\\
\hline
\hline
{\textbf{NPID}} & & 1 & {\textbf{86}} & 5 & 9 & 0.123 & 0.61\\
CPMKS500 & & 1 & 60 & 18 & 21 & 0.443 & 0.01\\
CPMCvM500 &  & 5 & 49 & 25 & 21 & 0.424 & 0.01\\
{\textbf{ECP}} & (V1)  & 12 & {\textbf{83}} & 3 & 2 & 0.450 & 0.52\\
NMCD &  & 0 & 44 & 23 & 33 & 0.384 & 0.01\\
NBS  &  & 93 & 1 & 1 & 5 & 0.978 & 0.15\\
NWBS  &  & 90 & 4 & 1 & 5 & 0.958 & 3.92\\
\hline
\hline
{\textbf{NPID}}  & & 2 & {\textbf{94}} & 2 & 2 & 0.075 & 2.88\\
CPMKS500 & & 0 & 7 & 11 & 82 & 0.637 & 0.17\\
CPMCvM500 & & 0 & 9 & 12 & 79 & 0.576 & 0.19\\
CPMKS1000 & & 0 & 27 & 17 & 56 & 0.442 & 0.25\\
CPMCvM1000 & & 0 & 30 & 16 & 54 & 0.392 & 0.26\\
{\textbf{ECP}} & (D1) & 0 & {\textbf{96}} & 3 & 1 & 0.036 & 7.68\\
NMCD & & 0 & 27 & 13 & 60 & 0.439 & 0.12\\
NBS  & & 53 & 41 & 6 & 0 & 0.579 & 0.96\\
NWBS & & 61 & 30 & 4 & 5 & 0.678 & 22.16\\
\hline
\end{tabular}}}
\label{NC1}
\end{table}
\begin{table}
\centering
\caption{Distribution of $\hat{N} - N$ over 100 simulated data sequences from the structures (MM1), (MM2), (MV\_Gauss), and (MV\_Gauss2). The average $d_H$ and computation time are also given}
{\small{
\begin{tabular}{|l|l|l|l|l|l|l|l|l|}
\cline{1-9}
 & & \multicolumn{5}{|c|}{} & & \\
 & & \multicolumn{5}{|c|}{$\hat{N} - N$} & & \\
Method & Model & $\leq -2$ & $-1$ & $0$ & $1$ &$\geq 2$ & $d_H$ & Time (s)\\
\hline
{\textbf{NPID}} & & 0 & 0 & {\textbf{97}} & 3 & 0 & 0.090 & 0.34\\
CPMKS500 & & 0 & 0 & 45 & 32 & 23 & 0.291 & 0.01\\
CPMCvM500 & & 0 & 0 & 62 & 17 & 21 & 0.225 & 0.01\\
{\textbf{ECP}} & (MM\_Gauss) & 0 & 0 & {\textbf{92}} & 6 & 2 & 0.108 & 2.21\\
NMCD & & 0 & 0 & 83 & 13 & 4 & 0.106 & 0.02\\
NBS  & & 36 & 6 & 43 & 7 & 8 & 0.914 & 0.23\\
NWBS & & 13 & 18 & 54 & 4 & 11 & 0.523 & 6.21\\
\hline
\hline
{\textbf{NPID}} & & 0 & 0 & {\textbf{97}} & 3 & 0 & 0.090 & 0.30\\
CPMKS500 & & 0 & 0 & 37 & 29 & 34 & 0.303 & 0.01\\
CPMCvM500 & & 0 & 0 & 50 & 20 & 30 & 0.265 & 0.01\\
ECP & (MM\_Gauss\_tr) & 0 & 32 & 59 & 9 & 0 & 0.467 & 2.01\\
NMCD & & 0 & 0 & 82 & 14 & 4 & 0.107 & 0.02\\
NBS  & & 41 & 10 & 36 & 7 & 6 & 1.027 & 0.22\\
NWBS & & 7 & 20 & 57 & 7 & 9 & 0.480 & 5.63\\
\hline
\hline
{\textbf{NPID}} & & 9 & 9 & {\textbf{81}} & 1 & 0 & 0.347 & 0.28\\
CPMKS500 & & 0 & 0 & 30 & 37 & 33 & 0.403 & 0.01\\
CPMCvM500 & & 0 & 0 & 47 & 30 & 23 & 0.295 & 0.01\\
{\textbf{ECP}} & (MM\_Student\_$t_3$) & 0 & 1 & {\textbf{84}} & 13 & 2 & 0.183 & 2.50\\
NMCD & & 0 & 0 & 47 & 30 & 23 & 0.321 & 0.02\\
NBS  & & 48 & 6 & 26 & 12 & 8 & 1.121 & 0.24\\
NWBS & & 21 & 22 & 42 & 10 & 5 & 0.619 & 6.40\\
\hline
\hline
{\textbf{NPID}} &  & 0 & 0 & {\textbf{97}} & 3 & 0 & 0.085 & 6.25\\
CPMKS500 &  & 0 & 0 & 2 & 3 & 95 & 0.452 & 0.01\\
CPMCvM500 & &  0 & 0 & 1 & 7 & 92 & 0.446 & 0.01\\
CPMKS2000 &  & 0 & 0 & 22 & 36 & 42 & 0.266 & 0.02\\
CPMCvM2000 & &  0 & 0 & 25 & 33 & 42 & 0.274 & 0.01\\
{\textbf{ECP}} & (MM\_Gauss2) & 0 & 0 & {\textbf{91}} & 9 & 0 & 0.089 & 113.25\\
NMCD &  & 0 & 0 & 72 & 24 & 4 & 0.089 & 0.05\\
NBS &  & 22 & 13 & 43 & 22 & 0 & 0.489 & 0.21\\
{\textbf{NWBS}} &  & 0 & 2 & {\textbf{96}} & 2 & 0 & 0.166 & 46.69\\
\hline
\hline
{\textbf{NPID}} &  & 0 & 9 & {\textbf{91}} & 0 & 0 & 0.131 & 0.02\\
{\textbf{ECP}} & (MM\_Pois) & 0 & 0 & {\textbf{84}} & 15 & 1 & 0.114 & 2.10\\
{\textbf{NMCD}} &  & 0 & 0 & {\textbf{89}} & 9 & 2 & 0.066 & 0.02\\
NBS &  & 0 & 43 & 54 & 2 & 1 & 0.489 & 0.21\\
NWBS &  & 0 & 50 & 45 & 4 & 1 & 0.551 & 5.71\\
\hline
\hline
{\textbf{NPID}} &  & 0 & 9 & {\textbf{91}} & 0 & 0 & 0.131 & 0.02\\
ECP & (MM\_Pois\_tr) & 1 & 67 & 28 & 4 & 0 & 0.721 & 1.83\\
{\textbf{NMCD}} &  & 0 & 0 & {\textbf{86}} & 11 & 3 & 0.098 & 0.02\\
NBS &  & 0 & 39 & 56 & 3 & 2 & 0.458 & 0.22\\
NWBS &  & 0 & 48 & 48 & 1 & 3 & 0.531 & 6.05\\
\hline
\end{tabular}}}
\label{MM1MM2}
\end{table}
\begin{table}
\centering
\caption{Distribution of $\hat{N} - N$ over 100 simulated data sequences from the structures (MV\_Gauss), (MV\_Gauss2), (MD1), (MD2), and (MD3). The average $d_H$ and computation time are also given}
{\small{
\begin{tabular}{|l|l|l|l|l|l|l|l|l|}
\cline{1-9}
 & & \multicolumn{5}{|c|}{} & & \\
 & & \multicolumn{5}{|c|}{$\hat{N} - N$} & & \\
Method & Model & $\leq -2$ & $-1$ & $0$ & $1$ &$\geq 2$ & $d_H$ & Time (s)\\
\hline
{\textbf{NPID}} &  & 0 & 1 & {\textbf{87}} & 9 & 3 & 0.102 & 0.79\\
CPMKS500 &  & 0 & 0 & 19 & 24 & 57 & 0.278 & 0.03\\
CPMCvM500 &  & 0 & 0 & 21 & 26 & 53 & 0.277 & 0.02\\
CPMKS1000 &  & 0 & 1 & 37 & 27 & 35 & 0.223 & 0.03\\
CPMCvM1000 &  & 0 & 0 & 47 & 24 & 29 & 0.194 & 0.03\\
ECP & (MV\_Gauss) & 0 & 78 & 20 & 2 & 0 & 0.631 & 3.84\\
NMCD &  & 0 & 0 & 33 & 18 & 49 & 0.197 & 0.03\\
NBS &  & 87 & 3 & 4 & 4 & 2 & 1.574 & 0.40\\
NWBS &  & 48 & 20 & 19 & 9 & 4 & 1.181 & 10.86\\
\hline
\hline
{\textbf{NPID}} &  & 0 & 3 & {\textbf{85}} & 7 & 5 & 0.171 & 2.32\\
CPMKS500 &  & 0 & 0 & 5 & 7 & 88 & 0.424 & 0.05\\
CPMCvM500 &  & 0 & 2 & 12 & 19 & 67 & 0.407 & 0.05\\
CPMKS1000 &  & 0 & 1 & 16 & 28 & 55 & 0.374 & 0.07\\
CPMCvM1000 &  & 1 & 3 & 31 & 21 & 44 & 0.391 & 0.06\\
ECP & (MV\_Gauss2) & 71 & 23 & 5 & 0 & 1 & 0.736 & 13.60\\
NMCD &  & 0 & 0 & 26 & 18 & 56 & 0.266 & 0.07\\
NBS &  & 87 & 5 & 4 & 1 & 3 & 2.934 & 0.81\\
NWBS &  & 88 & 2 & 4 & 2 & 4 & 1.513 & 21.55\\
\hline
\hline
{\textbf{NPID}} &  & 0 & 3 & {\textbf{97}} & 0 & 0 & 0.070 & 0.84\\
ECP & (MD1) & 17 & 52 & 25 & 6 & 0 & 0.919 & 5.46\\
NMCD &  & 0 & 0 & 29 & 36 & 35 & 0.312 & 0.07\\
NBS &  & 0 & 51 & 41 & 7 & 1 & 0.569 & 0.65\\
NWBS &  & 2 & 60 & 31 & 3 & 4 & 0.684 & 17.45\\
\hline
\hline
{\textbf{NPID}} &  & 0 & 0 & {\textbf{98}} & 2 & 0 & 0.069 & 0.52\\
CPMKS500 &  & 0 & 0 & 23 & 36 & 51 & 0.347 & 0.01\\
CPMCvM500 &  & 0 & 0 & 34 & 29 & 37 & 0.301 & 0.01\\
{\textbf{ECP}} & (MD2) & 0 & 0 & {\textbf{93}} & 6 & 1 & 0.080 & 3.17\\
NMCD &  & 0 & 0 & 52 & 24 & 24 & 0.150 & 0.02\\
NBS &  & 16 & 4 & 62 & 4 & 14 & 0.486 & 0.30\\
NWBS &  & 9 & 12 & 65 & 6 & 8 & 0.377 & 8.11\\
\hline
\hline
{\textbf{NPID}} &  & 0 & 11 & {\textbf{86}} & 3 & 0 & 0.173 & 2.20\\
CPMKS500 &  & 0 & 0 & 8 & 14 & 78 & 0.418 & 0.07\\
CPMCvM500 &  & 0 & 0 & 13 & 12 & 75 & 0.367 & 0.05\\
CPMKS1000 &  & 0 & 1 & 25 & 22 & 52 & 0.323 & 0.08\\
CPMCvM1000 &  & 0 & 0 & 31 & 22 & 47 & 0.288 & 0.06\\
ECP & (MD3) & 0 & 34 & 63 & 2 & 1 & 0.334 & 11.56\\
{\textbf{NMCD}} &  & 0 & 0 & {\textbf{77}} & 19 & 4 & 0.092 & 0.09\\
NBS &  & 0 & 53 & 32 & 11 & 4 & 0.555 & 0.82\\
NWBS &  & 0 & 69 & 22 & 2 & 7 & 0.624 & 21.51\\
\hline
\end{tabular}}}
\label{MV1toMD3}
\end{table}
\section{Real data examples}
\label{sec:real_data}
\subsection{Micro-array data}
\label{subsec:microarray}
In this section, we investigate the performance of our method on micro-array data, which consists of individuals with a bladder tumour. The data set can be obtained from the {\textsf{R}} package {\bf{ecp}} \citep{MJ_ecp}, and has been previously analysed in the literature. More specifically, in \cite{Matteson_James_2014} the whole multivariate dataset is analysed, while in \cite{Padilla_et_al_2019}, non-parametric change-point detection is carried out for the first individual in the data set. 
The results for NPID regarding the first individual can be found in Figure \ref{fig:ACGH_only_us}. 
\begin{figure}
\centering
\includegraphics[width=12cm,height=5cm]{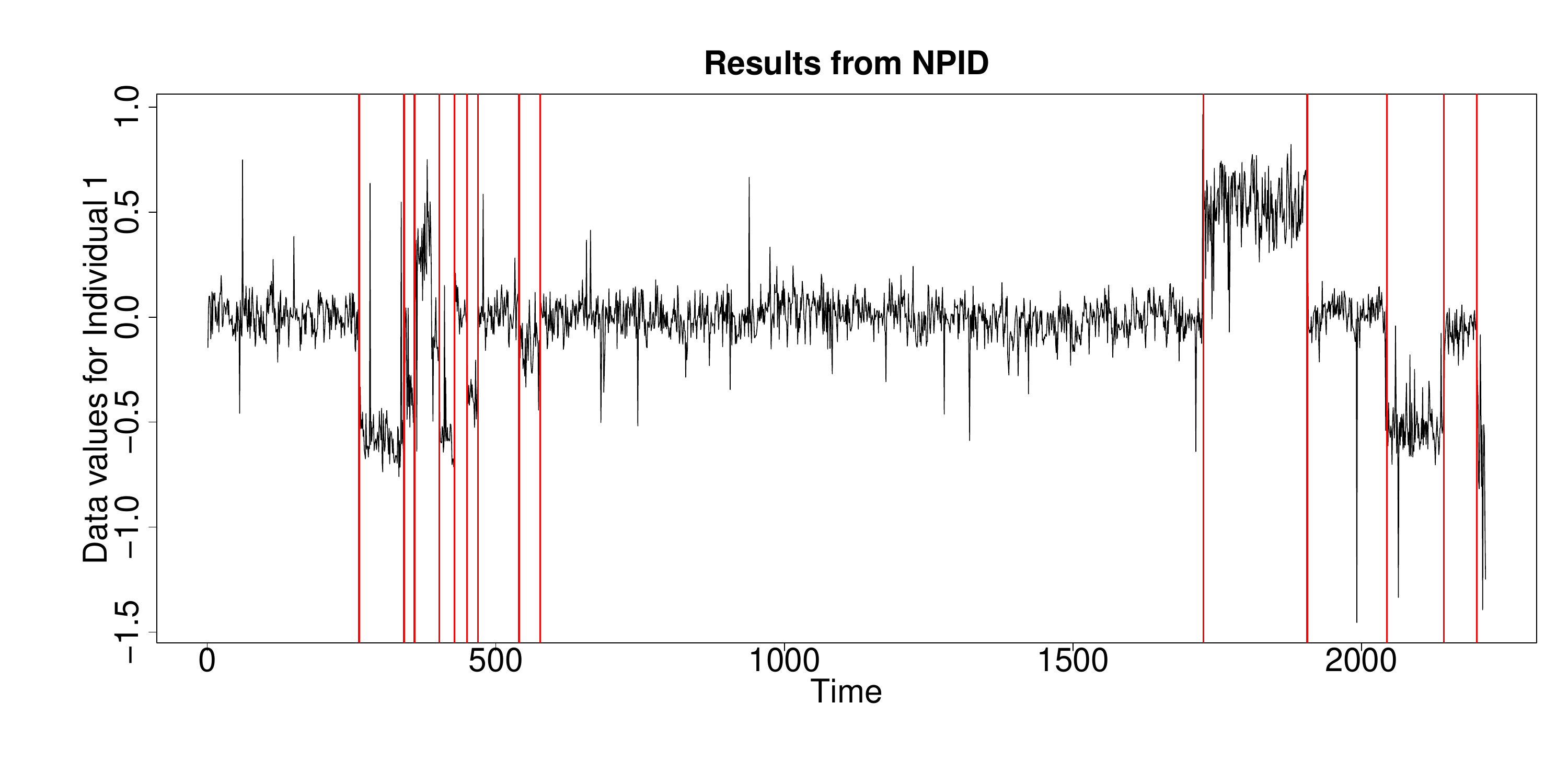}
\caption{Change-point detection for the first individual in the micro-array data set. The NPID-estimated change-point locations are marked with red solid vertical lines.}
\label{fig:ACGH_only_us}
\end{figure}
Our method seems to capture the important movements (that are mainly in the mean) in the micro-array data for this specific individual and the obtained results are similar to those presented in \cite{Padilla_et_al_2019}, with a deviation regarding the estimated number of change-points (14 for NPID compared to 16 for NWBS). We further investigate the performance of the methods for other individuals in the data set, also including ECP and NMCD in the analysis. The results for Individuals 4 and 39 can be found in Figures \ref{fig:ind_4} and \ref{fig:ind_39}, respectively. ECP and NMCD may be suspected of overestimation in both subjects. NWPS appears to slightly overestimate change-points for Individual 4 and underestimate change-points for Individual 39. NPID exhibits (at least visually) good performance in both subjects.
\begin{figure}
\centering
\includegraphics[width=14cm,height=6cm]{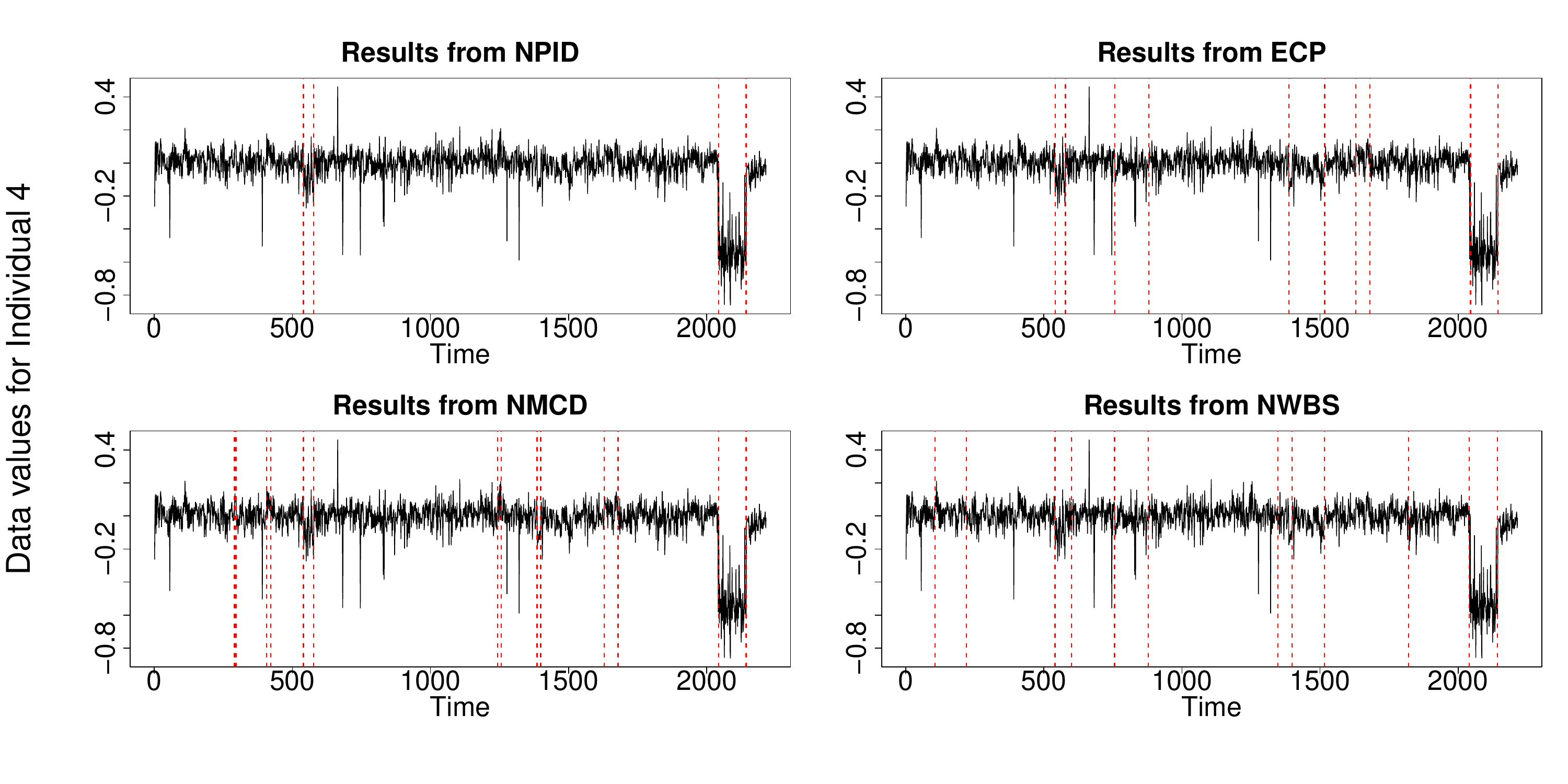}
\vspace{-0.15in}
\caption{Change-point detection for Individual 4 in the micro-array data set. The estimated change-point locations are given with dashed vertical lines. Top row: The results for the NPID and ECP methods. Bottom row: The results for the NMCD and NWBS methods.}
\label{fig:ind_4}
\end{figure}
\begin{figure}
\centering
\includegraphics[width=14cm,height=6cm]{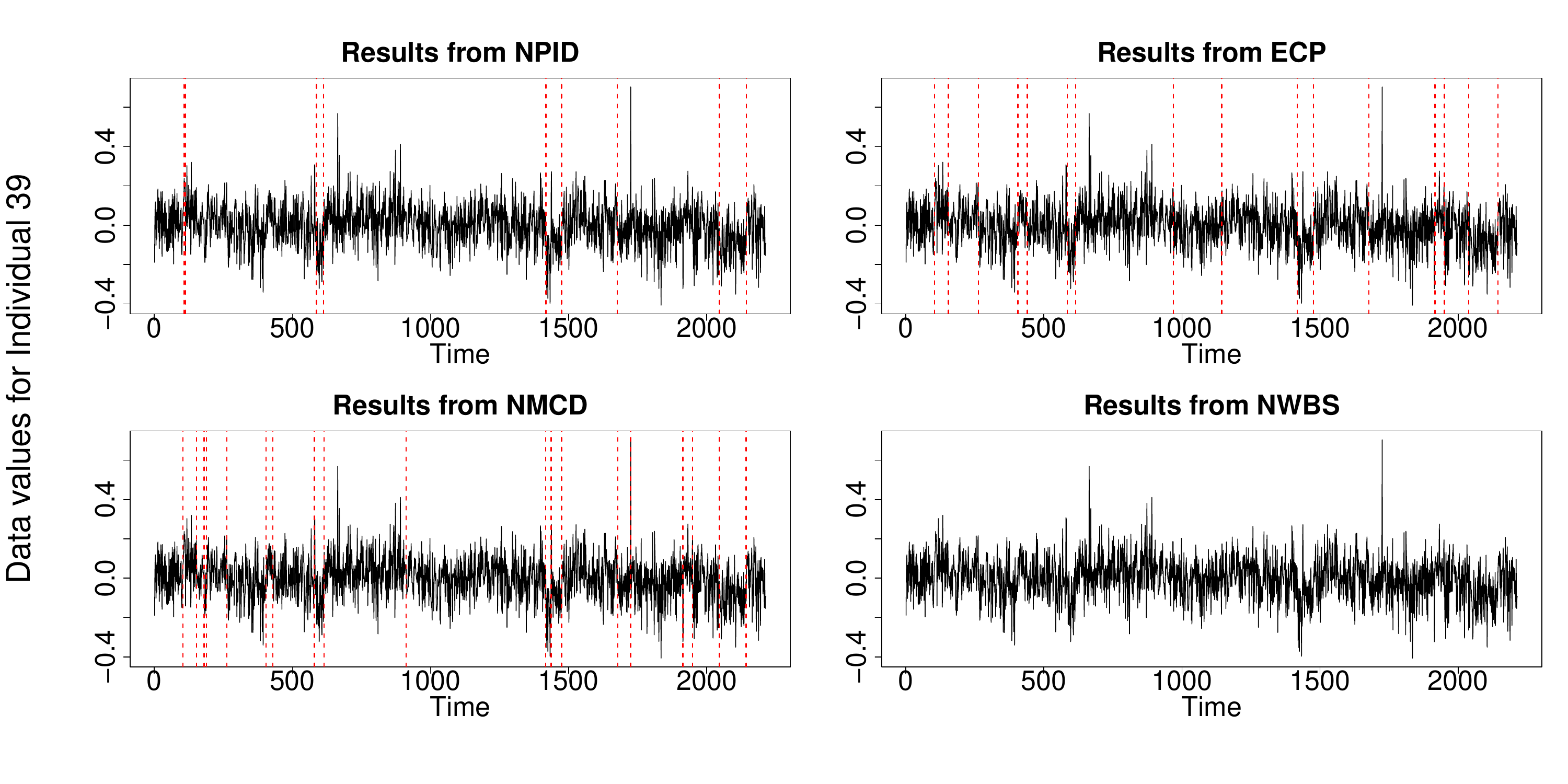}
\vspace{-0.15in}
\caption{Change-point detection for Individual 39 in the micro-array data set. The estimated change-point locations are given with dashed vertical lines. Top row: The results for the NPID and ECP methods. Bottom row: The results for the NMCD and NWBS methods.}
\label{fig:ind_39}
\end{figure}
In real-data examples, it is a challenge to decide which of the methods give the best
segmentation. However, the NPID solution path algorithm (Section \ref{subsec:sSIC}) can be used to obtain a range of different segmentation models providing users with the flexibility to choose according to their preferred model selection criterion. In Figure \ref{fig:solution_path}, we show the solution path for Individual 4 with the 12 most prominent change-points.
\begin{figure}
\centering
\includegraphics[width=12cm,height=5cm]{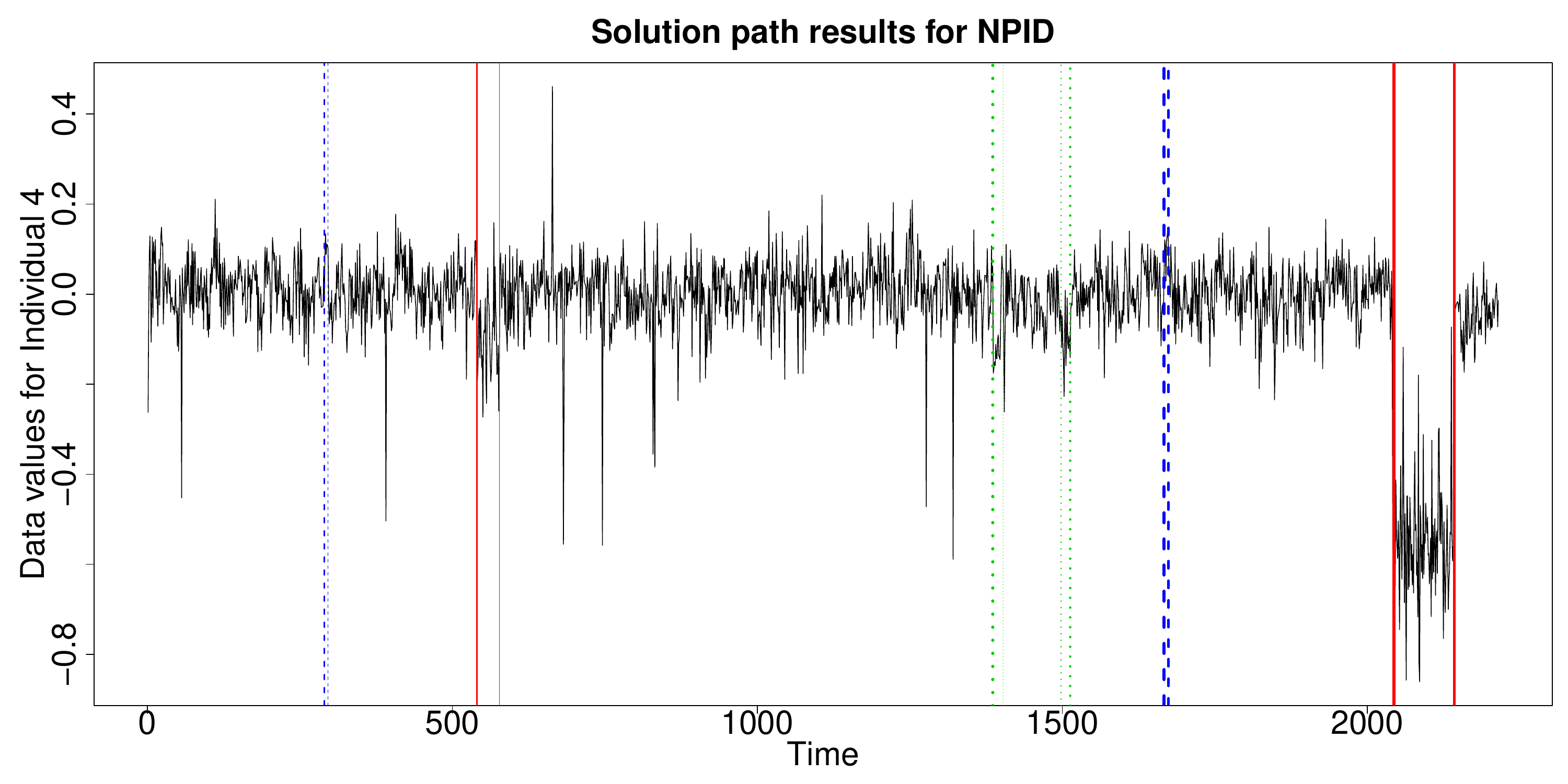}
\caption{Solution path algorithm results for Individual 4 in the micro-array data set. The 12 most important locations according to the solution path algorithm are given. Three subcategories of importance are created. The four most important locations are given with red solid lines, while the elements in positions 5-8 and 9-12 of the solution path vector are presented with blue dashed and green dotted lines, respectively. Within each subcategory, the thicker the line, the more important the estimated location.}
\label{fig:solution_path}
\end{figure}
\subsection{Major American stock indices.}
\label{subsec:DJIA}
We first analyze the Dow Jones Industrial Average (DJIA) index daily log returns from the 14\textsuperscript{th} of April 2020 until the 17\textsuperscript{th} of April 2025. (All data in this section is taken from \url{https://fred.stlouisfed.org/series/DJIA}.) Figure \ref{fig:DJIA_5_all} shows the results for NPID, as well as for the ECP, NMCD, and NWBS algorithms.
NPID, ECP and NMCD exhibit a similar behavior detecting 4, 3, and 5 change-points, respectively, while NWBS does not detect any change-points. It is interesting to see that NPID and NMCD capture, through the last estimated change-point, an important increase in the volatility having taken place in April 2025. ECP and NWBS appear insufficiently sensitive in this instance.

Further, we consider the S\&P 500 and Nasdaq composite indices over the same time period; the results are provided in Figures \ref{fig:SP_500_all} and \ref{fig:Nasdaq_all}, respectively. Regarding NPID, NMCD, and NWBS, the results for the three market indexes are very similar. NPID and NMCD capture important distributional changes including the recent volatility change in April 2025; NWBS does not detect any change-points in any index. ECP exhibits similar behaviour for the S\&P 500 and Nasdaq composite indexes, while it detects one more change-point for the DJIA index.
\begin{figure}
\centering
\includegraphics[width=14cm,height=6cm]{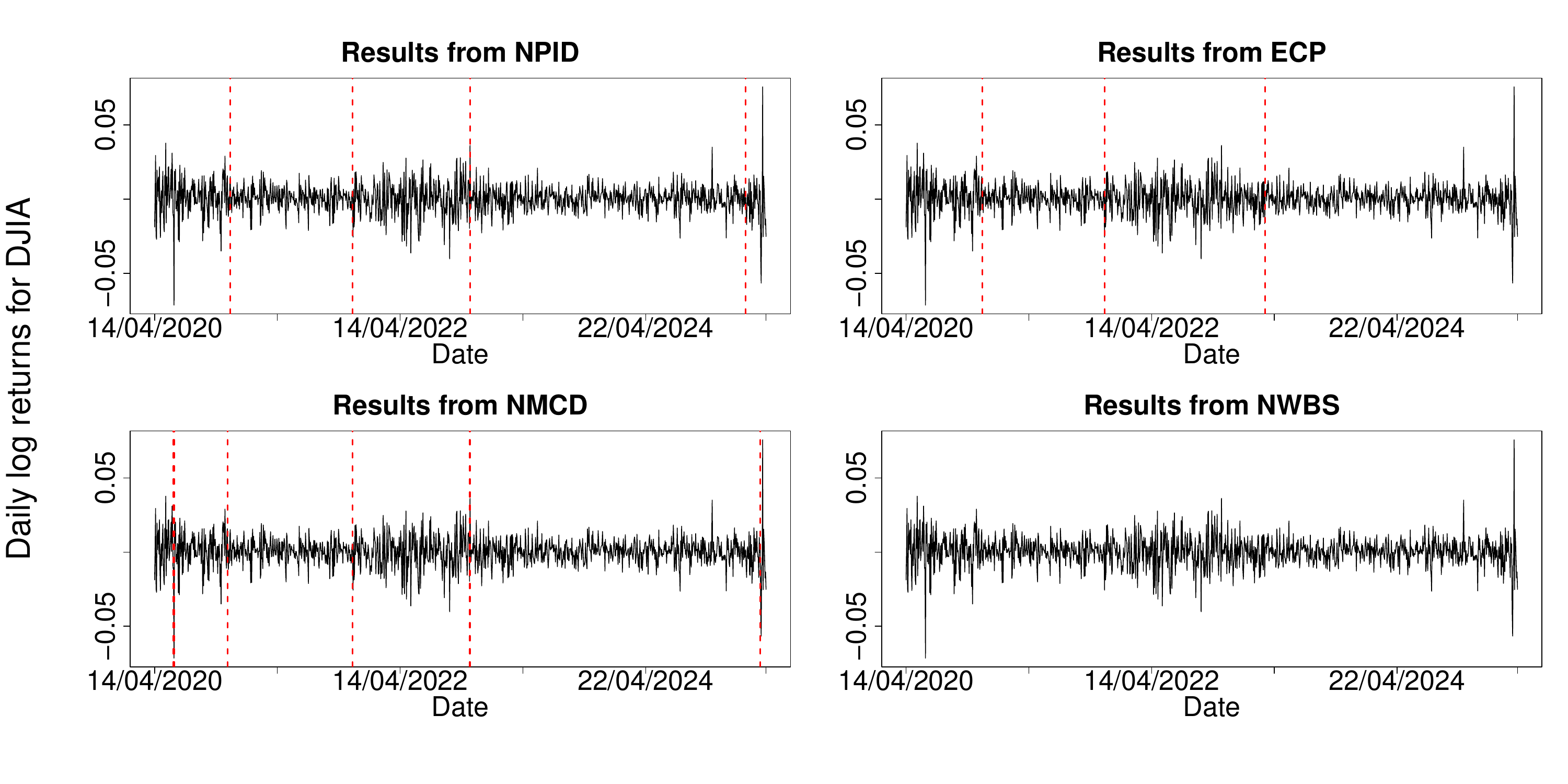}
\vspace{-0.1in}
\caption{Change-point detection for the DJIA index daily log returns. Top row: The results for the NPID and ECP methods. Bottom row: The results for the NMCD and NWBS methods.}
\label{fig:DJIA_5_all}
\end{figure}
\begin{figure}
\centering
\includegraphics[width=14cm,height=6cm]{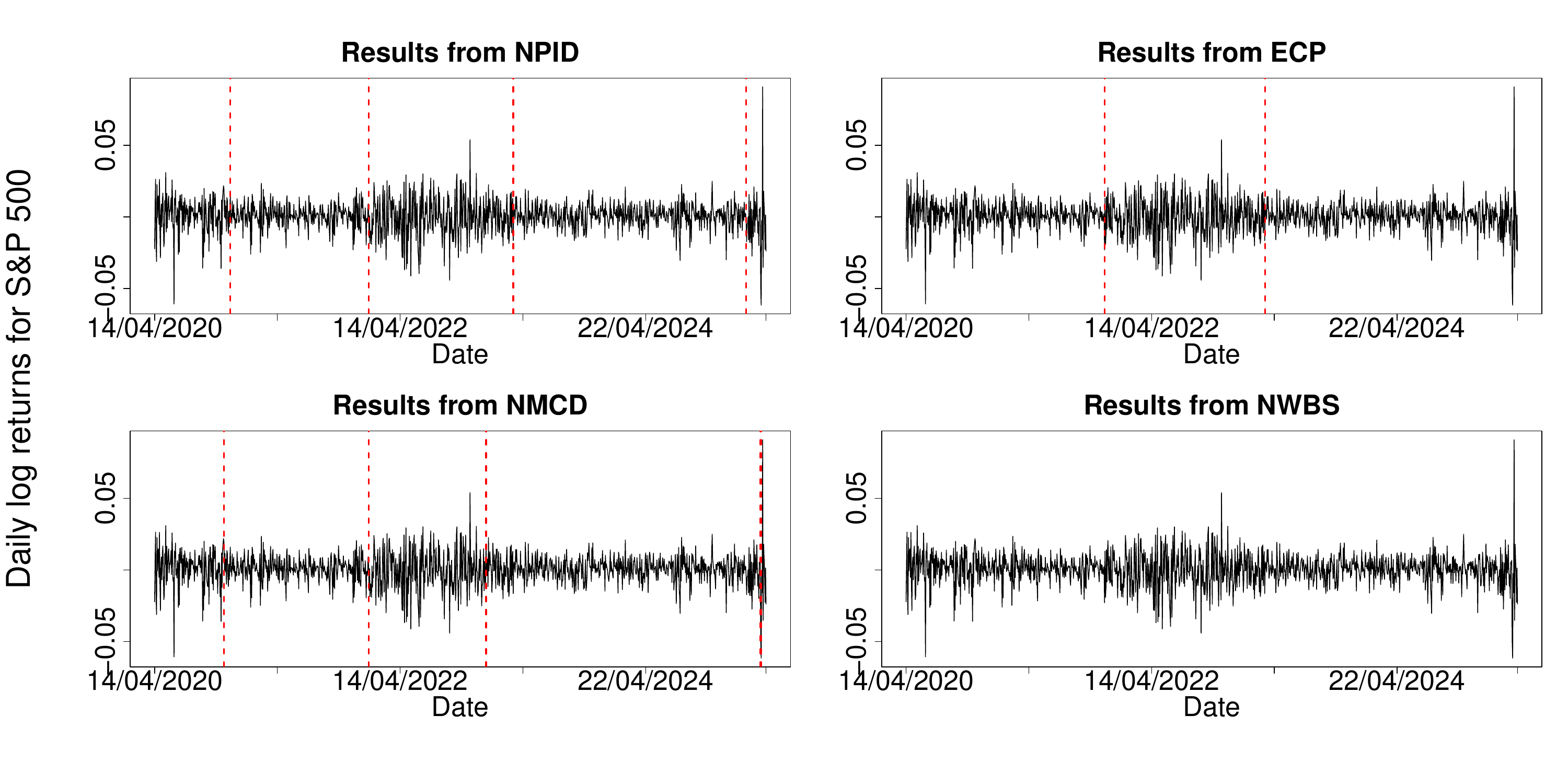}
\vspace{-0.1in}
\caption{Change-point detection for the S\&P 500 index daily log returns. Top row: The results for the NPID and ECP methods. Bottom row: The results for the NMCD and NWBS methods.}
\label{fig:SP_500_all}
\end{figure}
\begin{figure}
\centering
\includegraphics[width=14cm,height=6cm]{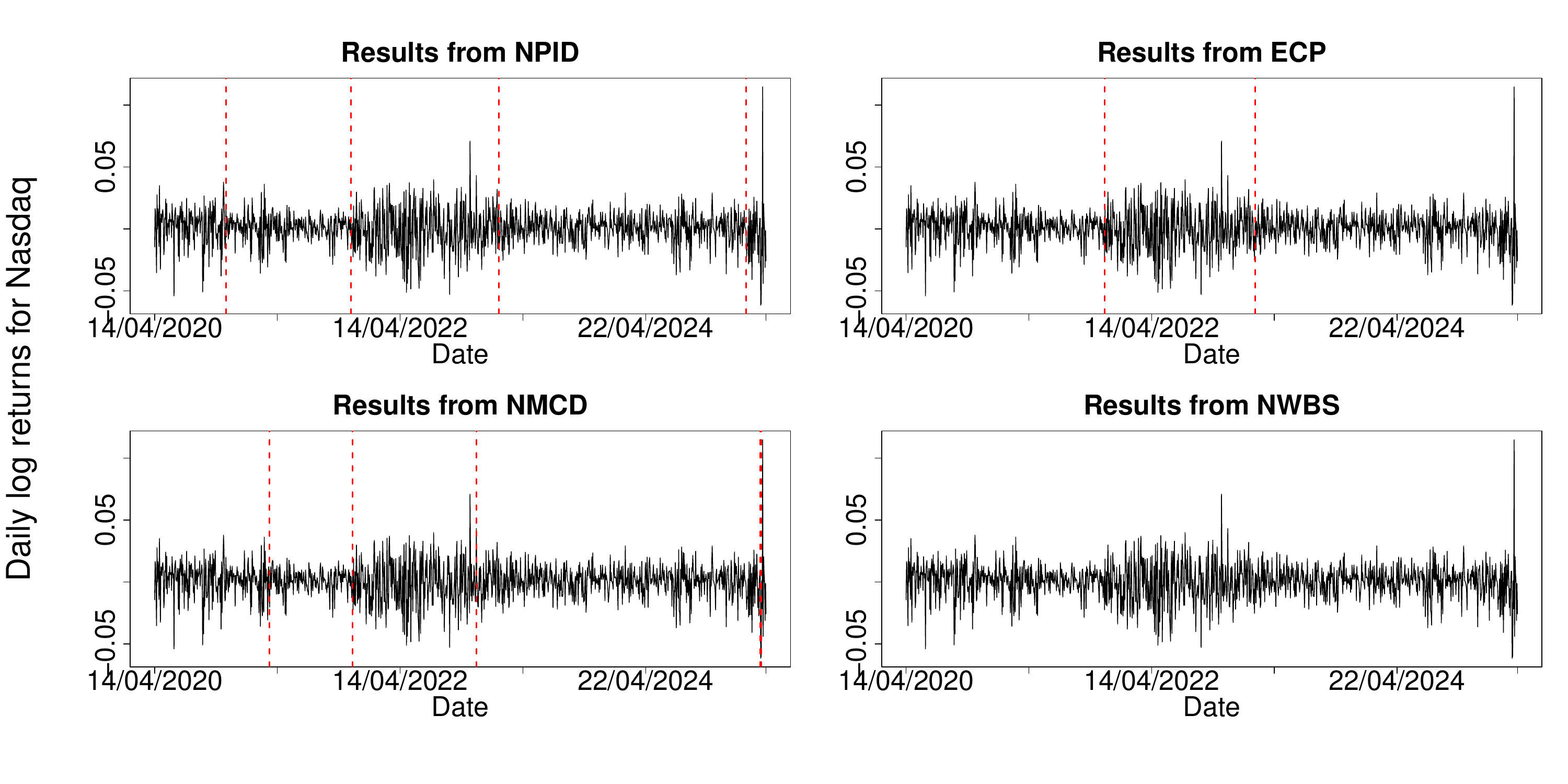}
\vspace{-0.1in}
\caption{Change-point detection for the Nasdaq composite index daily log returns. Top row: The results for the NPID and ECP methods. Bottom row: The results for the NMCD and NWBS methods.}
\label{fig:Nasdaq_all}
\end{figure}
\subsection{The COVID-19 outbreak in the UK}
The performance of NPID is investigated on data from the COVID-19 pandemic; more specifically, we attempt to find changes in the percentage change in COVID-19 seven-day case rates in the UK. The data concern the period from the beginning of September 2021 until the $19^{th}$ of February 2022 and they are available from \url{https://coronavirus.data.gov.uk}. We again look for any changes in the distribution of the data. Figure \ref{fig:Covid-19_percentage} shows the results for the NPID method, as well as for the ECP, NMCD, and NWBS methods.
\begin{figure}
\centering
\includegraphics[width=14cm,height=6cm]{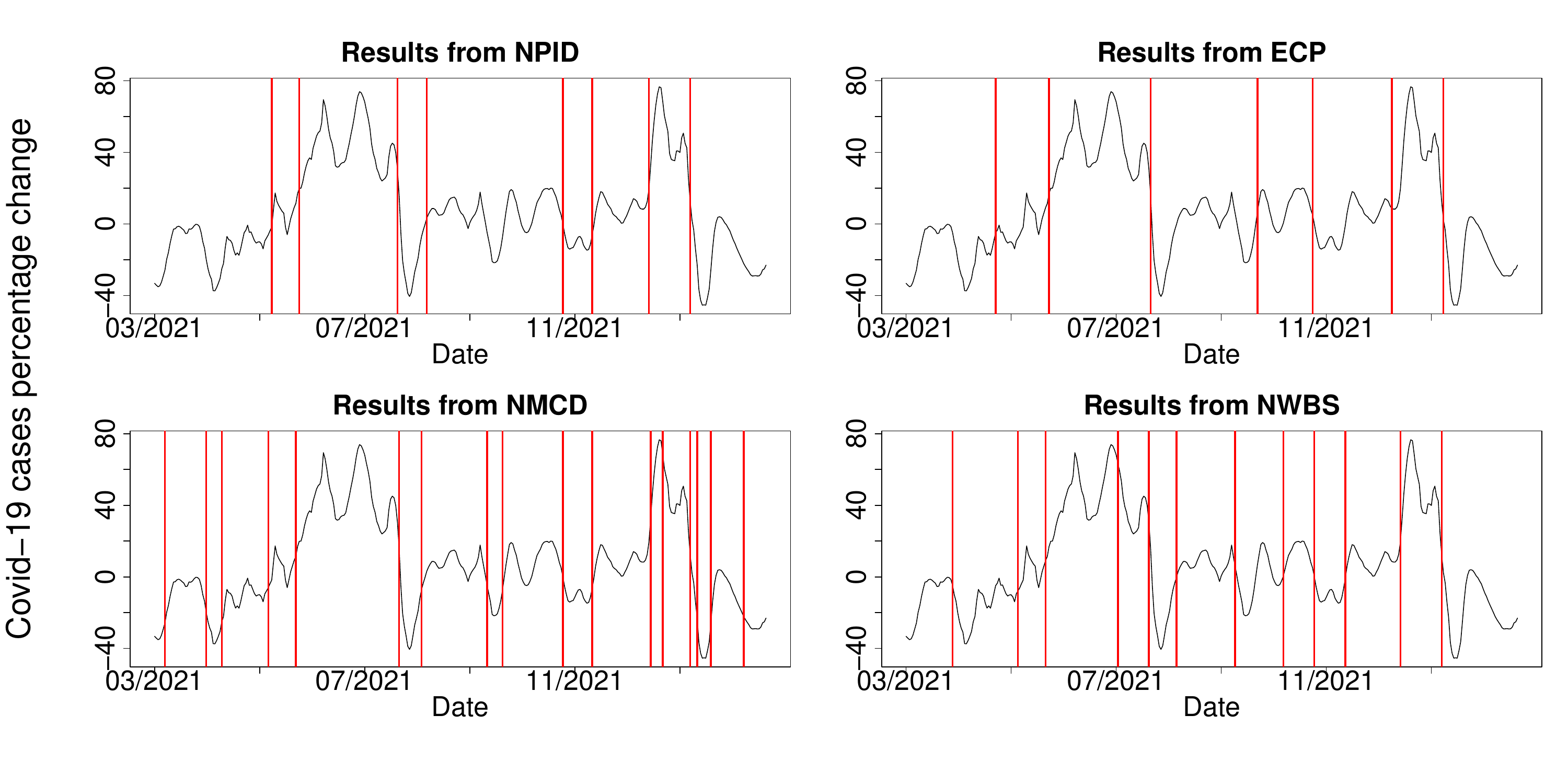}
\vspace{-0.1in}
\caption{Change-point detection for the percentage change in COVID-19 seven-day case rates in the UK. The estimated change-point locations are given with solid vertical lines. Top row: The results for the NPID and ECP methods. Bottom row: The results for the NMCD and NWBS methods.}
\label{fig:Covid-19_percentage}
\end{figure}
The estimated numbers of change-points for NPID, ECP, NMCD, and NWBS are 8, 7, 17, and 12, respectively. While the changes returned by NPID and ECP are relatively easy to justify visually, NWBS, and even more so, NMCD, can be suspected of overestimating the number of change-points here. However, we have to bear in mind that we work under model mis-specification in this example, as the data points are clearly not serially independent.

\section{Discussion}
\label{sec:conclusion}

We highlight that the method can be extended to variables with values in an arbitrary metric space, $\mathcal{E}$, as long as the distributional difference is detectable on a Vapnik-Cervonenkis (VC) class. The way to achieve this is through results appearing in \cite{Dumbgen1991} for the single change-point scenario. More specifically, one can choose a seminorm $N_T(\cdot)$, deterministic or random, on the space $\mathcal{S}$ of all finite signed measures on $\mathcal{E}$ and compute the distributional difference based on this seminorm. Upon the assumptions that first, there is a VC class $\mathcal{D}$ of measurable subsets of $\mathcal{E}$, such that
\begin{equation}
\nonumber N_T(v) \leq \sup\{|v(D)|:D\in\mathcal{D}\}, \forall T \geq 2, \forall v\in \mathcal{S}
\end{equation}
and, second, that the probability of the distributional difference (measured on the seminorm $N_T(\cdot)$) being bounded away from zero goes to 1 as $T \to \infty$, one can show consistency of the estimated number and locations of the change-points, as in Theorem \ref{consistency_theorem}. For an example of a VC class of measurable subsets in $\mathbb{R}^p$ see p.1474 in \cite{Dumbgen1991}.

\begin{appendix}
\section{Models used in the simulation study}
\label{sec:App_A}
The characteristics of the data sequences $X_t$, which were used in the simulation study are given in the list below.
\begin{itemize}
\item[(NC)] {\textit{constant signal}}: length 500 with no change-points.

\item[(M1)] {\textit{mean}}: length 200 with one change-point at 100. The distribution changes from $\mathcal{N}(0,1)$ to $\mathcal{N}(1,1)$.

\item[(V1)] {\textit{variance}}: length 500 with one change-point at 250. The distribution changes from $\mathcal{N}(0,1)$ to $\mathcal{N}(0,4)$.

\item[(D1)] {\textit{distributional}}: length 1000 with one distributional change at 500; there are no changes in the first two moments. In the first segment the distribution is ${\rm Unif}(-3,3)$, and in the second one it is Student-$t_{3}$.

\item[(MM\_Gauss)] {\textit{multi\_mean\_Gaus}}: length 400 with three change-points at 100, 200, 300. The distribution in the four different segments is $\mathcal{N}(0,1)$, $\mathcal{N}(1,1)$, $\mathcal{N}(-0.2,1)$, and $\mathcal{N}(-1.3,1)$.

\item[(MM\_Gauss\_tr)] {\textit{multi\_mean\_Gauss\_tr}}: For this scenario, we transform the data sequences created from (MM\_Gauss) using the exponential function.

\item[(MM\_Student\_$t_3$)] {\textit{multi\_mean\_Student}}: length 400 with three changes in the mean at 100, 200, 300. The values for the piecewise-constant signal in the four segments are $0, 1, -0.2, -1.3$. On this signal, we add noise following the Student-$t_3$ distribution.

\item[(MM\_Gauss2)] {\textit{multi\_mean\_Gauss\_2}}: length 1600 and with 19 change-points at \linebreak $80, 160, \ldots , 1520$. There
are 20 segments. The distribution of the odd numbered segments is $\mathcal{N}(0, 1)$, while the
distribution of the even numbered ones is $\mathcal{N}(2, 1)$.

\item[(MM\_Pois)] {\textit{multi\_mean\_Pois:}} length 400 with three changes in the mean at 100, 200, 300. The values for the piecewise-constant signal in the four segments are $0, 1, -0.2, -1.3$. On this signal, we add noise following the Poisson(1) distribution.

\item[(MM\_Pois\_tr)] {\textit{multi\_mean\_Pois\_tr}}: For this scenario, we transform the data sequences created from (MM\_Pois) using the exponential function.

\item[(MV\_Gauss)] {\textit{multi\_var1:}} length 600 with three change-points at the locations 150, 350, 500. The distribution in the four different segments is $\mathcal{N}(0,1)$, $\mathcal{N}(0,9)$, $\mathcal{N}(0,1.44)$, and $\mathcal{N}(0,0.1)$.

\item[(MV\_Gauss2)] {\textit{multi\_var2:}} length 1000 with five change-points at the locations 200, 350, 550, 700, 900. The distribution in the six different segments is $\mathcal{N}(0,10)$, $\mathcal{N}(0,2)$, $\mathcal{N}(0,0.3)$, $\mathcal{N}(0,4)$, $\mathcal{N}(0,20)$, and $\mathcal{N}(0,2)$.

\item[(MD1)] {\textit{multi\_dis1:}} length 750 with two distributional changes at 250 and 500; there are no changes in the first two moments. In the first segment the distribution is ${\rm Gamma}(1,1)$, in the second one it is Poisson(1), while in the last segment the distribution is ${\rm Unif}(1 - \sqrt{3}, 1 + \sqrt{3})$.

\item[(MD2)] {\textit{multi\_dis2:}} length 500 with three distributional changes at 100, 250, 350. In this example, there are also changes in the first two momemts. More specifically, in the first segment the distribution is ${\rm N}(0,1)$, in the second one it is $\chi^2_1$, in the third one it is Student-$t_3$, while in the last segment the distribution is ${\rm N}(1,1)$.

\item[(MD3)] {\textit{multi\_dis3:}} length 1000 with three distributional changes at 200, 500, 750. In this example, there are also changes in the first two momemts. More specifically, in the first segment the distribution is ${\rm Gamma}(1,1)$, in the second one it is $\chi^2_3$, in the third one it is ${\rm N}(0.5,1)$, while in the last segment the distribution is Student-$t_5$.
\end{itemize}
\section{Brief discussion on the steps we followed for the proof of Theorem \ref{consistency_theorem}}
\label{sec:App_B}
In this section, we provide for a better understanding, an informal explanation of the main steps of the proof of Theorem \ref{consistency_theorem}. A full mathematical proof is given in the supplement. For any $u \in \mathbb{R}$, we denote by $B_{t}(u) := \mathbbm{1}_{\left\lbrace X_t \leq u\right\rbrace}$, and the notation for $\tilde{B}_{s,e}^{b}(u)$ is as in \eqref{contrast_non_par}, while, for $F_{t}(u) := \Prob(X_t \leq u)$,
\begin{equation}
\label{CUSUM_F}
\tilde{F}_{s,e}^{b}(u) = \sqrt{\frac{e-b}{(b-s+1)(e-s+1)}}\sum_{t=s}^{b}F_{t}(u) - \sqrt{\frac{b-s+1}{(e-b)(e-s+1)}}\sum_{t=b + 1}^{e}F_{t}(u).
\end{equation}
In the proof, we derive results for $F_{t}(u)$. However, the consistency proof is concerned with the estimated number and locations of the change-points in the processes $\left\lbrace B_{t}(X_i)\right\rbrace_{\substack{t=1,\ldots,T\\ i=1,\ldots,T}}$, and by extension in the original data sequence $X_1, X_2, \ldots, X_T$. Therefore, in order to be able to deduce consistency related to $X_t$ from our $F_{t}(u)$-reliant proof, we need first to show that for a fixed interval $[s,e)$, where $1\leq s < e \leq T$, and for all $b \in [s,e)$, the observed quantity $\tilde{B}_{s,e}^{b}(u)$ given in \eqref{contrast_non_par} is uniformly close to $\tilde{F}_{s,e}^{b}(u)$, for any $u \in \mathbb{R}$; this is achieved in Lemma 1 provided in the supplementary material. We only cover the case where there is at most one true change-point, namely $r_j$ in the interval $[s,e)$ because our methodology, by construction, prevents the examination of intervals that include more than one change-points. Then, we proceed in the proof of Theorem \ref{consistency_theorem} by showing that as the NPID algorithm proceeds, each change-point will get isolated in an interval where its detection will occur with high probability; for this result we use again Lemma 1 available in the main paper. It suffices to restrict our proof on a single change-point detection framework within an interval $[s_j,e_j)$ which contains only the change-point $r_j$ and no other change-point, and also the maximum CUSUM value for a data point within $[s_j,e_j)$ is greater than the threshold $\zeta_T$.

For each such interval, and for $\hat{r}_j$ being the point with the maximum contrast function value (greater than $\zeta_T$) within $[s_j, e_j)$, we will prove that $(\hat{r}_j - r_j)/\gamma_{j,T}^2 = \mathcal{O}_{{\rm p}}\left(1\right), \; \forall j = 1, \ldots, N$, where $\gamma_{j,T}$ is as in Assumption (A1); to achieve this we employ Lemmas 2 and 3 available in the main paper, as well as results from \cite{Dumbgen1991}.

Because upon detection NPID proceeds from the endpoint of the interval that the detection took place, we also show that with probability one there is no change-point in those bypassed points (between the detection and the new starting point). Furthermore, after each detection, the new starting point is at a place that allows the detection of the next change-point. The last step of the proof is to show that after detecting all the change-points, then NPID, with high probability, will stop since there are no more change-points in the remaining interval $[s,e)$.

\end{appendix}
\bibliographystyle{abbrv}
\bibliography{main_arxiv}  
\clearpage
\begingroup
\renewcommand{\thepage}{S\arabic{page}} 
\setcounter{page}{1}                    
\renewcommand{\thesection}{S\arabic{section}} 
\setcounter{section}{0}
\renewcommand{\thefigure}{S\arabic{figure}}
\renewcommand{\thetable}{S\arabic{table}}
\renewcommand{\theequation}{S\arabic{equation}}

\setcounter{figure}{0}
\setcounter{table}{0}
\setcounter{equation}{0}
%
%

\begin{center}
  \LARGE \textbf{Supplementary Material for ``Non-parametric multiple change-point detection''} \\[1em]
  \normalsize
  Andreas Anastasiou\textsuperscript{1}, Piotr Fryzlewicz\textsuperscript{2} \\
  \textsuperscript{1}Department of Mathematics and Statistics, University of Cyprus\\
  \textsuperscript{2}Department of Statistics, London School of Economics \\
\end{center}
\vspace{2em}

\noindent\textbf{Abstract:} In this supplement, we provide tables related to the Type I error obtained when the default threshold values were used for the detection process; more details are provided in Section 3.2 of the main paper. In addition, we provide the details for the different types of changes used in the simulation study of Section 3.2 that lead to an appropriate choice for the threshold constant. Furthermore, we give the step-by-step proof of Theorem 1, which shows the consistency of our method in accurately estimating the true number and the locations of the change-points.

\vspace{2em}
\section{Tables related to Section 3.2 of the main paper}
\label{supp_App_A}
In Section 3.2 of the main paper, a large-scale simulation study was carried out to decide the best values for the threshold constant. The best behavior occurred when, approximately, $C=0.6$ and $C = 0.9$ for the mean-dominant norms $L_2$ and $L_\infty$, respectively. In an attempt to measure the Type I error obtained from this choice of threshold constants under the scenario of no change-points, we ran 100 replications for twenty different scenarios of no change-points, covering scenarios from both continuous and discrete distributions. The models used are given in Section 3.2 of the main paper. Tables \ref{tab:typeI_cont} and \ref{tab:typeI_disc} below present the frequency distribution of $\hat{N} - N$ for all the above scenarios and for both the $L_{\infty}$ and $L_2$ mean-dominant norms, when their respective default threshold constants are used.
\begin{table}[H]
\centering
\caption{Distribution of $\hat{N} - N$ over 100 simulated data sequences from the Gaussian or Cauchy models.}
{\small{
\begin{tabular}{|l|l|l|l|l|l|}
\cline{1-6}
 & & & \multicolumn{3}{|c|}{} \\
 & & & \multicolumn{3}{|c|}{$\hat{N} - N$} \\
Model & $T$  & Method & 0 & 1 & $\geq 2$\\
\hline
Gaussian & 30 & $L_{\infty}$ & 97 & 3 & 0\\
 &  & $L_{2}$ & 98 & 2 & 0\\
\cline{2-6}
& 75 & $L_{\infty}$ & 98 & 2 & 0\\
 &  & $L_{2}$ & 100 & 0 & 0\\
\cline{2-6}
 & 200 & $L_{\infty}$ & 99 & 1 & 0\\
 &  & $L_{2}$ & 100 & 0 & 0\\
\cline{2-6}
& 500 & $L_{\infty}$ & 95 & 5 & 0\\
 &  & $L_{2}$ & 100 & 0 & 0\\
\hline
Cauchy & 30 & $L_{\infty}$ & 93 & 7 & 0\\
 &  & $L_{2}$ & 99 & 1 & 0\\
\cline{2-6}
 & 75 & $L_{\infty}$ & 97 & 3 & 0\\
 &  & $L_{2}$ & 100 & 0 & 0\\
\cline{2-6}
 & 200 & $L_{\infty}$ & 92 & 5 & 3\\
 &  & $L_{2}$ & 99 & 0 & 1\\
\cline{2-6}
 & 500 & $L_{\infty}$ & 97 & 3 & 0\\
 &  & $L_{2}$ & 100 & 0 & 0\\
\hline
\end{tabular}}}
\label{tab:typeI_cont}
\end{table}
\begin{table}[H]
\centering
\caption{Distribution of $\hat{N} - N$ over 100 simulated data sequences from the Poisson model with different rate values.}
{\small{
\begin{tabular}{|l|l|l|l|l|l|}
\cline{1-6}
 & & & \multicolumn{3}{|c|}{} \\
 & & & \multicolumn{3}{|c|}{$\hat{N} - N$} \\
Model & $T$  & Method & 0 & 1 & $\geq 2$\\
\hline
Poisson(0.3) & 30 & $L_{\infty}$ & 100 & 0 & 0\\
 &  & $L_{2}$ & 88 & 8 & 4\\
\cline{2-6}
 & 75 & $L_{\infty}$ & 99 & 0 & 1\\
 &  & $L_{2}$ & 98 & 1 & 1\\
\cline{2-6}
 & 200 & $L_{\infty}$ & 100 & 0 & 0\\
 &  & $L_{2}$ & 100 & 0 & 0\\
\cline{2-6}
 & 500 & $L_{\infty}$ & 100 & 0 & 0\\
 &  & $L_{2}$ & 100 & 0 & 0\\
\hline
Poisson(3) & 30 & $L_{\infty}$ & 100 & 0 & 0\\
 &  & $L_{2}$ & 100 & 0 & 0\\
\cline{2-6}
 & 75 & $L_{\infty}$ & 99 & 0 & 1\\
 &  & $L_{2}$ & 100 & 0 & 0\\
\cline{2-6}
 & 200 & $L_{\infty}$ & 99 & 1 & 0\\
 &  & $L_{2}$ & 100 & 0 & 0\\
\cline{2-6}
 & 500 & $L_{\infty}$ & 100 & 0 & 0\\
 &  & $L_{2}$ & 100 & 0 & 0\\
\hline
Poisson(30) & 30 & $L_{\infty}$ & 94 & 5 & 1\\
 &  & $L_{2}$ & 99 & 1 & 0\\
\cline{2-6}
 & 75 & $L_{\infty}$ & 97 & 3 & 0\\
 &  & $L_{2}$ & 100 & 0 & 0\\
\cline{2-6}
 & 200 & $L_{\infty}$ & 98 & 2 & 0\\
 &  & $L_{2}$ & 100 & 0 & 0\\
\cline{2-6}
 & 500 & $L_{\infty}$ & 96 & 4 & 0\\
 &  & $L_{2}$ & 100 & 0 & 0\\
\hline
\end{tabular}}}
\label{tab:typeI_disc}
\end{table}
\section{Details regarding the simulation study in Section 3.2}
\label{sup_App_C}
In this section of the supplement, we provide all the details regarding the different types of changes ((M), (V), and (D)) used in the simulation study, as explained in Section 3.2 of the main paper, that lead to an appropriate choice of the threshold constant. If types (M) and (V) were chosen, then the change in the mean or variance, respectively, had magnitude which followed the normal distribution with mean zero and variance $\sigma^2 \in \left\lbrace 1,3,5 \right\rbrace$. If type (D) was chosen then a distribution from a list of distributions (Normal, Student-$t$, Uniform, and Poisson) was chosen. The first two moments remained unchanged in the distributional changes. We always started with the first segment (the part of the data sequence before the first change-point) being from the normal distribution. For each value of $N_\alpha$ and $T$ we generated 1000 replicates and estimated the number of change-points using our NPID method with threshold $\zeta_T$ as in Equation (11) of the main paper for a great variety of constant values $C$.
\section{Proof of Theorem 1}
\label{App_B}
In this section, we provide a thorough proof of the main theorem in our paper. We first need to introduce some further notation, as well as state and proof some important lemmas. As mentioned in Appendix B of the main paper, we will mainly be working within an interval $[s,e)$ with at most one change-point. From now on, we denote by
\begin{equation}
\label{J_se}
J_{s,e} := \left\lbrace \frac{1}{e-s+1}, \frac{2}{e-s+1}, \ldots, \frac{e-s}{e-s+1}\right\rbrace,
\end{equation}
and let $b^* \in J_{s,e}$. We denote by
\begin{align}
\label{notation_empirical}
\nonumber \hat{h}_{b^*}^B(u) & = \frac{1}{b^*(e-s+1)}\sum_{t=1}^{b^*(e-s+1)}\mathbbm{1}_{\left\lbrace X_{t+s-1} \leq u\right\rbrace}\\
\hat{h}_{b^*}^A(u) & = \frac{1}{(1-b^*)(e-s+1)}\sum_{t=b^*(e-s+1) + 1}^{e-s+1}\mathbbm{1}_{\left\lbrace X_{t+s-1} \leq u\right\rbrace}.
\end{align}
The above functions $\hat{h}_{b^*}^B(u)$ and $\hat{h}_{b^*}^A(u)$ are used to calculate the empirical cumulative distribution at $u \in \mathbb{R}$ up to and after the point $b^*$, respectively. In addition, for $r_0 = 0$ and $r_{N+1} = T$, then with $r_j \in [s,e)$ and $r_j^* = \frac{r_j - s +1}{e-s+1}$, we have the notation
\begin{align}
\label{notation_mixture}
\nonumber h_{b^*}^B(u) & := \mathbbm{1}_{\left\lbrace b^* \leq r_j^*\right\rbrace}F_{r_j}(u) + \mathbbm{1}_{\left\lbrace b^* > r_j^*\right\rbrace}\frac{1}{b^*}\left(r_j^*F_{r_j}(u) + (b^*- r_j^*)F_{r_{j+1}}(u)\right)\\
h_{b^*}^A(u) & := \mathbbm{1}_{\left\lbrace b^* \leq r_j^*\right\rbrace}\frac{1}{1-b^*}\left((r_j^* - b^*)F_{r_j}(u) + (1-r_j^*)F_{r_{j+1}}(u)\right) + \mathbbm{1}_{\left\lbrace b^* > r_j^* \right\rbrace}F_{r_{j+1}}(u).
\end{align}
The expressions in \eqref{notation_mixture} are the unknown mixture distributions. Note that if we are under the case of $r_{j-1} < s < e \leq r_j$, then it is straightforward that both $h_{b^*}^B(u)$ and $h_{b^*}^A(u)$ are equal to $F_{r_j}(u)$. We denote by
\begin{equation}
\label{delta_mixture}
\delta_{s,e}^{b^*}(u) := h_{b^*}^A(u) - h_{b^*}^B(u) = \left(\frac{1-r_j^*}{1-b^*}\mathbbm{1}_{\left\lbrace b^*\leq r_j^* \right\rbrace} + \frac{r_j^*}{b^*}\mathbbm{1}_{\left\lbrace b^* > r_j^* \right\rbrace}\right)\Delta_j(u),
\end{equation}
with $\Delta_j(u)$ as in Equation (9) of the main paper. In the same way,
\begin{align}
\label{delta_empirical}
\nonumber & d_{s,e}^{b^*}(u) := \hat{h}_{b^*}^A(u) - \hat{h}_{b^*}^B(u)\\
& =  \frac{1}{(1-b^*)(e-s+1)}\sum_{t=b^*(e-s+1) +1}^{e-s+1}\mathbbm{1}_{\left\lbrace X_{t+s-1} \leq u \right\rbrace} - \frac{1}{b^*(e-s+1)}\sum_{t=1}^{b^*(e-s+1)}\mathbbm{1}_{\left\lbrace X_{t+s-1} \leq u \right\rbrace}.
\end{align}
We notice that if $b^*$ is near 0 or 1, then both \eqref{delta_mixture} and \eqref{delta_empirical} behave badly. Therefore, we introduce weights $w(b^*) = \sqrt{b^*(1-b^*)}$ and for $J_{s,e}$ as in \eqref{J_se}, we consider the measures
\begin{equation}
\label{bigD}
D_{s,e}^{b^*}(u) = w(b^*)d_{s,e}^{b^*}(u), \quad b^* \in J_{s,e}
\end{equation}
which estimate
\begin{equation}
\label{bigDelta}
\Delta_{s,e}^{b^*}(u) = w(b^*)\delta_{s,e}^{b^*}(u) = \rho(b^*)\Delta_j(u),
\end{equation}
where 
\begin{equation}
\label{rho}
\rho(b^*) = \frac{(1-r_j^*)\sqrt{b^*}}{\sqrt{1-b^*}}\mathbbm{1}_{\left\lbrace b^*\leq r_j^* \right\rbrace} + \frac{r_j^*\sqrt{1-b^*}}{\sqrt{b^*}}\mathbbm{1}_{\left\lbrace b^* > r_j^* \right\rbrace}.
\end{equation}
Note that the CUSUM expressions $\tilde{B}_{s,e}^{b}(u)$ and $\tilde{F}_{s,e}^{b}(u)$ in Equations (5) and (16) of the main paper are related to $D_{s,e}^{b^*}(u)$ and $\Delta_{s,e}^{b^*}(u)$, respectively. For $b = b^*(e-s+1)+ s -1$, simple steps yield
\begin{align}
\label{relation_to_CUSUM}
\nonumber D_{s,e}^{b^*}(u) & = \frac{\sqrt{(b-s+1)(e-b)}}{e-s+1}\left(\frac{1}{e-b}\sum_{t=b - s + 2}^{e-s+1}\mathbbm{1}_{\left\lbrace X_{t+s-1} \leq u \right\rbrace} - \frac{1}{b-s+1}\sum_{t=1}^{b - s + 1}\mathbbm{1}_{\left\lbrace X_{t+s-1} \leq u \right\rbrace}\right)\\
& = -\frac{1}{\sqrt{e-s+1}}\tilde{B}_{s,e}^{b}(u).
\end{align}
Similar steps lead as well to $$\Delta_{s,e}^{b^*}(u) = -\frac{1}{\sqrt{e-s+1}}\tilde{F}_{s,e}^{b}(u).$$
From now on, we also denote by
\begin{align}
\label{notation_for_error}
\nonumber & H_{b^*}^B(u) = \left| h_{b^*}^B(u) - \hat{h}_{b^*}^B(u) \right|\\
\nonumber & H_{b^*}^A(u) = \left| h_{b^*}^A(u) - \hat{h}_{b^*}^A(u) \right|\\
& e_{b^*}(u) = H_{b^*}^B(u) + H_{b^*}^A(u)
\end{align}
A series of lemmas now follows, with results that will be helpful in the proof of Theorem 1. For any data sequence $\left\lbrace X_t\right\rbrace_{t=1,2,\ldots,T}$, allow us first to denote by
\begin{align}
\label{notation_vector}
\nonumber & \boldsymbol{\tilde{B}_{s,e}^b} := \left(\tilde{B}_{s,e}^b(X_1), \ldots, \tilde{B}_{s,e}^b(X_T)\right)\\
& \boldsymbol{\tilde{F}_{s,e}^b} := \left(\tilde{F}_{s,e}^b(X_1), \ldots, \tilde{F}_{s,e}^b(X_T)\right)\\
\nonumber & \boldsymbol{D_{s,e}^b} := \left(D_{s,e}^b(X_1), \ldots, D_{s,e}^b(X_T)\right)\\
& \boldsymbol{\Delta_{s,e}^b} := \left(\Delta_{s,e}^b(X_1), \ldots, \Delta_{s,e}^b(X_T)\right),
\end{align}
where $1 \leq s \leq b < e \leq T$. Lemma \ref{lemma_distance_P_B} below shows that the quantity $L\left(\left|\boldsymbol{\tilde{B}_{s,e}^{b}}\right|\right)$ is uniformly close to $L\left(\left|\boldsymbol{\tilde{F}_{s,e}^{b}}\right|\right)$. 
\begin{lemma}
\label{lemma_distance_P_B}
For any data sequence $\left\lbrace X_t\right\rbrace_{t=1,2,\ldots,T}$ and for any interval $[s,e)$ which includes at most one change-point (this is denoted by $r_j$), we have that both $\Prob(A_T) \geq 1- \frac{12}{T}$ and $\Prob(A_T^*) \geq 1- \frac{12}{T}$, where, for any mean-dominant norm $L(\cdot)$,
\begin{align}
\label{A_T}
\nonumber & A_T = \left\lbrace \max_{b: s \leq b < e}\left|L\left(\left|\boldsymbol{\tilde{B}_{s,e}^{b}}\right|\right) - L\left(\left|\boldsymbol{\tilde{F}_{s,e}^{b}}\right|\right)\right|\leq 4\sqrt{\log T} \right\rbrace\\
& A_T^* = \left\lbrace \max_{b: s \leq b < e}\left\lbrace L\left(\left|\left|\boldsymbol{\tilde{B}_{s,e}^{b}}\right| - \left|\boldsymbol{\tilde{F}_{s,e}^{b}}\right|\right|\right)\right\rbrace \leq 4\sqrt{\log T} \right\rbrace.
\end{align}
The notations $\boldsymbol{\tilde{B}_{s,e}^{b}}$ and $\boldsymbol{\tilde{F}_{s,e}^{b}}$ are as in \eqref{notation_vector}.
\end{lemma}
\begin{proof}
For any mean-dominant norm $L(\cdot)$, using the mean dominance property as given in p.190 of \cite{Carlstein1988_supp}, it holds that
\begin{equation*}
0 \leq L_1(\boldsymbol{x}) \leq L(\boldsymbol{x}) \leq L_{\infty}(\boldsymbol{x}), \quad \forall \boldsymbol{x}\in (\mathbb{R}^d)^{+},
\end{equation*}
where $d \in \mathbb{Z}^{+}$. Furthermore,
\begin{equation}
\nonumber \left|L\left(\left|\boldsymbol{\tilde{B}_{s,e}^{b}}\right|\right) - L\left(\left|\boldsymbol{\tilde{F}_{s,e}^{b}}\right|\right)\right| \leq L\left(\left|\left|\boldsymbol{\tilde{B}_{s,e}^{b}}\right|-\left|\boldsymbol{\tilde{F}_{s,e}^{b}}\right|\right|\right).
\end{equation}
Therefore, it is straightforward to see that in order to proof Lemma \ref{lemma_distance_P_B}, it suffices to show just that $\Prob(A_T^*) \geq 1- \frac{12}{T}$, with $A_T^*$ as in \eqref{A_T}. In order to show this result, we use the notation in \eqref{notation_for_error} and after simple steps we can show that for any $u \in \mathbb{R}$
\begin{align}
\nonumber \left|d_{s,e}^{b^*}(u)\right| & \leq \left|\hat{h}_{b^*}^A(u) - h_{b^*}^A(u)\right| + \left|h_{b^*}^A(u) - h_{b^*}^B(u)\right| + \left|h_{b^*}^B(u) - \hat{h}_{b^*}^B(u)\right|\\
\nonumber & = \left|\delta_{s,e}^{b^*}(u)\right| + e_{b^*}(u).
\end{align}
Employing now the notation in \eqref{notation_vector}, it is easy to see that
\begin{align}
\label{useful_bound}
\nonumber & L\left(\left|\left|\boldsymbol{D_{s,e}^{b^*}}\right| - \left|\boldsymbol{\Delta_{s,e}^{b^*}}\right|\right|\right) \leq \sup_{u \in \mathbb{R}}\left|\left|D_{s,e}^{b^*}(u)\right| - \left|\Delta_{s,e}^{b^*}(u)\right|\right| \leq \sqrt{b^*(1-b^*)}\sup_{u \in \mathbb{R}}e_{b^*}(u)\\
& \leq \sqrt{b^*(1-b^*)}\left(\sup_{u \in \mathbb{R}}H_{b^*}^B(u) + \sup_{u \in \mathbb{R}}H_{b^*}^A(u)\right).
\end{align}
The proof will be split into two cases depending on whether there exists one true change-point or no change-points in $[s,e)$.\\
\textbf{Case 1:} $r_{j-1} < s \leq r_j < e \leq r_{j+1}$. By the definition of $H_{b^*}^B(u)$ in \eqref{notation_for_error} and for $F_{t}(u)$ as in Equation (9) of the main paper, we have that
\begin{align}
\nonumber & H_{b^*}^B(u) = \left|\sum_{t=1}^{b^*(e-s+1)}\frac{\mathbbm{1}_{\left\lbrace X_{t+s-1}\leq u\right\rbrace}}{b^*(e-s+1)} - \mathbbm{1}_{\left\lbrace b\leq r_j \right\rbrace}F_{r_j}(u) - \mathbbm{1}_{\left\lbrace b > r_j\right\rbrace}\frac{\left(r_j^*F_{r_j}(u) + (b^*-r_j^*)F_{r_{j+1}}(u)\right)}{b^*}\right|
\end{align}
Therefore,
\begin{align}
\label{bound_for_BH1}
\sup_{u \in \mathbb{R}} H_{b^*}^B(u) & \leq \mathbbm{1}_{\left\lbrace b \leq r_j \right\rbrace}\sup_{u\in \mathbb{R}}\left|\sum_{t=1}^{b^*(e-s+1)}\frac{\mathbbm{1}_{\left\lbrace X_{t+s-1}\leq u\right\rbrace}}{b^*(e-s+1)} - F_{r_j}(u)\right|\\
\nonumber & \qquad + \mathbbm{1}_{\left\lbrace b > r_j \right\rbrace}\left(\sup_{u \in \mathbb{R}}\left|\sum_{t=1}^{r_j^*(e-s+1)}\frac{\mathbbm{1}_{\left\lbrace X_{t+s-1}\leq u\right\rbrace}}{b^*(e-s+1)} - \frac{r_j^*}{b^*}F_{r_j}(u)\right|\right.\\
\label{bound_for_BH2}
& \qquad\qquad\qquad\qquad  \left. + \sup_{u \in \mathbb{R}}\left|\sum_{t=r_j^*(e-s+1)+1}^{b^*(e-s+1)}\frac{\mathbbm{1}_{\left\lbrace X_{t+s-1}\leq u\right\rbrace}}{b^*(e-s+1)} - \frac{b^*-r_j^*}{b^*}F_{r_{j+1}}(u)\right|\right)
\end{align}
In the same way, for $H_{b^*}^A(u)$ as in \eqref{notation_for_error}, we have that
\begin{align}
\nonumber \sup_{u \in \mathbb{R}}H_{b^*}^A(u) & \leq \mathbbm{1}_{\left\lbrace b \leq r_j\right\rbrace}\left(\sup_{u \in \mathbb{R}}\left|\sum_{t=b^{*}(e-s+1)+1}^{r_j^*(e-s+1)}\frac{\mathbbm{1}_{\left\lbrace X_{t+s-1} \leq u \right\rbrace}}{(e-s+1)(1-b^{*})} - \frac{r_j^* - b^{*}}{1-b^{*}}F_{r_j}(u)\right|\right.\\
\label{bound_for_AH1}
& \qquad\qquad\qquad \left. + \sup_{u \in \mathbb{R}}\left|\sum_{t=r_j^*(e-s+1)+1}^{e-s+1}\frac{\mathbbm{1}_{\left\lbrace X_{t+s-1} \leq u \right\rbrace}}{(e-s+1)(1-b^{*})} - \frac{1-r_j^*}{1-b^{*}}F_{r_{j+1}}(u)\right|\right)\\
\label{bound_for_AH2}
& \qquad + \mathbbm{1}_{\left\lbrace b > r_j\right\rbrace}\sup_{u \in \mathbb{R}}\left|\sum_{t=b^{*}(e-s+1)+1}^{e-s+1}\frac{\mathbbm{1}_{\left\lbrace X_{t+s-1} \leq u \right\rbrace}}{(e-s+1)(1-b^{*})} - F_{r_{j+1}}(u)\right|.
\end{align}
Using the results in  \eqref{useful_bound}, \eqref{bound_for_BH1}, \eqref{bound_for_BH2}, \eqref{bound_for_AH1}, and \eqref{bound_for_AH2}, yields
\begin{equation}
\nonumber L\left(\left|\left|\boldsymbol{D_{s,e}^{b^*}}\right| - \left|\boldsymbol{\Delta_{s,e}^{b^*}}\right|\right|\right) \leq \sqrt{b^*(1-b^*)}(\eqref{bound_for_BH1} + \eqref{bound_for_BH2} + \eqref{bound_for_AH1} + \eqref{bound_for_AH2}).
\end{equation} 
This means that for $\epsilon_1 > 0$ and with $J_{s,e}$ as in \eqref{J_se},
\begin{align}
\nonumber & \Prob\left(\max_{b^* \in J_{s,e}}L\left(\left|\left|\boldsymbol{D_{s,e}^{b^*}}\right| - \left|\boldsymbol{\Delta_{s,e}^{b^*}}\right|\right|\right) > \epsilon_1 \right)\\
\nonumber & \leq \Prob\left(\max_{b^*\in J_{s,e}}\sqrt{b^*(1-b^*)}(\eqref{bound_for_BH1} + \eqref{bound_for_BH2} + \eqref{bound_for_AH1} + \eqref{bound_for_AH2}) > \epsilon_1 \right)\\
\label{boundB}
& \leq \Prob\left(\max_{b^* \in J_{s,e}}\sqrt{b^*(1-b^*)}(\eqref{bound_for_BH1} + \eqref{bound_for_BH2} > \frac{\epsilon_1}{2} \right)\\
\label{boundA}
& \qquad + \Prob\left(\max_{b^*\in J_{s,e}}\sqrt{b^*(1-b^*)}(\eqref{bound_for_AH1} + \eqref{bound_for_AH2} > \frac{\epsilon_1}{2} \right) 
\end{align}
We will now show how to bound the probability in \eqref{boundB}, and in the same way one can work to bound the probability in \eqref{boundA}. For $J_{s,e}$ as in \eqref{J_se} and using the Bonferroni inequality we obtain that
\begin{align}
\label{BP1_midpoint}
\nonumber & \Prob\left(\max_{b^* \in J_{s,e}}\sqrt{b^*(1-b^*)}(\eqref{bound_for_BH1} + \eqref{bound_for_BH2} > \frac{\epsilon_1}{2}\right)\\
\nonumber & \leq \sum_{b^* \in J_{s,e}}\Prob\left(\sqrt{b^*(1-b^*)}(\eqref{bound_for_BH1} + \eqref{bound_for_BH2} > \frac{\epsilon_1}{2}\right)\\
\nonumber & = \sum_{\substack{b^*\in J_{s,e}\\
                  b^* \leq r_j^*}}\Prob\left(\sqrt{b^*(1-b^*)}(\eqref{bound_for_BH1} + \eqref{bound_for_BH2} > \frac{\epsilon_1}{2}\right) + \sum_{\substack{b^*\in J_{s,e}\\
                  b^* > r_j^*}}\Prob\left(\sqrt{b^*(1-b^*)}(\eqref{bound_for_BH1} + \eqref{bound_for_BH2} > \frac{\epsilon_1}{2}\right)\\
\nonumber & \leq \sum_{\substack{b^*\in J_{s,e}\\
                  b^* \leq r_j^*}}\Prob\left(\sqrt{b^*(1-b^*)}\sup_{u\in \mathbb{R}}\left|\sum_{t=1}^{b^*(e-s+1)}\frac{\mathbbm{1}_{\left\lbrace X_{t+s-1}\leq u\right\rbrace}}{b^*(e-s+1)} - F_{r_j}(u)\right|> \frac{\epsilon_1}{2}\right)\\
\nonumber & \quad + \sum_{\substack{b^*\in J_{s,e}\\
                  b^* > r_j^*}}\left\lbrace\Prob\left(\sqrt{b^*(1-b^*)}\sup_{u \in \mathbb{R}}\left|\sum_{t=1}^{r_j^*(e-s+1)}\frac{\mathbbm{1}_{\left\lbrace X_{t+s-1}\leq u\right\rbrace}}{b^*(e-s+1)} - \frac{r_j^*}{b^*}F_{r_j}(u)\right| > \frac{\epsilon_1}{4}\right)\right.\\
& \qquad\qquad\left. + \Prob\left(\sqrt{b^*(1-b^*)}\sup_{u \in \mathbb{R}}\left|\sum_{t=r_j^*(e-s+1)+1}^{b^*(e-s+1)}\frac{\mathbbm{1}_{\left\lbrace X_{t+s-1}\leq u\right\rbrace}}{b^*(e-s+1)} - \frac{b^*-r_j^*}{b^*}F_{r_{j+1}}(u)\right| > \frac{\epsilon_1}{4}\right)\right\rbrace.
\end{align}
To bound the above probabilities we will employ the Dvoretzky-Kiefer-Wolfowitz inequality as expressed in \cite{DKW_supp}. Since $b^* \in (0,1)$, we have that for $b^* \leq r_j^*$,
\begin{align}
\label{BP1}
\nonumber & \Prob\left(\sqrt{b^*(1-b^*)}\sup_{u\in \mathbb{R}}\left|\sum_{t=1}^{b^*(e-s+1)}\frac{\mathbbm{1}_{\left\lbrace X_{t+s-1}\leq u\right\rbrace}}{b^*(e-s+1)} - F_{r_j}(u)\right|> \frac{\epsilon_1}{2}\right)\\
& \leq \Prob\left(\sup_{u\in \mathbb{R}}\left|\sum_{t=1}^{b^*(e-s+1)}\frac{\mathbbm{1}_{\left\lbrace X_{t+s-1}\leq u\right\rbrace}}{b^*(e-s+1)} - F_{r_j}(u)\right|> \frac{\epsilon_1}{2\sqrt{b^*}}\right) \leq 2\exp\left\lbrace -\frac{e-s+1}{2}\epsilon_1^2\right\rbrace.
\end{align}
In the same way, for $b^* > r_j^*$,
\begin{equation}
\label{BP2}
\Prob\left(\sqrt{b^{*}(1-b^{*})}\sup_{u \in \mathbb{R}}\left|\sum_{t=1}^{r_j^*(e-s+1)}\frac{\mathbbm{1}_{\left\lbrace X_{t+s-1}\leq u\right\rbrace}}{b^{*}(e-s+1)} - \frac{r_j^*}{b^{*}}F_{r_j}(u)\right| > \frac{\epsilon_1}{4}\right) \leq 2\exp\left\lbrace -\frac{e-s+1}{8}\epsilon_1^2 \right\rbrace.
\end{equation}
and
\begin{align}
\label{BP3}
\nonumber & \Prob\left(\sqrt{b^{*}(1-b^{*})}\sup_{u \in \mathbb{R}}\left|\sum_{t=r_j^*(e-s+1)+1}^{b^{*}(e-s+1)}\frac{\mathbbm{1}_{\left\lbrace X_{t+s-1}\leq u\right\rbrace}}{b^{*}(e-s+1)} - \frac{b^{*}-r_j^*}{b^{*}}F_{r_{j+1}}(u)\right| > \frac{\epsilon_1}{4}\right)\\
& \leq 2\exp\left\lbrace -\frac{e-s+1}{8}\epsilon_1^2 \right\rbrace.
\end{align}
The results in \eqref{BP1_midpoint}, \eqref{BP1}, \eqref{BP2}, and \eqref{BP3}, lead to
\begin{equation}
\nonumber \eqref{boundB} \leq 2(e-s+1)\left(\exp\left\lbrace -\frac{e-s+1}{2}\epsilon_1^2 \right\rbrace + 2\exp\left\lbrace -\frac{e-s+1}{8}\epsilon_1^2\right\rbrace\right).
\end{equation}
The same bound holds for \eqref{boundA}, which means that
\begin{align}
\label{finalboundstep1_1}
\nonumber & \Prob\left(\max_{b^{*} \in J_{s,e}}L\left(\left|\left|\boldsymbol{D_{s,e}^{b^*}}\right| - \left|\boldsymbol{\Delta_{s,e}^{b^*}}\right|\right|\right) > \epsilon_1 \right)\\
& \leq 4(e-s+1)\left(\exp\left\lbrace -\frac{e-s+1}{2}\epsilon_1^2 \right\rbrace + 2\exp\left\lbrace -\frac{e-s+1}{8}\epsilon_1^2 \right\rbrace\right)
\end{align}
{\raggedright{\textbf{Case 2:} $r_{j-1} < s < e \leq r_{j}$.}} Following the same process as in Case 1 and using the notation in \eqref{notation_for_error}, in this scenario we have that
\begin{align}
\label{boundB_case2}
& \sup_{u \in \mathbb{R}}H_{b^*}^B(u) = \sup_{u \in \mathbb{R}}\left|\sum_{t=1}^{b^*(e-s+1)}\frac{\mathbbm{1}_{\left\lbrace X_{t+s-1}\leq u \right\rbrace}}{b^{*}(e-s+1)} - F_{r_j}(u)\right|,\\
\label{boundA_case2}
& \sup_{u \in \mathbb{R}}H_{b^*}^A(u) = \sup_{u \in \mathbb{R}}\left|\sum_{t=b^{*}(e-s+1)+1}^{e-s+1}\frac{\mathbbm{1}_{\left\lbrace X_{t+s-1}\leq u \right\rbrace}}{(1-b^{*})(e-s+1)} - F_{r_j}(u)\right|.
\end{align}
Using again the Dvoretzky-Kiefer-Wolfowitz inequality, then simple calculations yield
\begin{align}
\nonumber & \Prob\left(\max_{b^{*} \in J_{s,e}}L\left(\left|\left|\boldsymbol{D_{s,e}^{b^*}}\right| - \left|\boldsymbol{\Delta_{s,e}^{b^*}}\right|\right|\right) > \epsilon_1 \right)\\
\nonumber & \leq \Prob\left(\max_{b^{*} \in J_{s,e}}\sup_{u \in \mathbb{R}}\left|\left|D_{s,e}^{b^{*}}(u)\right| - \left|\Delta_{s,e}^{b^{*}}(u)\right|\right| > \epsilon_1 \right)\\
\nonumber & \leq \Prob\left(\max_{b^{*} \in J_{s,e}}\sqrt{b^{*}(1-b^{*})}(\eqref{boundB_case2} + \eqref{boundA_case2} > \epsilon_1 \right)\\
\nonumber & \leq \Prob\left(\max_{b^{*} \in J_{s,e}}\sqrt{b^{*}(1-b^{*})}(\eqref{boundB_case2} > \frac{\epsilon_1}{2} \right) + \Prob\left(\max_{b^* \in J_{s,e}}\sqrt{b^*(1-b^*)}(\eqref{boundA_case2} > \frac{\epsilon_1}{2} \right)\\
\nonumber & \leq 4\exp\left\lbrace -\frac{e-s+1}{2}\epsilon_1^2 \right\rbrace,
\end{align}
which is less than the upper bound for Case 1 given in \eqref{finalboundstep1_1}. Therefore using the result in \eqref{finalboundstep1_1}, we have that for $\epsilon_1^* = \epsilon_1\sqrt{e-s+1}$, \eqref{relation_to_CUSUM} leads to
\begin{align}
\label{newstep1}
\nonumber & \Prob\left(\max_{b^* \in J_{s,e}}L\left(\left|\left|\boldsymbol{D_{s,e}^{b^*}}\right| - \left|\boldsymbol{\Delta_{s,e}^{b^*}}\right|\right|\right) > \epsilon_1 \right) = \Prob\left(\max_{b: s \leq b < e}L\left(\left|\left|\boldsymbol{\tilde{B}_{s,e}^{b}}\right| - \left|\boldsymbol{\tilde{F}_{s,e}^{b}}\right|\right|\right) > \epsilon_1^*\right)\\
\nonumber & \leq 4(e-s+1)\left(\exp\left\lbrace -\frac{(\epsilon_1^*)^2}{2}\right\rbrace + 2\exp\left\lbrace -\frac{(\epsilon_1^*)^2}{8}\right\rbrace\right)\\
& \leq 4T\left(\exp\left\lbrace -\frac{(\epsilon_1^*)^2}{2}\right\rbrace + 2\exp\left\lbrace -\frac{(\epsilon_1^*)^2}{8}\right\rbrace\right),
\end{align}
Through the result in \eqref{newstep1}, and since $\forall \epsilon > 0$,
\begin{equation}
\nonumber \Prob\left(\max_{b: s \leq b < e}\left|L\left(\left|\boldsymbol{\tilde{B}_{s,e}^{b}}\right|\right) - L\left(\left|\boldsymbol{\tilde{F}_{s,e}^{b}}\right|\right)\right| > \epsilon\right) \leq \Prob\left(\max_{b: s \leq b < e}L\left(\left|\left|\boldsymbol{\tilde{B}_{s,e}^{b}}\right| - \left|\boldsymbol{\tilde{F}_{s,e}^{b}}\right|\right|\right) > \epsilon\right)
\end{equation}
then it is straightforward that for $A_T$ and $A_T^*$ as in \eqref{A_T},
\begin{equation}
\nonumber \Prob(A_T^{c}) \leq \Prob\left((A_T^*)^{c}\right) \leq 4T\left(\exp\left\lbrace -8\log T\right\rbrace + 2\exp\left\lbrace -2\log T\right\rbrace\right) =  \frac{4}{T}\left(2 + \frac{1}{T^6}\right)\leq \frac{12}{T},
\end{equation}
which completes the proof.
\end{proof}
\begin{lemma}
\label{lemma_orderB_P}
For any interval $[s,e)$ that has only one true change-point, namely $r_j$, and for any mean-dominant norm, $L(\cdot)$, we have that
\begin{equation}
\label{probability_order_CUSUM}
L\left(\left|\boldsymbol{B_{s,e}^{r_j}} - \boldsymbol{F_{s,e}^{r_j}}\right|\right) = \mathcal{O}_{{\rm p}}(1).
\end{equation}
\end{lemma}
\begin{proof}
Using the definition of $\tilde{B}_{s,e}^{r_j}(u)$ and $\tilde{F}_{s,e}^{r_j}(u)$, we have for $l = e - s + 1$ that
\begin{align}
\nonumber & L\left(\left|\boldsymbol{B_{s,e}^{r_j}} - \boldsymbol{F_{s,e}^{r_j}}\right|\right) \leq \sup_{u \in \mathbb{R}}\left|\tilde{B}_{s,e}^{r_j}(u) - \tilde{F}_{s,e}^{r_j}(u)\right|\\
\nonumber & = \sup_{u \in \mathbb{R}}\left|\sqrt{\frac{e - r_j}{l(r_j-s + 1)}}\left(\sum_{t=s}^{r_j}\mathbbm{1}_{\left\lbrace X_t \leq u\right\rbrace} - (r_j-s+1)F_{r_j}(u)\right)\right.\\
\nonumber & \left.\qquad\qquad\qquad - \sqrt{\frac{r_j - s + 1}{l(e - r_j)}}\left(\sum_{t=r_j+1}^{e}\mathbbm{1}_{\left\lbrace X_t \leq u\right\rbrace} - (e - r_j)F_{r_{j+1}}(u)\right)\right|,
\end{align}
and therefore for any $\epsilon > 0$,
\begin{align}
\nonumber & \Prob\left(L\left(\left|\boldsymbol{B_{s,e}^{r_j}} - \boldsymbol{F_{s,e}^{r_j}}\right|\right) > \epsilon\right)\\
\nonumber & \leq \Prob\left(\sup_{u \in \mathbb{R}}\left|\sqrt{\frac{e - r_j}{l(r_j-s + 1)}}\left(\sum_{t=s}^{r_j}\mathbbm{1}_{\left\lbrace X_t \leq u\right\rbrace} - (r_j-s+1)F_{r_j}(u)\right)\right| > \frac{\epsilon}{2}\right)\\
\nonumber & \qquad + \Prob\left(\sup_{u \in \mathbb{R}}\left|\sqrt{\frac{r_j - s + 1}{l(e - r_j)}}\left(\sum_{t=r_j+1}^{e}\mathbbm{1}_{\left\lbrace X_t \leq u\right\rbrace} - (e - r_j)F_{r_{j+1}}(u)\right)\right| > \frac{\epsilon}{2}\right)\\
\nonumber & = \Prob\left(\sup_{u \in \mathbb{R}}\left|\frac{1}{r_j-s+1}\sum_{t=s}^{r_j}\mathbbm{1}_{\left\lbrace X_t \leq u\right\rbrace} - F_{r_j}(u)\right| > \frac{\epsilon\sqrt{l}}{2\sqrt{(e - r_j)(r_j - s +1)}}\right)\\
\nonumber & \qquad + \Prob\left(\sup_{u \in \mathbb{R}}\left|\frac{1}{e - r_j}\sum_{t=r_j+1}^{e}\mathbbm{1}_{\left\lbrace X_t \leq u\right\rbrace} - F_{r_{j+1}}(u)\right| > \frac{\epsilon\sqrt{l}}{2\sqrt{(e - r_j)(r_j - s +1)}}\right)\\
\nonumber & \leq 2\left({\rm exp}\left\lbrace -\frac{\epsilon^2l}{2(e - r_j)}\right\rbrace + {\rm exp}\left\lbrace -\frac{\epsilon^2l}{2(r_j - s + 1)}\right\rbrace\right)
\end{align}
using the Dvoretzky-Kiefer-Wolfowitz inequality. Therefore, for any positive constant $K$ that does not depend on $T$, we have that
\begin{align}
\nonumber & \Prob\left(L\left(\left|\boldsymbol{B_{s,e}^{r_j}} - \boldsymbol{F_{s,e}^{r_j}}\right|\right) > \frac{K \sqrt{\max\left\lbrace e - r_j, r_j - s +1\right\rbrace}}{\sqrt{l}}\right)\\
\nonumber & \leq 2\left({\rm exp}\left\lbrace -\frac{K^2\max\left\lbrace e - r_j, r_j- s + 1\right\rbrace}{2(e - r_j)}\right\rbrace + {\rm exp}\left\lbrace -\frac{K^2\max \left\lbrace e - r_j, r_j-s+1 \right\rbrace}{2(r_j - s + 1)}\right\rbrace\right)\\
\nonumber & \leq 4\exp\left\lbrace-\frac{K^2}{2}\right\rbrace
\end{align}
which leads to the result in \eqref{probability_order_CUSUM}.
\end{proof}
%
\begin{lemma}
\label{lemma_rho_relation}
For any interval $[s, e)$ that has only one true change-point, namely $r_j$, we have that for any arbitrary $b^* \in J_{s,e}$, with $J_{s,e}$ defined in \eqref{J_se}, there is a constant $C > 0$ such that, for $r_j^* = \frac{r_j - s + 1}{e-s+1}$,
\begin{equation}
\label{rho_ineq}
\rho(r_j^*) - \rho(b^*) \geq C|b^* - r_j^*|,
\end{equation}
where $\rho(b^*)$ is given in \eqref{rho}.
\end{lemma}
\begin{proof}
We will split the proof in two cases depending on the position of $b^*$ with respect to the change-point $r_j^*$.\\
{\textbf{Case 1:}} $b^* \leq r_j^*$. We have that
\begin{align}
\nonumber \rho(r_j^*) - \rho(b^*) & = \sqrt{r_j^*(1-r_j^*)} - \sqrt{\frac{b^*}{1-b^*}}(1-r_j^*) \geq \sqrt{1-r_j^*}\left(\sqrt{r_j^*} - \sqrt{b^*}\right)\\
\nonumber & \geq \frac{\sqrt{1-r_j^*}}{2}\left(r_j^* - b^*\right),
\end{align}
due to the fact that $\sqrt{r_j^*} - \sqrt{b^*} = \frac{r_j^* - b^*}{\sqrt{r_j^*} + \sqrt{b^*}} \geq \frac{r_j^* - b^*}{2}$.\\
{\textbf{Case 2:}} $b^* > r_j^*$. Following a similar approach as in Case 1, we have that
\begin{align}
\nonumber \rho(r_j^*) - \rho(b^*) & \geq \frac{\sqrt{r_j^*}}{2}\left(b^* - r_j^*\right).
\end{align}
Combining the results of the two cases above, we can conclude that in general there exists a positive constant $C$, such that $\rho(r_j^*) - \rho(b^*) \geq C\left|b^* - r_j^*\right|$. 
\end{proof}
{\raggedright{{\textbf{Proof of Theorem 1}}}}\\
The proof uses the results of Lemmas \ref{lemma_distance_P_B} - \ref{lemma_rho_relation} and is based on two steps.\\
{\textbf{Step 1:}} For ease of understanding, we split this step into three smaller parts. From now on, we assume that $A_T^*$ and therefore, also $A_T$, as in Lemma \ref{lemma_distance_P_B} hold. The constants we use for Theorem 1 are
\begin{equation}
\label{constants}
C_1 > 4, \quad C_2 = \frac{1}{\sqrt{6}} - \frac{4}{\underline{C}},
\end{equation}
where $\underline{C}$ is as in assumption (A1).
\vspace{0.05in}
\\
{\textbf{Step 1.1:}} Firstly, $\forall j \in \left\lbrace 1, \ldots, N\right\rbrace$, we define the intervals 
\begin{equation}
\label{isolating_intervals}
I_{j}^R = \left[r_j + \frac{\delta_{j+1}}{3}, r_j + 2\frac{\delta_{j+1}}{3}\right), \qquad I_{j}^L = \left(r_j - 2\frac{\delta_j}{3}, r_j - \frac{\delta_j}{3}\right].
\end{equation}
In order for $I_j^R$ and $I_j^L$ to have at least one point, we actually implicitly require that $\delta_j > 3, \quad \forall j = 1, \ldots, N + 1$, which is the case for sufficiently large $T$; see assumption (A1). Since the lengths of $I_j^R$ and $I_j^L$ as in \eqref{isolating_intervals} are equal to $\delta_{j+1}/3$ and $\delta_j/3$, respectively, and taking $\lambda_T \leq \delta_T/3$, where $\delta_T$ is as in Equation (6) of the main paper the minimum distance between two change-points, then the NPID method ensures that for $K=\left\lceil T/\lambda_T  \right\rceil$ and $k,m \in \left\lbrace 1,\ldots,K\right\rbrace$, there exist at least one $c_k^r = k\lambda_T + 1$ and at least one $c_m^l = T-m\lambda_T$ that are in $I_j^R$ and $I_j^L$, respectively, $\forall j =1,\ldots,N$.

Depending on whether $r_1 \leq T - r_N$ then $r_1$ or $r_N$ will first get isolated in a right- or left-expanding interval, respectively. W.l.o.g., assume that $r_1 \leq T - r_N$. Our aim is to first show that there will be at least an interval of the form $[1,c_{\tilde{k}}^r]$, for $\tilde{k} \in \left\lbrace 1,\ldots,K \right\rbrace$, which contains only $r_1$ and no other change-point, such that
$\max_{1\leq t < c_{\tilde{k}}^r}L\left(\left|\boldsymbol{\tilde{B}_{1,c_{\tilde{k}}^r}^{t}}\right|\right) > \zeta_T$, where for any $1 \leq s \leq b < e \leq T$, $\boldsymbol{B_{s,e}^{b}}$ is as in \eqref{notation_vector}. As already mentioned, our method due to its expansion approach naturally ensures that $\exists k \in \left\lbrace 1,\ldots, K\right\rbrace$ such that $c_k^r \in I_1^R$. There is no other change-point in $[1,c_k^r]$ apart from $r_1$. We will now show that for $b = {\rm argmax}_{1 \leq t \leq c_k^r}L(|\boldsymbol{\tilde{B}_{1, c_k^r}^{t}}|)$, then $L(|\boldsymbol{\tilde{B}_{1, c_k^r}^{b}}|) > \zeta_T$. Using Lemma \ref{lemma_distance_P_B}, we have that
\begin{equation}
\label{thresholdpassing_firststep}
L\left(\left|\boldsymbol{\tilde{B}_{1, c_k^r}^{b}}\right|\right) \geq L\left(\left|\boldsymbol{\tilde{B}_{1, c_k^r}^{r_1}}\right|\right) \geq L\left(\left|\boldsymbol{\tilde{F}_{1, c_k^r}^{r_1}}\right|\right)  - 4\sqrt{\log\;T}
\end{equation}
But,
\begin{align}
\label{midstep1}
\nonumber L\left(\left|\boldsymbol{\tilde{F}_{1,c_k^r}^{r_1}}\right|\right) & = \sqrt{\frac{(c_k^r - r_1)r_1}{c_k^r}}L\left(\left|\boldsymbol{\Delta_1}\right|\right) \geq \sqrt{\frac{(c_k^r - r_1)r_1}{2\max\left\lbrace c_k^r - r_1, r_1 \right\rbrace}}L\left(\left|\boldsymbol{\Delta_1}\right|\right)\\
& = \sqrt{\frac{\min\left\lbrace c_k^r - r_1, r_1 \right\rbrace}{2}}L\left(\left|\boldsymbol{\Delta_1}\right|\right).
\end{align}
From our notation of $r_0 = 0$ and of $\delta_j$ as in Equation (9) of the main paper, we know that $r_1 = \delta_1$. In addition, since $c_k^r \in I_1^R$, then
\begin{equation}
\nonumber \frac{\delta_{2}}{3} \leq c_k^r - r_1 < 2\frac{\delta_{2}}{3},
\end{equation}
meaning that
\begin{equation}
\label{mindistance1}
\min\left\lbrace c_k^r - r_1, r_1 \right\rbrace \geq \frac{1}{3}\min\left\lbrace \delta_1, \delta_2\right\rbrace.
\end{equation}
The result in \eqref{thresholdpassing_firststep}, the assumption (A1) and the application of \eqref{mindistance1} to \eqref{midstep1} yield, for $\underline{m}_T$ as in (A1), with probability going to 1 as $T \to \infty$,
\begin{align}
\label{thresholdpassing}
\nonumber L\left(\left|\boldsymbol{\tilde{B}_{1,c_k^r}^{b}}\right|\right) & \geq \sqrt{\frac{\min\left\lbrace \delta_1,\delta_2\right\rbrace}{6}}L\left(\left|\boldsymbol{\Delta_1}\right|\right) - 4\sqrt{\log T} \geq \sqrt{\frac{\min\left\lbrace \delta_1,\delta_2\right\rbrace}{6}}\frac{\tilde{C}_1}{\gamma_{1,T}} - 4\sqrt{\log T}\\
\nonumber & \geq \frac{\underline{m}_T}{\sqrt{6}} - 4\sqrt{\log T} = \left(\frac{1}{\sqrt{6}} - \frac{4\sqrt{\log T}}{\underline{m}_T}\right)\underline{m}_T\\
& \geq \left(\frac{1}{\sqrt{6}} - \frac{4}{\underline{C}}\right)\underline{m}_T = C_2\underline{m}_T > \zeta_T.
\end{align}
Therefore, there will be an interval of the form $[1,c_{\tilde{k}}^r]$, with $c_{\tilde{k}}^r > r_1$, such that $[1,c_{\tilde{k}}^r]$ contains only $r_1$ and $\max_{1\leq b < c_{\tilde{k}}^r}L\left(\left|\boldsymbol{\tilde{B}_{1,c_{\tilde{k}}^r}^{b}}\right|\right) > \zeta_T$. Let us, for $k^* \in \left\lbrace 1,\ldots, K \right\rbrace$, to denote by $c_{k^*}^r \leq c_{\tilde{k}}^r$ the first right-expanding point where this happens and let $b_1 = {\rm argmax}_{1\leq t < c_{k^*}^r}L\left(\left|\boldsymbol{\tilde{B}_{1,c_{k^*}^r}^{t}}\right|\right)$ with $L\left(\left|\boldsymbol{\tilde{B}_{1,c_{k^*}^r}^{b_1}}\right|\right) > \zeta_T$.
\vspace{0.1in}
\\
{\textbf{Step 1.2:}} We will now show that $|b_1 - r_1|/\gamma_{1,T}^2 = \mathcal{O}_{{\rm p}}\left(1\right)$ through a more general result which holds for any interval $[s,e)$ that has only one true change-point and also, as in Step 1.1, the maximum aggregated (using the mean-dominant norm $L(\cdot)$) contrast function value for a point within $[s,e)$ is greater that the threshold $\zeta_T$. From now on, the only change-point in $[s,e)$ is denoted by $r_j$ and $\hat{r}_j$ is the point that has the maximum CUSUM value in that interval with its value exceeding $\zeta_T$. For ease of presentation, we also denote by $l := e-s+1$. Our aim is to show that for $r_j^* = \frac{r_j - s + 1}{e-s+1}$,
\begin{align}
\label{main_Result_consistency}
\liminf_{T\rightarrow\infty}\Prob\left(\Lambda(d)\right)\rightarrow 1,
\end{align}
where $d \rightarrow \infty$ as $T \rightarrow \infty$. The constant $C_{*} > 0$ is independent of $T$ and
\begin{align}
\label{LambdaTd}
\nonumber \Lambda(d) & := \left\lbrace\vphantom{\sup_{i=1,\ldots,T}\left|D_{s,e}^{b^*}(X_i)\right|} L\left(\left|\boldsymbol{D_{s,e}^{b^*}}\right|\right) - L\left(\left|\boldsymbol{D_{s,e}^{r_j^*}}\right|\right) \leq -C_{*}\left|b^* - r_j^*\right|/\gamma_{j,T},\right.\\
& \left.\qquad\qquad \forall b^* \in J_{s,e}\setminus N(d)\qquad {{\rm and}}\;L\left(\left|\boldsymbol{D_{s,e}^{r_j^*}}\right|\right)\geq C_*/\gamma_{j,T}\vphantom{\sup_{q=1,\ldots,Q}\left|D_{s,e}^{b^*}(l_q)\right|}\right\rbrace.
\end{align}
with
\begin{equation}
\label{Ntd}
N(d):=\left\lbrace b^*\in J_{s,e}:\left|b^*-r_j^*\right|\leq \frac{d\gamma_{j,T}^2}{(e-s+1)}\right\rbrace.
\end{equation}
Proving the result in \eqref{main_Result_consistency} and using that $r_j^* = \frac{r_j-s+1}{e-s+1}$ and $b^* = \frac{b-s+1}{e-s+1}$ for any $b\in [s,e)$, then we can get that $|\hat{r}_j - r_j|/\gamma_{j,T}^2 = \mathcal{O}_{{\rm p}}\left(1\right)$.

For $\boldsymbol{D_{s,e}^{r_j^*}}$ as in \eqref{notation_vector}, and using that $\forall u \in \mathbb{R}$, then $\tilde{F}_{s,e}^{r_j}(u) = -\sqrt{\frac{(e - r_j)(r_j -s+1)}{l}}\Delta_j(u)$, we have, due to (A1), with probability tending to 1 as $T \to \infty$, that
\begin{align}
\nonumber  L\left(\left|\boldsymbol{D_{s,e}^{r_j^*}}\right|\right) & = \frac{1}{\sqrt{l}}L\left(\left|\boldsymbol{\tilde{B}_{s,e}^{r_j}}\right|\right)\\
\nonumber & = \frac{1}{\sqrt{l}}L\left(\left|\boldsymbol{\tilde{B}_{s,e}^{r_j}} - \boldsymbol{\tilde{F}_{s,e}^{r_j}} + \boldsymbol{\tilde{F}_{s,e}^{r_j}}\right|\right)\\
\nonumber & \geq \frac{1}{\sqrt{l}}\left(L\left(\left|\boldsymbol{\tilde{F}_{s,e}^{r_j}}\right|\right) - L\left(\left|\boldsymbol{\tilde{B}_{s,e}^{r_j}} - \boldsymbol{\tilde{F}_{s,e}^{r_j}}\right|\right)\right)\\
\nonumber & \geq \frac{1}{\sqrt{l}}\left(\sqrt{\frac{(e - r_j)(r_j -s+1)}{l}}\frac{\tilde{C}_j}{\gamma_{j,T}} - L\left(\left|\boldsymbol{\tilde{B}_{s,e}^{r_j}} - \boldsymbol{\tilde{F}_{s,e}^{r_j}}\right|\right)\right).
\end{align}
Because $l \rightarrow \infty$ as $T \rightarrow \infty$ means that $\frac{1}{\sqrt{l}} = o(1)$ and we have using \eqref{probability_order_CUSUM} that for $w(r_j^*) = \sqrt{r_j^*(1-r_j^*)}$
\begin{equation}
\label{first_important_result}
L\left(\left|\boldsymbol{D_{s,e}^{r_j^*}}\right|\right) \geq w(r_j^*)\frac{\tilde{C}_j}{\gamma_{j,T}} \geq C^*/\gamma_{j,T}
\end{equation}
with probability tending to 1 for any arbitrary $C^* \in (0,w(r_j^*))$. For $\boldsymbol{M_{s,e}^{b^*}} = \boldsymbol{D_{s,e}^{b^*}} - \boldsymbol{\Delta_{s,e}^{b^*}}$, we have that
\begin{align}
\nonumber & L\left(\left|\boldsymbol{D_{s,e}^{b^*}}\right|\right) = L\left(\left|\boldsymbol{M_{s,e}^{b^*}} - \boldsymbol{M_{s,e}^{r_j^*}} + \boldsymbol{M_{s,e}^{r_j^*}} + \boldsymbol{\Delta_{s,e}^{b^*}}\right|\right)\\
\nonumber & = L\left(\left|\boldsymbol{M_{s,e}^{b^*}} - \boldsymbol{M_{s,e}^{r_j^*}} + \boldsymbol{M_{s,e}^{r_j^*}} + \frac{\rho(b^*)}{w(r_j^*)}\boldsymbol{\Delta_{s,e}^{r_j^*}}\right|\right)\\
\nonumber & = L\left(\left|\boldsymbol{M_{s,e}^{b^*}} - \boldsymbol{M_{s,e}^{r_j^*}} + \boldsymbol{M_{s,e}^{r_j^*}} + \frac{\rho(b^*)}{w(r_j^*)}\left(\boldsymbol{D_{s,e}^{r_j^*}} - \boldsymbol{M_{s,e}^{r_j^*}}\right)\right|\right)\\
\nonumber & \leq L\left(\left|\boldsymbol{M_{s,e}^{b^*}} - \boldsymbol{M_{s,e}^{r_j^*}}\right|\right) + L\left(\left|\frac{\rho(r_j^*) - \rho(b^*)}{w(r_j^*)}\boldsymbol{M_{s,e}^{r_j^*}}\right|\right) + \frac{\rho(b^*)}{w(r_j^*)}L\left(\left|\boldsymbol{D_{s,e}^{r_j^*}}\right|\right).
\end{align}
Therefore,
\begin{align}
\label{diff_mid_step}
\nonumber & L\left(\left|\boldsymbol{D_{s,e}^{b^*}}\right|\right) - L\left(\left|\boldsymbol{D_{s,e}^{r_j^*}}\right|\right)\\
\nonumber & \leq L\left(\left|\boldsymbol{M_{s,e}^{b^*}} - \boldsymbol{M_{s,e}^{r_j^*}}\right|\right) + L\left(\left|\frac{\rho(r_j^*) - \rho(b^*)}{w(r_j^*)}\boldsymbol{M_{s,e}^{r_j^*}}\right|\right) + \frac{\rho(b^*) - \rho(r_j^*)}{w(r_j^*)}L\left(\left|\boldsymbol{D_{s,e}^{r_j^*}}\right|\right)\\
& = L\left(\left|\boldsymbol{M_{s,e}^{b^*}} - \boldsymbol{M_{s,e}^{r_j^*}}\right|\right) + \frac{\rho(r_j^*) - \rho(b^*)}{w(r_j^*)}\left(L\left(\left|\boldsymbol{M_{s,e}^{r_j^*}}\right|\right) -L\left(\left|\boldsymbol{D_{s,e}^{r_j^*}}\right|\right)\right).
\end{align}
However, we know that $$L\left(\left|\boldsymbol{D_{s,e}^{r_j^*}}\right|\right) \geq L\left(\left|\boldsymbol{\Delta_{s,e}^{r_j^*}}\right|\right) - L\left(\left|\boldsymbol{M_{s,e}^{r_j^*}}\right|\right)$$ and continuing from \eqref{diff_mid_step}, we obtain that
\begin{align}
\label{diff_mid_step2}
\nonumber & L\left(\left|\boldsymbol{D_{s,e}^{b^*}}\right|\right) - L\left(\left|\boldsymbol{D_{s,e}^{r_j^*}}\right|\right)\\
\nonumber & \leq L\left(\left|\boldsymbol{M_{s,e}^{b^*}} - \boldsymbol{M_{s,e}^{r_j^*}}\right|\right)\\
\nonumber & \qquad + \frac{\rho(r_j^*) - \rho(b^*)}{w(r_j^*)}\left(2L\left(\left|\boldsymbol{M_{s,e}^{r_j^*}}\right|\right) - L\left(\left|\boldsymbol{\Delta_{s,e}^{r_j^*}}\right|\right)\right)\\
& = L\left(\left|\boldsymbol{M_{s,e}^{b^*}} - \boldsymbol{M_{s,e}^{r_j^*}}\right|\right) - (\rho(r_j^*) - \rho(b^*))\left(\frac{\tilde{C}_j}{\gamma_{j,T}} - \frac{2}{w(r_j^*)}L\left(\left|\boldsymbol{M_{s,e}^{r_j^*}}\right|\right)\right).
\end{align}
Using the result in \eqref{probability_order_CUSUM} from Lemma \ref{lemma_orderB_P}, we have that since $l \rightarrow \infty$ as $T \rightarrow \infty$,
\begin{align}
\nonumber & \frac{1}{w(r_j^*)}L\left(\left|\boldsymbol{M_{s,e}^{r_j^*}}\right|\right) = \frac{1}{w(r_j^*)\sqrt{l}}L\left(\left|\boldsymbol{\tilde{B}_{s,e}^{r_j}} - \boldsymbol{\tilde{F}_{s,e}^{r_j}}\right|\right) = o_{{\rm p}}(1).
\end{align}
Therefore, continuing from \eqref{diff_mid_step2},
\begin{align}
\nonumber & L\left(\left|\boldsymbol{D_{s,e}^{b^*}}\right|\right) - L\left(\left|\boldsymbol{D_{s,e}^{r_j^*}}\right|\right)\\
\nonumber & \leq L\left(\left|\boldsymbol{M_{s,e}^{b^*}} - \boldsymbol{M_{s,e}^{r_j^*}}\right|\right) - (\rho(r_j^*) - \rho(b^*))\left(\frac{\tilde{C}_j}{\gamma_{j,T}} + o_{{\rm p}}(1)\right)
\end{align}
for all $b^* \in J_{s, e}$ with probability tending to 1. Using now the result of Lemma \ref{lemma_rho_relation}, we have that
\begin{align}
\label{first_prob}
\nonumber & \Prob\left(L\left(\left|\boldsymbol{D_{s,e}^{b^*}}\right|\right) - L\left(\left|\boldsymbol{D_{s,e}^{r_j^*}}\right|\right) \leq L\left(\left|\boldsymbol{M_{s,e}^{b^*}} - \boldsymbol{M_{s,e}^{r_j^*}}\right|\right) \right.\\
& \qquad\qquad\qquad\qquad\qquad\qquad\qquad \left. -\frac{\left|b^*-r_j^*\right|}{\gamma_{j,T}}C'_j \;\;{\rm for\;all}\;b^*\right)\rightarrow 1,
\end{align}
where $C'_j = C\tilde{C}_jC_0'$ for any $C_0'\in (0,1)$. For the term $L\left(\left|\boldsymbol{M_{s,e}^{b^*}} - \boldsymbol{M_{s,e}^{r_j^*}}\right|\right)$ in \eqref{first_prob}, we have that
\begin{align}
\nonumber L\left(\left|\boldsymbol{M_{s,e}^{b^*}} - \boldsymbol{M_{s,e}^{r_j^*}}\right|\right) & \leq L\left(\left|\boldsymbol{M_{s,e}^{b^*}}\right|\right) + L\left(\left|\boldsymbol{M_{s,e}^{r_j^*}}\right|\right)\\
\nonumber & \leq 2\max_{j \in J_{s,e}}L\left(\left|\boldsymbol{M_{s,e}^{j}}\right|\right)\\
\nonumber & \leq 2\max_{j \in J_{s,e}}L_{\infty}\left(\left|\boldsymbol{M_{s,e}^{j}}\right|\right)\\
\nonumber & \leq K(\log \log l)^{1/2}(l)^{-1/2}
\end{align}
with the last inequality coming from Lemma 2 in \cite{Dumbgen1991_supp}, where $K > 0$ is independent of the sample size. From Assumption (A1), 
we can deduce that $\gamma_{j,T}\sqrt{\log T}(\min\{\delta_j, \delta_{j+1}\})^{-1/2} = O(1)$. Therefore, $\gamma_{j,T}(\log \log l)^{1/2}(l)^{-1/2} = o(1)$ and therefore, \eqref{first_prob} implies that there is a constant $C'' > 0$ such that
\begin{align}
\label{second_prob}
\nonumber & \Prob\left(\vphantom{\sup_{q = 1, \ldots, Q}\left|D_{s,e}^{b^*,q}\right| - \sup_{q = 1, \ldots, Q}\left|D_{s,e}^{r_j^*,q}\right|}L\left(\left|\boldsymbol{D_{s,e}^{b^*}}\right|\right) - L\left(\left|\boldsymbol{D_{s,e}^{r_j^*}}\right|\right) \leq - C''\left|b^*-r_j^*\right|/\gamma_{j,T}\right.\\
& \qquad\qquad\qquad\qquad\qquad\qquad\qquad\left. {\rm for\;all}\;b^* \in J_{s,e}\setminus[\alpha,1-\alpha]\vphantom{\sup_{q = 1, \ldots, Q}\left|D_{s,e}^{b^*,q}\right| - \sup_{q = 1, \ldots, Q}\left|D_{s,e}^{r_j^*,q}\right|}\right)\rightarrow 1
\end{align}
for arbitrary and fixed $\alpha \in (0,1/2)$. This means that it suffices to consider $L\left(\left|\boldsymbol{M_{s,e}^{b^*}} - \boldsymbol{M_{s,e}^{r_j^*}}\right|\right)$ on compact subintervals of $[0,1]$ and the result in \eqref{main_Result_consistency} follows if we show that for arbitrary constants $C''' > 0$ and $\alpha \in (0,1/2)$ then
\begin{align}
\label{assertion_follows}
\nonumber & \liminf_{T \rightarrow \infty}\Prob\left(\vphantom{\sup_{q=1,\ldots,Q}\left|M_{s,e}^{b^*,q} - M_{s,e}^{r_j^*,q}\right|}L\left(\left|\boldsymbol{M_{s,e}^{b^*}} - \boldsymbol{M_{s,e}^{r_j^*}}\right|\right) \leq C'''\left|b^* - r_j^*\right|/\gamma_{j,T} \right.\\
& \qquad\qquad\qquad\qquad \left. {\rm for\;all}\;b^* \in J_{s,e}\cap[\alpha,1-\alpha]\setminus N_T(d)\vphantom{\sup_{q=1,\ldots,Q}\left|M_{s,e}^{b^*,q} - M_{s,e}^{r_j^*,q}\right|}\right)\rightarrow 1
\end{align}
as $d \rightarrow \infty$. W.l.o.g. we take $b^* \leq r_j^*$, and simple calculations yield
\begin{align}
\nonumber  &M_{s,e}^{b^*}(X_i) - M_{s,e}^{r_j^*}(X_i) = \frac{1}{l}\sum_{t=r_j^* l + 1}^{l}\mathbbm{1}_{\left\lbrace X_{t+s-1}\leq X_i \right\rbrace}\left(\sqrt{\frac{b^*}{1-b^*}} - \sqrt{\frac{r_j^*}{1-r_j^*}}\right)\\
\nonumber & \quad + \frac{1}{l}\sum_{t=1}^{r_j^* l}\mathbbm{1}_{\left\lbrace X_{t+s-1}\leq X_i \right\rbrace}\left(\sqrt{\frac{1-r_j^*}{r_j^*}} - \sqrt{\frac{1-b^*}{b^*}}\right)\\
\nonumber & \quad + \frac{1}{\sqrt{b^*(1-b^*)}l}\sum_{t=b^*l+1}^{r_j^*l}\mathbbm{1}_{\left\lbrace X_{t+s-1}\leq X_i \right\rbrace} + \left(\sqrt{r_j^*(1-r_j^*)} - \frac{\sqrt{b^*}(1-r_j^*)}{\sqrt{1-b^*}}\right)\Delta_j(X_i)\\
\nonumber & = \left(\frac{1}{l}\sum_{t=r_j^*l+1}^{l}\mathbbm{1}_{\left\lbrace X_{t+s-1} \leq X_i \right\rbrace} - (1-r_j^*)F_{r_j+1}(X_i)\right)\left(\sqrt{\frac{b^*}{1-b^*}} - \sqrt{\frac{r_j^*}{1-r_j^*}}\right)\\
\nonumber & \quad + \left(\frac{1}{l}\sum_{t=1}^{r_j^*l}\mathbbm{1}_{\left\lbrace X_{t+s-1} \leq X_i \right\rbrace} - r_j^*F_{r_j}(X_i)\right)\left(\sqrt{\frac{1-r_j^*}{r_j^*}} - \sqrt{\frac{1-b^*}{b^*}}\right)\\
\nonumber & \quad + \frac{1}{\sqrt{b^*(1-b^*)}}\left(\frac{1}{l}\sum_{t=b^*l+1}^{r_j^*l}\mathbbm{1}_{\left\lbrace X_{t+s-1}\leq X_i \right\rbrace} - (r_j^* - b^*)F_{r_j}(X_i)\right).
\end{align}
Adding now and subtracting $\frac{1}{\sqrt{r_j^*(1-r_j^*)}}\left(\frac{1}{l}\sum_{t=b^*l+1}^{r_j^*l}\mathbbm{1}_{\left\lbrace X_{t+s-1}\leq X_i \right\rbrace} - (r_j^* - b^*)F_{r_j}(X_i)\right)$, we have that
\begin{align}
\nonumber & M_{s,e}^{b^*}(X_i) - M_{s,e}^{r_j^*}(X_i) = \left(\frac{1}{l}\sum_{t=r_j^*l+1}^{l}\mathbbm{1}_{\left\lbrace X_{t+s-1} \leq X_i \right\rbrace} - (1-r_j^*)F_{r_j+1}(X_i)\right)\left(\sqrt{\frac{b^*}{1-b^*}} - \sqrt{\frac{r_j^*}{1-r_j^*}}\right)\\
\nonumber & \qquad + \left(\frac{1}{l}\sum_{t=1}^{r_j^*l}\mathbbm{1}_{\left\lbrace X_{t+s-1} \leq X_i \right\rbrace} - r_j^*F_{r_j}(X_i)\right)\left(\sqrt{\frac{1-r_j^*}{r_j^*}} - \sqrt{\frac{1-b^*}{b^*}}\right)\\
\nonumber & \qquad + \left(\frac{1}{\sqrt{b^*(1-b^*)}} - \frac{1}{\sqrt{r_j^*(1-r_j^*)}}\right)\left(\frac{1}{l}\sum_{t=b^*l+1}^{r_j^*l}\mathbbm{1}_{\left\lbrace X_{t+s-1}\leq X_i \right\rbrace} - (r_j^* - b^*)F_{r_j}(X_i)\right)\\
\nonumber & \qquad + \frac{1}{\sqrt{r_j^*(1-r_j^*)}}\left(\frac{1}{l}\sum_{t=b^*l+1}^{r_j^*l}\mathbbm{1}_{\left\lbrace X_{t+s-1}\leq X_i \right\rbrace} - (r_j^* - b^*)F_{r_j}(X_i)\right).
\end{align}
The functions $\sqrt{\frac{b^*}{1-b^*}}, \sqrt{\frac{1-b^*}{b^*}}$ and $\frac{1}{\sqrt{b^*(1-b^*)}}$ are Lipschitz continuous on the interval $[\alpha, 1-\alpha]$ for any arbitrary $\alpha \in (0,1/2)$. Therefore, there is a constant $C > 0$ such that for all $b^* \in J_{s,e} \cap [\alpha,1-\alpha]$, we have that
\begin{align}
\label{second_important_ineq}
\nonumber & \sup_{i = 1,\ldots,T}\left|M_{s,e}^{b^*}(X_i) - M_{s,e}^{r_j^*}(X_i) - \frac{1}{\sqrt{r_j^*(1-r_j^*)}}\left(\frac{1}{l}\sum_{t=b^*l+1}^{r_j^*l}\mathbbm{1}_{\left\lbrace X_{t+s-1}\leq X_i \right\rbrace} - (r_j^* - b^*)F_{r_j}(X_i)\right)\right|\\
& \leq C\left|b^* - r_j^*\right|\max_{h \in [0,1]}\sup_{i=1,\ldots,T}A,
\end{align} 
where from now on
\begin{align}
\nonumber & A:= A(h,i) = \left|\mathbbm{1}_{\left\lbrace h \leq r_j^*\right\rbrace}\frac{1}{l}\sum_{t=hl+1}^{r_j^*l}\mathbbm{1}_{\left\lbrace X_{t+s-1}\leq X_i \right\rbrace} - (r_j^* - h)F_{r_j}(X_i)\right.\\
\nonumber & \left.\qquad\qquad\qquad\qquad\qquad + \mathbbm{1}_{\left\lbrace h > r_j^*\right\rbrace}\frac{1}{l}\sum_{t=r_j^*l+1}^{h^*l}\mathbbm{1}_{\left\lbrace X_{t+s-1}\leq X_i \right\rbrace} - (h - r_j^*)F_{r_j+1}(X_i)\right|
\end{align}
We will now show that $\max_{h \in [0,1]}\sup_{i=1,\ldots,T}A = \mathcal{O}_{{\rm p}}\left((l)^{-1/2}\right)$. Due to the fact that we have a process with independent increments, then as shown in Lemma 2 of \cite{Dumbgen1991_supp} we know that for any $\eta > 0$,
\begin{align}
\nonumber \Prob\left(\max_{h \in [0,1]}\sup_{i=1,\ldots,T}A > \eta \sqrt{l}\sqrt{1-r_j^*}\right)\leq K_1\exp\left\lbrace-\frac{K_2\eta^2}{4}\right\rbrace,
\end{align}
where $K_1$ and $K_2$ are just positive constants. Applying the above maximal inequality to the result in \eqref{second_important_ineq}, and with $\boldsymbol{\mathbbm{1}}_{X} : = \left(\mathbbm{1}_{\left\lbrace X \leq X_1 \right\rbrace}, \mathbbm{1}_{\left\lbrace X \leq X_2 \right\rbrace}, \ldots, \mathbbm{1}_{\left\lbrace X \leq X_T \right\rbrace}\right)$ and $\boldsymbol{F_{r_j}} := \left(F_{r_j}(X_1), F_{r_j}(X_2), \ldots, F_{r_j}(X_T)\right)$, it follows that for every $\alpha \in (0,1/2)$,
\begin{align}
\nonumber & L\left(\left|\boldsymbol{M_{s,e}^{b^*}} - \boldsymbol{M_{s,e}^{r_j^*}} - \frac{1}{\sqrt{r_j^*(1-r_j^*)}}\left(\frac{1}{l}\sum_{t=b^*l+1}^{r_j^*l}\boldsymbol{\mathbbm{1}}_{X_{t+s-1}} - (r_j^* - b^*)\boldsymbol{F_{r_j}}\right)\right|\right)\\
\nonumber & \leq \sup_{i = 1,\ldots,T}\left|M_{s,e}^{b^*}(X_i) - M_{s,e}^{r_j^*}(X_i) - \frac{1}{\sqrt{r_j^*(1-r_j^*)}}\left(\frac{1}{l}\sum_{t=b^*l+1}^{r_j^*l}\mathbbm{1}_{\left\lbrace X_{t+s-1} \leq X_i \right\rbrace} - (r_j^* - b^*)F_{r_j}(X_i)\right)\right|\\
\nonumber & \leq C\left|b^* - r_j^*\right|\mathcal{O}{{\rm p}}\left((l)^{-1/2}\right),
\end{align}
where $\mathcal{O}_{{\rm p}}\left((l)^{-1/2}\right)$ denotes a random variable that does not depend on $b^* \in J_{s,e}$. From the result above, it is easy to see that \eqref{assertion_follows} would follow from the fact that for arbitrary constant $C''' > 0$
\begin{equation}
\label{assertion_follows2}
\limsup_{T \rightarrow \infty}\Prob\left(\sup_{i=1,\ldots,T}\left|lA(b^*,i)\right| > C'''l\left|b^* - r_j^*\right|/\gamma_{j,T}\;\;{\rm for\;some}\;b^* \in J_{s,e}\setminus N_T(d)\right) \rightarrow 0
\end{equation}
as $d \rightarrow \infty$. Let us now denote by $Y_1, Y_2, \ldots$ and $\tilde{Y}_1, \tilde{Y}_2, \ldots$ to be independent random variables from the distribution of our data in $(r_{j-1}, r_j]$ and $(r_{j}, r_{j+1}]$, respectively. We define
\begin{equation}
R_{m}(X_i) := \sum_{t=1}^{m}\left(\mathbbm{1}_{\left\lbrace Y_{t}\leq X_i \right\rbrace} - F_{r_j}(X_i)\right), \qquad \tilde{R}_{m}(X_i) := \sum_{t=1}^{m}\left(\mathbbm{1}_{\left\lbrace \tilde{Y}_{t}\leq X_i \right\rbrace} - F_{r_{j+1}}(X_i)\right).
\end{equation}
Then, because
\begin{align}
\nonumber lA(b^*,i) = & \left|\mathbbm{1}_{\left\lbrace h \leq r_j^*\right\rbrace}\sum_{t=hl+1}^{r_j^*l}\left(\mathbbm{1}_{\left\lbrace X_{t+s-1}\leq X_i \right\rbrace} - F_{r_j}(X_i)\right)\right.\\
\nonumber & \qquad \left. + \mathbbm{1}_{\left\lbrace h > r_j^*\right\rbrace}\sum_{t=r_j^*l+1}^{h^*l}\left(\mathbbm{1}_{\left\lbrace X_{t+s-1}\leq X_i \right\rbrace} - F_{r_j+1}(X_i)\right)\right|,
\end{align} 
continuing from \eqref{assertion_follows2} we have that
\begin{align}
\nonumber & \Prob\left(\sup_{i=1,\ldots,T}\left|lA(b^*,i)\right| > C'''l\left|b^* - r_j^*\right|/\gamma_{j,T}\;\;{\rm for\;some}\;b^* \in J_{s,e}\setminus N_T(d)\right)\\
\nonumber & \leq \Prob\left(\max_{m\geq d\gamma_{j,T}^2}\frac{1}{m}\sup_{i=1,\ldots,T}\left|R_m(X_i)\right| \geq  C'''/\gamma_{j,T}\right) + \Prob\left(\max_{m\geq d\gamma_{j,T}^2}\frac{1}{m}\sup_{i=1,\ldots,T}\left|\tilde{R}_m(X_i)\right| \geq  C'''/\gamma_{j,T}\right).
\end{align}
It now follows from the result in \cite{Dumbgen1991_supp}, page 1489, that
\begin{equation}
\label{expectation_result}
\Prob\left(\max_{m\geq d\gamma_{j,T}^2}\frac{1}{m}\sup_{i=1,\ldots,T}\left|R_m(X_i)\right| \geq  C'''/\gamma_{j,T}\right) \leq \frac{\gamma_{j,T}}{C'''}\E\left(\frac{1}{m_0}\sup_{i=1,\ldots,T}\left|R_{m_0}(X_i)\right|\right),
\end{equation}
where $m_0:=\min\left\lbrace m\in\mathbb{N}: m\geq d\gamma_{j,T}^2\right\rbrace$. The same result holds for $\tilde{R}_m(X_i)$ in the place of $R_m(X_i)$. Now, due to the Dvoretzky-Kiefer-Wolfowitz inequality we have that for any $\epsilon > 0$
\begin{align}
\nonumber \Prob\left(\frac{1}{\sqrt{m_0}}\sup_{i=1,\ldots,T}\left|R_{m_0}(X_i)\right| > \epsilon\right) & = \Prob\left(\sup_{i=1,\ldots,T}\left|\frac{1}{m_0}\sum_{t=1}^{m_0}\mathbbm{1}_{\left\lbrace Y_{t}\leq X_i \right\rbrace} - F_{r_j}(X_i)\right| > \frac{\epsilon}{\sqrt{m_0}}\right)\\
\nonumber & \leq \Prob\left(\sup_{x \in \mathbb{R}}\left|\frac{1}{m_0}\sum_{t=1}^{m_0}\mathbbm{1}_{\left\lbrace Y_{t}\leq x \right\rbrace} - P(X_{r_j}\leq x)\right| > \frac{\epsilon}{\sqrt{m_0}}\right)\\
\nonumber & \leq 2\exp\left\lbrace-2\epsilon^2\right\rbrace.
\end{align}
Using this result, then we get that for the right-hand side of \eqref{expectation_result} the following holds
\begin{align}
\nonumber \frac{\gamma_{j,T}}{C'''}\E\left(\frac{1}{m_0}\sup_{i=1,\ldots,T}\left|R_{m_0}(X_i)\right|\right) & = \frac{\gamma_{j,T}}{C'''\sqrt{m_0}}\E\left(\frac{1}{\sqrt{m_0}}\sup_{i=1,\ldots,T}\left|R_{m_0}(X_i)\right|\right)\\
\nonumber & \leq \frac{2\gamma_{j,T}}{C'''\sqrt{m_0}}\int\exp\left\lbrace-2x^2\right\rbrace{\mathrm{d}}x\\
\nonumber & \leq \frac{2}{\sqrt{d}C'''}\int\exp\left\lbrace-2x^2\right\rbrace{\mathrm{d}}x,
\end{align}
since $m_0 \geq d\gamma_{j,T}^2$. It is now straightforward that the right-hand side of the above inequality goes to 0 as $d \rightarrow \infty$ and the result in \eqref{main_Result_consistency} follows. Therefore,
\begin{equation}
\label{consistency_distance}
(\hat{r}_j - r_j)/\gamma_{j,T}^2 = \mathcal{O}_{{\rm p}}\left(1\right), \forall j\in \left\lbrace 1,\ldots,N\right\rbrace.
\end{equation}
Therefore, for $\lambda_T \leq (\delta_T/3)$ we have proven that working under the set $A_T^*$ (which implies the set $A_T$), and due to (A1), there will be an interval of the form $[1,c_{k^*}^r]$, with $L\left(\left|\boldsymbol{\tilde{B}_{1,c_{k^*}^r}^{b_1}}\right|\right)>\zeta_T$ with probability tending to 1 as $T \to \infty$, where $b_1$ is an estimation of $r_1$, that also satisfies $|b_1 - r_1|/\gamma_{1,T}^2 = \mathcal{O}_{{\rm p}}\left(1\right)$.
\vspace{0.1in}
\\
{\textbf{Step 1.3:}}  After detecting the first change-point, NPID follows the exact same process as in Steps 1.1 and 1.2, but only in the set $[c_{k^*}^r,T]$, which contains $r_2, r_3, \ldots, r_N$. This means that we bypass, without checking for possible change-points, the interval $[b_1+1, c_{k^*}^r)$. However, we need to prove that:
\begin{itemize}
\item[(S.1)] There is no change-point in $[b_1 + 1,c_{k^*}^r)$, apart from maybe the already detected $r_1$;
\item[(S.2)] $c_{k^*}^r$ is at a location which allows for detection of $r_2$.
\end{itemize}
{\textbf{For (S.1):}} We will split the explanation into two cases with respect to the location of $b_1$.\\
{\textbf{Case 1:}} $b_1 < r_1 < c_{k^*}^r$. Using \eqref{main_Result_consistency}, \eqref{LambdaTd} and \eqref{Ntd} and imposing the condition
\begin{equation}
\label{condition_delta1}
\Prob\left(\min\{\delta_j,\delta_{j+1}\} > 3d\gamma_{j,T}^2\right) \xrightarrow[T\rightarrow \infty]{}  1, \forall j \in \{1,\ldots,N\}
\end{equation}
for a $d$ which goes to infinity as $T \rightarrow \infty$, then since $c_{\tilde{k}}^r \in I_1^R$, we have, with probability going to 1 as $T \rightarrow \infty$, that
\begin{equation}
\nonumber c_{k^*}^r - b_1 \leq c_{\tilde{k}}^r - b_1 = c_{\tilde{k}}^r - r_1 + r_1 - b_1 < 2\frac{\delta_2}{3} + d\gamma_{1,T}^2 < \delta_2.
\end{equation}   
Since $r_2 - r_1 = \delta_2$ and $r_1$ is already in $[b_1+1,c_{k^*}^r)$, then there is no other change-point in $[b_1+1,c_{k^*}^r)$ apart from $r_1$. We need to highlight that in the case of $d$ being of order at most $\mathcal{O}(\log T)$, then the result in \eqref{condition_delta1} is not in fact an extra assumption and is satisfied through (A1).\\
{\textbf{Case 2:}} $r_1 \leq b_1 < c_{k^*}^r$. Since $c_{\tilde{k}}^r \in I_1^R$, then $c_{k^*}^r - r_1 \leq c_{\tilde{k}}^r - r_1 < 2\delta_2/3$, which means that apart from $r_1$ there is no other change-point in $[r_1,c_{k^*}^r)$. With $r_1 \leq b_1$, then $[b_1+1,c_{k^*}^r)$ does not have any change-point.

Cases 1 and 2 above show that no matter the location of $b_1$, there is no change-point in $[b_1+1,c_{k^*}^r)$ other than possibly the previously detected $r_1$. Similarly to the approach in Steps 1.1 and 1.2, our method applied now in $[c_{k^*}^r,T]$,  will first isolate $r_2$ or $r_N$ depending on whether $r_2 - c_{k^*}^r$ is smaller or larger than $T-r_N$. If $T-r_N < r_2 - c_{k^*}^r$ then $r_N$ will get isolated first in a left-expanding interval and the procedure to show its detection is exactly the same as in Step 1.1 where we explained the detection of $r_1$. Therefore, w.l.o.g. and also for the sake of showing (S.2) let us assume that $r_2 - c_{k^*}^r \leq T-r_N$.
\vspace{0.1in}
\\
{\textbf{For (S.2):}} With ${\rm{I}}_{s,e}^r$ as in (2.1) of \cite{Anastasiou_Fryzlewicz_supp}, there exists $c_{k_2}^r \in {\rm{I}}_{c_{k^*}^r,T}^r$ such that $c_{k_2}^r \in I_2^R$, with $I_{j}^R$ defined in \eqref{isolating_intervals}. We will show that $r_2$ gets detected in $[c_{k^*}^r,c_{k^*_2}^r]$, for $k_2^* \leq k_2$ and its detection, denoted by $b_2$, satisfies $\left|b_2 - r_2\right|/\gamma_{2,T}^2 = \mathcal{O}_{{\rm p}}\left(1\right)$, with $\gamma_{j,T}$ as in Assumption (A1). Following similar steps as in \eqref{midstep1}, we have that for $\tilde{b}_2 = {\rm argmax}_{c_{k^*}^r\leq t < c_{k_2}^r}L\left(\left|\boldsymbol{\tilde{B}_{c_{k^*}^r,c_{k_2}^r}^{t}}\right|\right)$,
\begin{align}
\label{midstepforr2}
\nonumber & L\left(\left|\boldsymbol{\tilde{B}_{c_{k^*}^r,c_{k_2}^r}^{\tilde{b}_2}}\right|\right|\\
& \geq L\left(\left|\boldsymbol{\tilde{F}_{c_{k^*}^r,c_{k_2}^r}^{r_2}}\right|\right) - 4\sqrt{\log T} \geq \sqrt{\frac{\min\left\lbrace c_{k_2}^r-r_2,r_2-c_{k^*}^r+1 \right\rbrace}{2}}L(|\boldsymbol{\Delta_2}|) - 4\sqrt{\log T}.
\end{align}
By construction,
\begin{align}
\nonumber & c_{k_2}^r - r_2 \geq \frac{\delta_3}{3}\\
\nonumber & r_2 - c_{k^*}^r + 1 \geq r_2 - c_{\tilde{k}}^r + 1 = r_2 - r_1 - (c_{\tilde{k}}^r - r_1) + 1 = \delta_2 - (c_{\tilde{k}}^r - r_1) + 1 > \delta_2 - 2\frac{\delta_2}{3} + 1\\
\nonumber & \qquad\qquad\;\quad > \frac{\delta_2}{3},
\end{align}
which means that $\min\left\lbrace c_{k_2}^r-r_2,r_2-c_{k^*}^r+1 \right\rbrace \geq \frac{1}{3}\min\left\lbrace \delta_2, \delta_3\right\rbrace$ and therefore continuing from \eqref{midstepforr2}, 
\begin{align}
\label{thresholdpassing2}
\nonumber L\left(\left|\boldsymbol{\tilde{B}_{c_{k^*}^r,c_{k_2}^r}^{\tilde{b}_2}}\right|\right) & \geq \sqrt{\frac{\min\left\lbrace\delta_2,\delta_3\right\rbrace}{6}}L(|\boldsymbol{\Delta_2}|) - 4\sqrt{\log T}\\
\nonumber & \geq \sqrt{\frac{\min\{\delta_2,\delta_3\}}{6}}\frac{\tilde{C}_2}{\gamma_{2,T}} - 4\sqrt{\log T}\\
\nonumber & \geq \left(\frac{1}{\sqrt{6}} - \frac{4\sqrt{\log T}}{\underline{m}_T}\right)\underline{m}_T\\
& \geq \left(\frac{1}{\sqrt{6}}-\frac{4}{\underline{C}}\right)\underline{m}_T = C_2\underline{m}_T >\zeta_T.
\end{align}
Therefore, based on (A1), for a $c_{\tilde{k}_2}^r \in {\rm{I}}_{c_{k^*}^r,T}^r$ we have shown that there exists an interval of the form $[c_{k^*}^r,c_{\tilde{k}_2}^r]$, with $\max_{c_{k^*}^r\leq b <c_{\tilde{k}_2}^r}L\left(\left|\boldsymbol{\tilde{B}_{c_{k^*}^r,c_{\tilde{k}_2}^r}^{b}}\right|\right) > \zeta_T$ with probability tending to 1. Let us denote by $c_{k^*_2}^r \in {\rm{I}}_{c_{k^*}^r,T}^r$ the first right-expanding point where this occurs  and let $$b_2 = {\rm argmax}_{c_{k^*}^r\leq t < c_{k^*_2}^r}L\left(\left|\boldsymbol{\tilde{B}_{c_{k^*}^r,c_{k^*_2}^r}^{t}}\right|\right)$$ with $L\left(\left|\boldsymbol{\tilde{B}_{c_{k^*}^r,c_{k^*_2}^r}^{b_2}}\right|\right) > \zeta_T$ with probability tending to 1 as $T \to \infty$.

It is straightforward to show that $\left|b_2 - r_2\right|/\gamma_{2,T}^2 = \mathcal{O}_{{\rm p}}\left(1\right)$ following exactly the same process as in Step 1.2 and we will not repeat this here. Having detected $r_2$, then our algorithm will proceed in the interval $[s,e]=[c_{k^*_2}^r, T]$ and all the change-points will get detected one by one since Step 1.3 will be applicable as long as there are previously undetected change-points in $[s,e]$. Denoting by $\hat{r}_j$ the estimation of $r_j$ as we did in the statement of the theorem, then we conclude that all change-points will get detected one by one and
\begin{equation}
\nonumber \left|\hat{r}_j - r_j\right|/\gamma_{j,T}^2 = \mathcal{O}_{{\rm p}}\left(1\right).
\end{equation}
due to the result in \eqref{main_Result_consistency}.\\
{\textbf{Step 2:}} The arguments given in Steps 1.1-1.3 hold in the set $A_T^*$ (which also implies $A_T$) defined in \eqref{A_T}. At the beginning of the algorithm, $s=1, e=T$ and for $N\geq 1$, there exist $k_1\in \left\lbrace 1,\ldots, K \right\rbrace$ such that $s_{k_1} = s, e_{k_1} \in I_1^R$ and $k_2\in \left\lbrace 1,\ldots, K \right\rbrace$ such that $s_{k_2} \in I_N^L, e_{k_2} = e$. As in our previous steps, w.l.o.g. assume that $r_1 \leq T - r_N$, meaning that $r_1$ gets isolated and detected first in an interval $[s,c_{k^*}^r]$, where $c_{k^*}^r \in {\rm{I}}_{1,T}^r$ and it is less than or equal to $e_{k_1}$. Then, $\hat{r}_1 = {\rm argmax}_{s \leq t < c_{k^*}^r}L\left(\left|\boldsymbol{\tilde{B}_{s,c_{k^*}^r}^{t}}\right|\right)$ and $\hat{r}_1$ is the estimated location for $r_1$ and $\left|r_1-\hat{r}_1\right|/\gamma_{1,T}^2 = \mathcal{O}_{{\rm p}}\left(1\right)$. After this, the algorithm continues in $[c_{k^*}^r,T]$ and keeps detecting all the change-points as explained in Step 1. It is important to note that there will not be any double detection issues because naturally, at each step of the algorithm, the new interval $[s,e]$ does not include any previously detected change-points.

Once all the change-points have been detected one by one, then $[s,e]$ will have no other change-points in it. Our method will keep interchangeably checking for possible change-points in intervals of the form $\left[s,c_{\tilde{k}_1}^r\right]$ and $\left[c_{\tilde{k}_2}^l,e\right]$ for $c_{\tilde{k}_1}^r \in {\rm{I}}_{s,e}^r$ and $c_{\tilde{k}_2}^l \in {\rm{I}}_{s,e}^l$. Allow us to denote by $[s^*,e^*]$ any of these intervals. Our algorithm will not detect anything in $[s^*,e^*]$ since $\forall i \in \left\lbrace 1,\ldots, T \right\rbrace$ and $\forall b \in [s^*,e^*)$,
\begin{equation}
\nonumber L\left(\left|\boldsymbol{\tilde{B}_{s^*,e^*}^{b}}\right|\right) \leq L\left(\left|\boldsymbol{\tilde{F}_{s^*,e^*}^{b}}\right|\right) + 4\sqrt{\log T} = 4\sqrt{\log T} < C_1\sqrt{\log T}\leq \zeta_T.
\end{equation}
After not detecting anything in all intervals of the above form, then the algorithm concludes that there are not any change-points in $[s,e]$ and stops.$\hfill\square$
\bibliographystyle{abbrv}

\endgroup
\end{document}